\documentclass[reqno,11pt,letterpaper]{amsart}
 
\usepackage{amsmath}
\usepackage{amsthm}  
\usepackage{verbatim}
\usepackage{enumerate} 
\usepackage{mathtools}  
\usepackage{amssymb}
\usepackage{mathrsfs}
\usepackage[top=1.5in, bottom=1.5in, left=0.7in, right=0.7in]{geometry}  
\usepackage[colorlinks]{hyperref}
\hypersetup{
	colorlinks,
	citecolor=blue,
	linkcolor=red
}   
\numberwithin{equation}{section}
\usepackage{xypic}
\usepackage{bm}
\usepackage{tikz}
\usetikzlibrary{decorations.pathreplacing}
\usepackage{xcolor}
\usepackage{float}
\usepackage{cleveref}
\usepackage{comment}

\newtheorem{theorem}{\textbf{Theorem}}[section]
\newtheorem{theorem*}{\textbf{Theorem}}
\newtheorem{thmx}{Theorem}

\newtheorem{propx}[thmx]{\textbf{Proposition}}
\newtheorem{conjecturex}[thmx]{\textbf{Conjecture}}
\newtheorem{definition}[theorem]{\textbf{Definition}}
\newtheorem{proposition}[theorem]{\textbf{Proposition}}
\newtheorem{lemma}[theorem]{\textbf{Lemma}}

\newtheorem{claim*}[theorem*]{\textbf{Claim}}

\newtheorem{corollary}[theorem]{\textbf{Corollary}}
\newtheorem{remark}[theorem]{\textbf{Remark}}

\newtheorem{example}[theorem]{\textbf{Example}}

\newtheorem{definition/proposition}[theorem]{\textbf{Definition/Proposition}}

\def\N{{\mathbb N}}
\def\R{\mathbb{R}}
\def\Z{{\mathbb Z}}

\def\C{{\mathbb C}}
\def\D{{\mathbb D}}

\def\Q{{\mathbb Q}}
\def\H{{\mathbb H}}

\newcommand{\CP}{\mathbb{C}\mathbb{P}}
\newcommand{\RP}{\mathbb{R}\mathbb{P}}
\newcommand{\HP}{\mathbb{H}\mathbb{P}}

\def\k{\kappa}

\def\i{\iota}

\def\cA{{\mathcal A}}

\def\cM{{\mathcal M}}

\def\cO{{\mathcal O}}

\def\rD{{\rm D}}

\def\rd{{\rm d}}

\newcommand{\CR}{\mathrm CR}

\def\la{\langle\,}
\def\ra{\,\rangle}

\def\std{\rm std}
\def\i{\mathbf{i}}
\def\j{\mathbf{j}}
\def\k{\mathbf{k}}
\def\CZ{\rm CZ}

\def\orb{\rm orb}

\DeclareMathOperator{\Ima}{im}
\DeclareMathOperator{\ind}{ind}
\DeclareMathOperator{\Hom}{Hom}
\DeclareMathOperator{\Id}{id}

\DeclareMathOperator{\coker}{coker}

\DeclareMathOperator{\rank}{rank}

\DeclareMathOperator{\End}{End}

\DeclareMathOperator{\virdim}{virdim}

\DeclareMathOperator{\Conj}{Conj}
\DeclareMathOperator{\diag}{diag}
\DeclareMathOperator{\age}{age}
\DeclareMathOperator{\Tor}{Tor}

\DeclareMathOperator{\lSFT}{lSFT}
\DeclareMathOperator{\ord}{ord}
\DeclareMathOperator{\md}{md}

\newcommand{\Addresses}{{
		\bigskip
		\footnotesize

	     Zhengyi Zhou, \par\nopagebreak
	    \textsc{Morningside Center of Mathematics and Institute of Mathematics, AMSS, CAS, China}\par\nopagebreak
		\textit{E-mail address}: \href{mailto:zhyzhou@amss.ac.cn}{zhyzhou@amss.ac.cn}

}}

\title{On fillings of  contact links of quotient singularities}
\author{Zhengyi Zhou}

\begin{document}
	\maketitle
\begin{abstract}
We study several aspects of fillings for links of general isolated quotient singularities using Floer theory, including co-fillings, Weinstein fillings, strong fillings, exact fillings and exact orbifold fillings, focusing on the non-existence of exact fillings of contact links of isolated terminal quotient singularities. We provide an extensive list of isolated terminal quotient singularities whose contact links are not exactly fillable, including $\C^n/(\Z/2)$ for $n\ge 3$, which settles a conjecture of Eliashberg, quotient singularities from general cyclic group actions and finite subgroups of $SU(2)$, and all terminal quotient singularities in complex dimension $3$. We also obtain uniqueness of the \emph{orbifold} diffeomorphism type of \emph{exact orbifold fillings} of contact links of some isolated terminal quotient singularities.
\end{abstract}
\tableofcontents
\section{Introduction}
Contact structures arise naturally from symplectic manifolds with a convex boundary. The study of contact manifolds as convex boundaries of symplectic manifolds and their symplectic fillings was pioneered by Gromov \cite{MR809718} and Eliashberg \cite{MR1171908}. In particular, they showed that any contact 3-fold with a (weak) symplectic filling is necessarily tight. This result was generalized to all dimensions by Massot, Niederkr\"{u}ger and Wendl,\cite{MR3044125}, see also \cite{schmaltz2020non}. As a consequence, the different types of fillability give a coarse understanding of the landscape of contact structures beyond the fundamental dichotomy between overtwistedness and tightness \cite{MR3455235,MR1022310}. Namely, we have the following inclusions among classes of contact structures:
$$\{\text{Weinstein fillable}\}\subset \{\text{exactly fillable}\} \subset \{\text{strongly fillable}\} \subset \{\text{weakly fillable}\} \subset \{\text{tight}\}.$$
All of the inclusions are proper: in dimension three, they were shown by Bowden \cite{MR2979998}, Ghiggini \cite{MR2175155}, Eliashberg \cite{MR1383953}, Etnyre and Honda \cite{MR1908061}, and in higher dimensions by Bowden, Crowley, and Stipsicz \cite{stein}, the author \cite{RP}, Bowden, Gironella, and Moreno \cite{MR4493326}, Massot, Niederkr{\"u}ger, and Wendl \cite{MR3044125}. 

The guiding question in understanding the second inclusion is the following conjecture by Eliashberg:
\begin{conjecturex}[{\cite[\S 1.9]{question}}]\label{conj:E}
$(\RP^{2n-1},\xi_{\std})$ is not exact fillable if $n\ge 3$.
\end{conjecturex}
Fortunately, the real projective space is natural and simple enough so that one can understand both pseudo-holomorphic curves as well as the topology sufficiently well.  We showed in \cite{RP} that Conjecture \ref{conj:E} holds if $n\ne 2^k$. The odd $n$ cases were also proved by Ghiggini and Niederkr\"{u}ger \cite{MR4406869} using different methods. In this paper, we fill the gap, hence settle the conjecture.
\begin{thmx}\label{thm:RP}
For $n\ge 3$, $(\RP^{2n-1},\xi_{\std})$ is not exactly fillable. 
\end{thmx}
When $n=2^k$, the total Chern class of $\xi_{\std}$ becomes trivial, which sabotages the topological obstructions in \cite{RP}. Hence to prove \Cref{thm:RP}, we need to enhance both the pseudo-holomorphic curve argument and the topological argument. On the pseudo-holomorphic curve side, we use symplectic cohomology for exact orbifolds developed by Gironella and the author \cite{gironella2021exact} and study the secondary coproduct on positive symplectic cohomology, which was proposed by Seidel \cite{biased}, defined by Ekholm and Oancea \cite{dga}, and recently studied extensively by Cieliebak, Hingston and Oancea \cite{cieliebak2020loop,cieliebak2020poincar,cieliebak2020multiplicative}. Then by a result of Eliashberg, Ganatra, and Lazarev \cite{EGL}, we can get information on the intersection form of the filling. On the topological side, the contradiction is a non-integer Chern number derived from the Hirzebruch signature theorem. The argument essentially boils down to studying the $2$ factors of the numerators of Bernoulli numbers, whose proof is motivated by a proof of the Clausen-Staudt Theorem dating back to 1883 \cite{MR1579945}. When $n=4$, the arguments above do not yield a contradiction. But we can build a spin manifold from a hypothetical exact filling. Finally, we can arrive at a contradiction by showing its $\hat{A}$-genus is not integer, violating the celebrated Atiyah–Singer index theorem for the Dirac operator. 

\subsection{Fillings of Contact links of isolated quotient singularities}
Real projective spaces with the standard contact structure fall in a much larger class of contact manifolds: contact links of isolated singularities. Links of isolated singularities are natural and interesting examples of contact manifolds, which are always strongly fillable by Hironaka’s resolution of singularities. In the case of smoothable singularities, the smoothing provides natural exact/Weinstein fillings to the contact link. In complex dimension $2$, contact properties of links, especially filling properties, in particular Weinstein fillings,  were extensively studied, see \cite{MR3707744} for an introduction and references therein for details. In such a low dimension, tools from low dimensional topology and intersection theory of pseudo-holomorphic curves allow us to obtain classification results. For example, the diffeomorphism type of symplectic fillings of lens spaces is classified by Lisca \cite{MR2346471}. Moreover, such classification can be strengthened to a classification up to symplectic deformation by a powerful result of Wendl \cite{MR2605865}. Moreover, there are surprising connections between the contact properties of the link and algebro-geometric properties of the singularity, e.g.\ the beautiful theorem of McLean \cite{Reeb} relating the terminality of the singularity with the SFT degrees.

In this paper, we study, by means of Floer theories, symplectic fillings of contact links of isolated quotient singularities in any dimension, where $(\RP^{2n-1},\xi_{\std})$ should be understood as the contact link of the simplest quotient singularity $\C^n/(\Z/2)$. More generally, let $G$ be a finite subgroup of $U(n)$, such that every nontrivial element in $G$ does not have one as an eigenvalue. Since $G$ preserves the standard contact structure on the unit sphere $S^{2n-1}\subset \C^n$, we get a natural contact manifold $(S^{2n-1}/G,\xi_{\std})$, which is the contact link of the isolated quotient singularity $\C^n/G$. Our discussion will cover co-fillings (a.k.a.\ semi-fillings), Weinstein fillings, and strong fillings, but with an emphasis on exact fillings and exact orbifold fillings, recalled in the following.

\begin{definition}
Let $(Y,\xi)$ be a co-oriented contact manifold, then a symplectic manifold $(W,\omega)$ is called
\begin{enumerate}
    \item a strong filling of $Y$, if $\partial W =Y$ and there exits a outward-pointing Liouville vector field $X$ near $\partial W$ (i.e.\ $L_X \omega=\omega$), such that $\iota_X \omega$ restricts to a contact form for $\xi$;
    \item a co-filling (semi-filling) of $Y$, if $W$ is a \textbf{connected} strong filling of $Y\sqcup Y'$, where $\bm{Y'\ne \emptyset}$ is another contact manifold; 
    \item an exact filling of $Y$, if $W$ is a strong filling and the Liouville vector field $X$ is globally defined;
    \item an exact orbifold filling of $Y$, if $W$ is an orbifold and exact;
    \item a Weinstein filling of $Y$, if $W$ is exact and the Liouville vector field $X$ is gradient-like for a Morse function $\phi$ with $\partial W$ a regular level set of $\phi$.
\end{enumerate}
An almost complex manifold $(W,J)$ is an almost Weinstein filling of $Y$, if $\partial W =Y$,  $JTY\cap TY =\xi$ such that $J|_{\xi}$ is homotopic to the almost complex structure compatible with the contact structure and $J$ sends the outward normal vector to the co-orientation along $Y$, and there exists a Morse function $\phi$  with $\partial W$ a regular level set of $\phi$, whose critical points have indices at most $\dim W/2$.
\end{definition}
\begin{remark}
The general definition of a weak filling was given in \cite[Definition 4]{MR3044125}. Since $H^2(S^{2n-1}/G;\R)=0$, there is no difference between weak fillings and strong fillings for contact links in this paper up to deformation by \cite[Lemma 2.10]{MR3044125}.
\end{remark}

\subsection{Weinstein fillings, strong fillings, and co-fillings.}
Unlike the situation in complex dimension $2$, it is a classical fact that isolated quotient singularities of complex dimension at least $3$ are not smoothable, which was proved using algebraic methods by showing those singularities are rigid \cite{rigid}. In particular, the contact link of an isolated quotient singularity of complex dimension at least $3$ can not be filled by a Milnor fiber. On the symplectic side, it is an easy corollary of the Eliashberg-McDuff-Floer theorem \cite{mcduff1990structure} that the contact link of an isolated quotient singularity of complex dimension at least $3$ has no Weinstein filling, see \Cref{prop:EFM}. We first state the following result on fundamental groups of symplectic fillings, which also yields the non-existence of Weinstein fillings in complex dimensions at least $3$.
\begin{thmx}\label{thm:A}
	Let $\C^n/G$ be an isolated singularity and $n\ge 2$. 
	\begin{enumerate}
	\item\label{1} If $W$ is an \emph{exact filling} of the contact link $(S^{2n-1}/G,\xi_{\std})$, then $G=\pi_1(S^{2n-1}/G)\to \pi_1(W)$ is not injective.
	\item If $\C^n/G$ is a terminal quotient singularity (\Cref{def:terminal}) and $W$ is a \emph{strong filling} of the contact link, then $\pi_1(S^{2n-1}/G)\to \pi_1(W)$ is surjective. If moreover, $W$ is an \emph{exact filling}, then $W$ is simply connected.
	\item If $G=\Z/m$, where $m$ is not divisible by a non-trivial square, and $\C^n/G$ is terminal, then any \emph{strong filling} of the contact link is simply connected.
	\end{enumerate}
\end{thmx}
\begin{remark}
\eqref{1} of \Cref{thm:A} also implies that the contact link of an isolated quotient singularity of complex dimension at least $3$ has no Weinstein filling. In fact, by the topological argument in \cite[\S 6.2]{order}, those contact links have no almost Weinstein filling whenever $n\ge 4$, see \Cref{prop:almost_Weinstein}.
\end{remark}

Terminal singularities are important players in the minimal model program in higher dimensions introduced by Reid \cite{canonical}. A deep theorem due to McLean \cite{Reeb} shows that the terminal property is in fact a contact property of the link. In the context of isolated quotient singularities, such relation is contained in the terminal criterion due to Reid, Shepherd-Barron \cite{canonical}  and Tai \cite{Kodaira}. By arguments in the same spirit of the proof of Theorem \ref{thm:A}, we can obtain the following result on co-fillings. 
\begin{propx}\label{prop:B}
Let $\C^n/G$ be an isolated terminal singularity, then $(S^{2n-1}/G,\xi_{std})$ is not co-fillable.
\end{propx}
Gromov \cite{MR809718}, Floer, Eliashberg and McDuff \cite{mcduff1990structure} showed that $(S^{2n-1},\xi_{\std})$ is not co-fillable when $n\ge 2$. Note that $(S^{2n-1},\xi_{\std})$ can be viewed as the link of a smooth point, which is terminal if and only if $n\ge 2$. Hence Proposition \ref{prop:B} can be viewed as a generalization of the above result. Obstructions to co-fillings were used to obstruct strong fillings, e.g.\ \cite{MR1383953}

\subsection{Non-existence of exact fillings}\label{ss:no}
The first example of a strongly but not exactly fillable contact manifold was found in dimension $3$ by Ghiggini \cite{MR2175155}, and many more examples in dimension $3$ were found by Min \cite{min2022strongly}. In higher dimensions, such examples were predicted by Conjecture \ref{conj:E}. It is worth noting that the threshold $n\ge 3$ in Conjecture \ref{conj:E} is precisely the threshold for the singularity $\C^n/(\Z/2)$ to be terminal. Therefore  \Cref{thm:RP} leads to the following conjecture in \cite{RP}.
\begin{conjecturex}\label{conj:no}
Let $\C^n/G$ be an isolated terminal singularity, then the contact link $(S^{2n-1}/G,\xi_{\std})$ is not exactly fillable. 
\end{conjecturex}
In view of the conjecture above and Hironaka’s resolution of singularities, isolated terminal quotient singularities shall provide tons of natural contact manifolds with strong fillings that are not exactly fillable. In this paper, we will prove this conjecture for a list of cases. 

For general group actions, we have the following result with a simple assumption. Let $\Conj(G)$ be the set of conjugacy classes of $G$.
\begin{thmx}\label{thm:J}
Assume $\C^n/G$ is an isolated terminal singularity with $n$ odd and $|\Conj(G)|$ even, then the contact link $(S^{2n-1}/G,\xi_{\std})$ has no exact filling.
\end{thmx}

In the smallest non-trivial dimension, namely complex dimension $3$, terminal quotient singularities are completely classified in \cite{MR722406}, where the group can be any finite cyclic group. Conjecture \ref{conj:no} indeed holds in those cases.
\begin{thmx}\label{thm:3}
The contact link of any isolated terminal quotient singularity in complex dimension $3$ has no exact filling.
\end{thmx}

We also consider the case of $\Z/3,\Z/4$ actions, more general cyclic group actions as well as finite subgroups of $SU(2)$ acting on $\C^{2n}$ diagonally. The proof of the above results is based on the study of symplectic cohomology of a hypothetical exact filling. The starting point is that the rational first Chern class $c_1^{\Q}(S^{2n-1}/G,\xi_{\std})$ vanishes, and all Reeb orbits are of torsion homotopy classes, therefore we have a well-defined rational Conley-Zehnder index assigned to each Reeb orbit. It is important to note that, as a hypothetical filling may not have vanishing first Chern class, those ``boundary"  Conley-Zehnder indices may not agree with actual  Conley-Zehnder indices used in the construction of symplectic cohomology. However, when the singularity is terminal, those boundary Conley-Zehnder indices induce a filtration on the cochain complex of the positive symplectic cohomology, which is compatible with the differential. Moreover, certain pages of the spectral sequence are independent of the filling/augmentation. Therefore, we can compare it with the computation we did in \cite[Theorem B]{gironella2021exact} for the standard exact \emph{orbifold} filling $\C^n/G$ and conclude that the spectral sequence degenerates at some point due to parity reasons. Moreover, we can show that the total dimension $\dim H^*(W;\Q)$ is $|\Conj(G)|$, the size of the set of conjugacy classes, for a hypothetical exact filling $W$ whenever $\C^n/G$ is terminal. Many topological properties (e.g.\ intersection form, fundamental groups, cup product length) of the hypothetical filling $W$ can be derived this way through the Floer theory,  which in many cases lead to contradictions to the existence of $W$.

\subsection{Uniqueness of the diffeomorphism type of exact orbifold fillings}
The celebrated Eliashberg-Floer-McDuff theorem \cite{mcduff1990structure} states that any exact filling of $(S^{2n-1},\xi_{\std})$ is diffeomorphic to a ball whenever $n\ge 2$. This theorem has been generalized in many directions \cite{MR4031531,MR2874896,ADC,product}. From the singularity perspective, $(S^{2n-1},\xi_{\std})$ is the link of a smooth point and a smooth point viewed as a singularity is terminal if and only if $n\ge 2$. Hence it naturally leads to the following conjecture, which in particular will imply Conjecture \ref{conj:no}.
\begin{conjecturex}\label{conj:uniqness}
Let $\C^n/G$ be an isolated terminal singularity, then exact \emph{orbifold} fillings of the contact link $(S^{2n-1}/G,\xi_{\std})$ have a unique diffeomorphism type\footnote{More aggressively, we can conjecture that exact orbifold fillings are unique.}. 
\end{conjecturex}
\begin{remark}
By \cite[Theorem C]{gironella2021exact}, any exact orbifold filling of $(S^{2n-1},\xi_{\std})$ for $n\ge 2$ must be a manifold. Therefore Conjecture \ref{conj:uniqness} holds for a smooth point by the  Eliashberg-Floer-McDuff theorem. 
\end{remark}

In this paper, we prove this conjecture in some simple cases. 
\begin{thmx}\label{thm:unique}
Let $n\ge 3$ be an odd number. Then any exact orbifold filling of $(\RP^{2n-1},\xi_{\std})$ is diffeomorphic to $\D^{2n}/(\Z/2)$ as an orbifold with boundary.
\end{thmx}

\subsection*{Organization of the paper} We discuss basic topological, contact, and algebro-geometric properties of quotient singularities in \S \ref{s2}. We study symplectic cohomology of hypothetical exact fillings in \S \ref{s3} and secondary coproduct in \S \ref{s4}. \S \ref{s5} focuses on fundamental groups of strong fillings of links. We prove \Cref{thm:RP} and \Cref{thm:unique} in \S \ref{s6} and other non-existence of exact fillings in \S \ref{s7}.  Appendix \ref{app} is devoted to proving some properties of the Bernoulli numbers. 

\subsection*{Acknowledgments}
The author would like to thank Ruizhi Huang for the helpful discussion on the modified signature and the anonymous referee for many helpful suggestions that improved this paper. The author is supported by the National Key R\&D Program of China under Grant No. 2023YFA1010500, the National Natural Science Foundation of China under Grant No. 12288201 and 12231010.
\section{Basics of the contact link $(S^{2n-1}/G,\xi_{\std})$}\label{s2}
In this section, we recall and prove some basics of the contact topology of $(S^{2n-1}/G,\xi_{\std})$. Unless specified otherwise, $G\subset U(n)$ acts on $\C^n$ with an isolated singularity throughout this paper. We write $\i:=\sqrt{-1}$.
\subsection{Topology of $S^{2n-1}/G$}\label{ss:top}
Since $G$ acts on $\C^n$ with an isolated singularity, the same is true for the diagonal action of $G$ on $(\C^n)^k$. In particular, this gives rise to a model of $K(G,1)=BG$ as $S^{\infty}/G:=\varinjlim \partial (\D^{2nk})/G$, where $\D^{2nk}$ is the unit ball in $\C^{nk}.$ As a consequence, we have the following.

\begin{proposition}\label{prop:cohomology}
The inclusion $S^{2n-1}/G\hookrightarrow BG=\varinjlim \partial (\D^{2nk})/G$ induces an isomorphism on cohomology over any coefficient for degrees smaller than $2n-1$. 
\end{proposition}
\begin{proof}
On the unit sphere $\partial (\D^{2nk})$, we have a Morse-Bott function
\begin{equation}\label{eqn:gradient}
    f=\sum_{s=0}^{k-1} s\sum_{j=1}^n |z_{sn+j}|^2
\end{equation}
which is $G$ invariant. The critical submanifolds of the induced function on $S^{2nk-1}/G$ are $k$ copies of $S^{2n-1}/G$ with Morse-Bott indices $0,2n,\ldots,2(k-1)n$. By the standard Morse-Bott theory, for any $k\ge 1$, we have that  $S^{2n-1}/G\hookrightarrow \partial (\D^{2nk})/G$ induces an isomorphism on cohomology over any coefficient for degrees smaller than $2n-1$. Hence the claim follows.
\end{proof}

\begin{remark}\label{rmk:connecting}
A few remarks on the cohomology of $BG$ are in order.
\begin{enumerate}
    \item Let $\bm{z}_s$ denote $(z_{sn+1},\ldots,z_{(s+1)n})$ in $\C^{nk}$ for $0\le s \le k-1$. Then $$\phi_t:(\bm{z}_0,\ldots,\bm{z}_{k-1})\mapsto \sum_{s=0}^{k-1} e^{st}\bm{z}_s \left/ | \sum_{s=0}^{k-1} e^{st}\bm{z}_s|\right.$$
    defines the flow of a $G$-equivariant gradient like vector field for \eqref{eqn:gradient}. Then the moduli space of Morse trajectories on $\partial(\D^{2nk})/G$ connecting the critical submanifolds of Morse-Bott indices $2jn$ and $2(j+1)n$ is diffeomorphic to $(S^{2n-1}\times S^{2n-1})/G$, with the evaluation maps to the two critical submanifolds given by the two projections. To see this, note that the moduli space of Morse trajectories on $\partial(\D^{2nk})$ connecting the critical submanifolds of Morse-Bott indices $2jn$ and $2(j+1)n$ is diffeomorphic to $S^{2n-1}\times S^{2n-1}$ by the evaluation maps at the two ends (as it is direct to check that given two endpoints, there is a unique Morse trajectory). Then the claim follows from the fact that the evaluation maps are obviously $G$-equivariant. Then the Morse-Bott cochain complex of $BG$ is $\prod_{j=1}^{\infty} H^*(S^{2n-1}/G;\Z)[-2jn]$, i.e.\ the direct product of the cohomology of critical manifolds with degree shifted by the Morse-Bott indices. The differential $\delta$ is defined by the push-pull formula
    \begin{equation}\label{eqn:pp}
        \delta:H^*(S^{2n-1}/G;\Z)[-2jn]\to H^*(S^{2n-1}/G;\Z)[-2(j+1)n],\quad x \mapsto (ev_+)_*\circ(ev_-)^*(x)
    \end{equation}
    where $ev_{\pm}$ are the evaluation maps on the moduli spaces connecting the critical manifolds of Morse-indices $2jn$ and $2(j+1)n$, i.e.\ the natural projections  $(S^{2n-1}\times S^{2n-1})/G \to S^{2n-1}/G$. Then by degree reasons, the only non-trivial term of $\delta$ is $\delta(\mathrm{vol}[-2jn])$ for the generator $\mathrm{vol}\in H^{2n-1}(S^{2n-1}/G;\Z)$, which is a multiple of $1\in H^0(S^{2n-1}/H;\Z)$. The pushforward $(ev_+)_*$ is the pushforward of homology composed with Poincar\'e dualities, hence we can compute that $\delta(\mathrm{vol}[-2jn])=|G|\cdot 1[-2(j+1)n]$, $H^{2n-1+2jn}(BG;\Z)=0$, $H^{2(j+1)n}(BG;\Z)=\Z/|G|$. As a consequence, we have that $H^*(BG;\Z)=H^{*+2n}(BG;\Z)$ for $*>0$. This puts strong restrictions on which finite group $G$ is allowed to have an isolated singularity $\C^n/G$. 
    
    The push-pull formula \eqref{eqn:pp} is a special case of a more general statement, where the push-pull on cohomology is the differential between critical manifolds with critical values next to each other. In particular, the moduli space of Morse trajectories between them is a closed manifold, hence we have Poincar\'e duality and the pushforward map. In general, there are more components to the differential between critical manifolds with critical values not next to each other, which can be roughly interpreted as the correction terms from the homological perturbation lemma. However, in the special case here, we do not need to consider those by degree reasons. The real coefficient version of such theory was worked out in \cite{zhou2019morse}.

    Instead of working out the integer version of \cite{zhou2019morse}, we can prove \eqref{eqn:pp} directly. We start by considering the case on $S^{4n-1}/G=:X_2$ with two critical manifolds of Morse-Bott indices $0$ and $2n$. Let $X_1$ denote the sub-level set $f^{-1}((-\infty,1/2])$. Then $\partial X_1\simeq (S^{2n-1}\times S^{2n-1})/G$ can be identified with the moduli space of Morse trajectories. In view of the long exact sequence from the pair $(X_2,X_1)$, we have $H^j(X_2;\Z)=H^j(X_1;\Z)=H^j(S^{2n-1}/G;\Z)$ when $j<2n-1$ and $H^j(X_2;\Z)=H^j(X_2,X_1;\Z)=H^{j-2n}(S^{2n-1}/G;\Z)$ when $j>2n$, where the last identification is from excision and the Thom isomorphism. The remaining term is determined by the following exact sequence
    $$0\to H^{2n-1}(X_2;\Z)\to H^{2n-1}(X_1;\Z)\to H^{2n}(X_2,X_1;\Z) \mapsto H^{2n}(X_2)\to 0.$$
    Note that the middle map factors as 
    $$H^{2n-1}(X_1;\Z)\to H^{2n-1}(\partial X_1;\Z) \to  H^{2n}(\overline{X_2\backslash X_1},\partial X_1;\Z) \stackrel{\text{excision}}{=}H^{2n}(X_2,X_1;\Z)$$
    Since $H^{2n-1}(S^{2n-1}/G;\Z)=H^{2n-1}(X_1;\Z)$ and the pullback $H^{2n-1}(X_1;\Z) \mapsto H^{2n-1}(\partial X_1;\Z)$ can be understood as $(ev_-)^*$ in \eqref{eqn:pp}, it suffices to show that $H^{2n-1}(\partial X_1;\Z )\to H^{2n}(\overline{X_2\backslash X_1},\partial X_1;\Z) \stackrel{\text{Thom iso}}{=}H^0(S^{2n-1}/G;\Z)$ is the pushforward $(ev_+)_*$. To see this, note that we have a commutative diagram
    $$
    \xymatrix{
    H^{2n-1}(\partial X_1;\Z) \ar[r]\ar[d]^{\mathrm{PD}} &  H^{2n}(\overline{X_2\backslash X_1},\partial X_1;\Z)\ar[d]^{\mathrm{LD}} \ar[r]^{\text{Thom iso}} & H^0(S^{2n-1}/G;\Z)\ar[d]^{\mathrm{PD}}  \\
    H_{2n-1}(\partial X_1;\Z) \ar[r] & H_{2n-1}(\overline{X_2\backslash X_1};\Z)\ar[r]^{\pi_*} & H_{2n-1}(S^{2n-1}/G;\Z)
    }
    $$
    where $\mathrm{PD},\mathrm{LD}$ are Poincar\'e duality and Lefschetz duality, and $\pi_*$ is the pushforward by the projection of the normal bundle of $f^{-1}(1)$ and all of them are isomorphisms. Since the composition $H_{2n-1}(\partial X_1;\Z) \to  H_{2n-1}(\overline{X_2\backslash X_1};\Z) \stackrel{\pi_*}{\to}  H_{2n-1}(S^{2n-1}/G;\Z)$ is the pushforward of $ev_+$ on homology, the claim follows. The general case follows from looking at the pair of sub-level sets $f^{-1}((-\infty, j+\frac{1}{2}])$ and $f^{-1}((-\infty, j+\frac{3}{2}])$, which up to homotopy is $(X_2,X_1)$ after $2jn$ suspensions.
    \item The above discussion can be generalized to the topological setting. Assume $G$ acts $S^{2n-1}$ freely. Then the diagonal $G$ action on the join $S^{4n-1}\simeq S^{2n-1}\star S^{2n-1}:=S^{2n-1}\times S^{2n-1}\times [0,1]/\sim$ is again free. We can inductively apply the argument above to obtain a model for $BG$. One can chop up $S^{4n-1}/G$ along the middle slice $(S^{2n-1}\times S^{2n-1}\times \{1/2\})/G$ (the moduli space of ``Morse trajectories") to compute the cohomology using the long exact sequence of pairs as before. Similarly, we also have the periodicity on $H^*(BG;\Z)$.
    \item The above discussion is an interpretation of a classical result on free actions on spheres using Morse theory. Smith \cite{MR10278} showed that if a finite group $G$ acts on $S^n$ freely\footnote{If $n$ is even, it is easy to show that $G$ has to be $\Z/2$ and reverses the orientation.}, then any abelian subgroup of $G$ must be cyclic. This turns out to be equivalent to the periodicity of the cohomology of $G$ \cite[Theorem 11.6]{homological} due to Artin and Tate.
    \item Groups that admit complex representations with isolated singularities are completely classified, c.f.\ \cite[Theorem 6.1.11,6.3.1]{MR2742530}.
\end{enumerate}
\end{remark}

\begin{proposition}\label{prop:abel}
Let $G$ be a finite abelian group acting on $\C^n$ with an isolated singularity, then $G=\Z/m$ and the action up to conjugation is given by sending the generator to multiplication of $\diag(e^{\frac{2\pi a_1 \i}{m}},\ldots, e^{\frac{2\pi a_n \i}{m}})$ for $1\le a_i<m$ coprime to $m$. 
\end{proposition}
\begin{proof}
In our linear setting\footnote{Even if we only assume $G$ is a continuous free action on $S^{2n-1}$, $G$ must be cyclic by \Cref{rmk:connecting} as $H^*(BG;\Z)$ is not periodic if $G$ is a finite non-cyclic abelian group.}, $G$ being abelian implies that they can be simultaneously diagonalized. Since the fixed subspace of every nontrivial element of $G$ is $\{0\}$, $G$ must be cyclic with the action in the statement.  
\end{proof}
We denote such lens space quotient $S^{2n-1}/G$ in \Cref{prop:abel} by $L(m;a_1,\ldots,a_n)$. 

\begin{proposition}\label{prop:chern}
Under the ring isomorphism $H^*(L(m;a_1,\ldots,a_n);\Z)=H^*(B\Z/m;\Z)$ for $*<2n-1$, the total Chern class of $(L(m;a_1,\ldots,a_n),\xi_{\std})$ is given by $\prod_{i=1}^n(1+a_iu)$ where $u$ is a generator of $H^2(B\Z/m;\Z)=\Z/m$.
\end{proposition}
Although the choice of the generator will be made clear in the proof, we do not need that level of precision as properties such as whether $c_k(\xi_{\std})=0$ and whether $c_k(\xi_{\std})$ generates $H^{2k}(L(m;a_1,\ldots,a_n);\Z)$ are independent of the choice of the generator.
\begin{proof}[Proof of \Cref{prop:chern}]
We use $\underline{\C}$ to denote the trivial complex bundle. Note that $\xi_{\std}\oplus\underline{\C}$ is isomorphic $\underline{\C}^n/(\Z/m)$, where $\Z/m$ acts diagonally as in \Cref{prop:abel}. Since $H^*(L(m;a_1,\ldots,a_n);\Z)=H^*(B\Z/m;\Z)$ for $*<2n-1$. We can compute the total Chern class of $\underline{\C}^n/(\Z/m)$ on $B\Z/m=\varinjlim \partial (\D^{2nk})/(\Z/m)$. Then the claim follows from the fact that $c(\underline{\C}/(\Z/m))=1+au$ for the $\Z/m$ action on $\C$ given by $z\mapsto e^{\frac{2\pi a \i}{m}}z$. Here $u$ is the first Chern class of $\underline{\C}/(\Z/m)$ for the $\Z/m$ action on $\C$ given by $z\mapsto e^{\frac{2\pi \i}{m}}z$, which is a generator of $H^2(B\Z/m;\Z)=H^2(\varinjlim \partial (\D^{2nk})/(\Z/m);\Z)$.
\end{proof}
\begin{proposition}\label{prop:EFM}
When $n\ge 3$,  $(S^{2n-1}/G,\xi_{\std})$ is not Weinstein fillable.
\end{proposition}
\begin{proof}
If $W$ is a Weinstein filling of $(S^{2n-1}/G,\xi_{\std})$, since $n\ge 3$, we have $\pi_1(W)=\pi_1(S^{2n-1}/G)=G$. Let $\widetilde{W}$ be the universal cover of $W$. Then $\widetilde{W}$ is a Weinstein filling of $(S^{2n-1},\xi_{\std})$, which by the Eliashberg-Floer-McDuff theorem \cite{mcduff1990structure} is diffeomorphic to a ball in $\C^n$. Now a non-trivial deck transformation acts freely, contradicting the Brouwer fixed-point theorem.
\end{proof}
In fact, most of them are not even almost Weinstein fillable.
\begin{proposition}\label{prop:almost_Weinstein}
 $(S^{2n-1}/G,J_{\std})$ is not almost Weinstein fillable when $n\ge 4$.
\end{proposition}
\begin{proof}
Assume for contradiction that there is an almost Weinstein filling $W$. Let $g\in G$ a non-trivial element, we consider the covering space $\widetilde{W}$ of $W$ corresponding to the subgroup $\la g \ra \subset \pi_1(W)=\pi_1(S^{2n-1}/G)=G$. As we assume $n\ge 3$, $\widetilde{W}$ is also constructed using only handles up to dimension $n$ with  $\partial \widetilde{W}=S^{2n-1}/\la g \ra $. Since $H^*(\partial \widetilde{W};\Z)=H^*(B\Z/\ord(g);\Z)$ as a ring in degrees less than $2n-1$, we have a class $\alpha$ in $H^2(\partial \widetilde{W};\Z)$ such that $\alpha^{n-1}\ne 0$. Since $\widetilde{W}$ is almost Weinstein and $n\ge 4$, we know that $H^2(\widetilde{W};\Z)\to H^2(\partial\widetilde{W};\Z)$ is an isomorphism. Therefore $\alpha=\beta|_{\partial\widetilde{W}}$ for $\beta \in H^2(\widetilde{W};\Z)$ and $\beta^{n-1}\ne 0$ as $\beta^{n-1}|_{\partial\widetilde{W}} = \alpha^{n-1}\ne 0$. This contradicts with the fact that $H^{2n-2}(\widetilde{W};\Z)=0$ when $n\ge 4$, since $\widetilde{W}$ is built from handles of indexes at most $n$.
\end{proof}

\begin{example}
Another intersecting class of examples of isolated quotient singularities comes from finite subgroups of $SU(2)$, namely the $A,D,E$ types of subgroups. The associated quotient singularity $\C^2/G$ is canonical and has a minimal (also crepant) resolution whose exceptional locus has an intersection pattern described by the associated Dynkin diagram. Moreover, the contact link $(S^3/G,\xi_{\std})$ has a Weinstein filling from plumbings of $T^*S^2$ according to the associated Dynkin diagram. The $A_{m-1}$ type subgroup is $\Z/m$, which acts on $\C^2$ by sending the generator to $\diag (e^{\frac{2\pi \i}{m}},e^{\frac{2(m-1)\pi \i}{m}})$. Details of $D,E$ types of subgroups can be found in \cite{MR1957212}. Given any finite $G\subset SU(2)$, it acts diagonally on $(\C^2)^n$ giving rise to an isolated terminal quotient singularity if $n>1$, see \Cref{prop:ADE_terminal}. The group cohomology ring $H^*(BG;\Z)$ is computed, e.g.\ in \cite{MR2484716}. In particular, we have a ring isomorphism $\oplus_{i=0}^\infty H^{4i}(BG;\Z)=\Z[v]/\la |G|\cdot v \ra$ with $\deg(v)=4$.
\end{example}

\begin{proposition}\label{prop:ADE_Chern}
Let $G\subset SU(2)$ be a finite subgroup, then $G$ acts on $(\C^2)^n$ by the natural diagonal action. The total Chern class of $(S^{4n-1}/G,\xi_{\std})$ is given by
$$(1-v)^n\mod \la |G|, v^n \ra,$$
where $v\in H^4(S^{4n-1}/G;\Z)=H^4(BG,\Z)=\Z/|G|$ is a generator.
\end{proposition}
\begin{proof}
We view $\C^2$ as the space of quaternions $\H$, then the $SU(2)\simeq S^3$ action on $\C^2$ is identified with the quaternion product by unit vectors in $\H$. Let $V$ denote the tautological quaternion line bundle over $\HP^{n-1}$, then $\underline{\H}^n/\H^{\times}=T\HP^{n-1}\oplus \underline \H = \oplus^n \Hom_{\H} (V,\underline{\H})$, where the first $\underline{\H}^n$ is the tangent bundle over $\H^n\backslash\{0\}$. Note that the underlying complex bundle $\left(\Hom_{\H} (V,\underline{\H})\right)_{\C}$ using $\i \in \H^{\times}$ is isomorphic to $\Hom_{\C}(V_{\C},\underline{\C})\simeq \overline{V_{\C}}$. We use $\pi$ to denote the map $S^{4n-1}/G\to \HP^{n-1}$ induced from the Hopf fibration $S^{4n-1}\to \HP^{n-1}$. Then the complex bundle $\xi_{\std}\oplus \underline{\C}$ is isomorphic to $\pi^*(\underline{\H}^n/\H^{\times})$ as a complex bundle, where  $\underline{\H}^n/\H^{\times}$ is equipped with complex structure from $\i\in \H^{\times}$. Since $c(V_{\C})=c(\overline{V_{\C}})=1-v$ for $v\in H^4(\HP^{n-1};\Z)$ is a generator and $\pi^*$ is reduction modulo $|G|$ on $H^{4i}(\HP^{n-1};\Z)$ for $1\le i\le n-1$, the claim follows.
\end{proof}

\begin{example}
Assume $G\subset O(n)$, such that the only fixed point of the $G$-action on $\R^n$ is $0$. If we view $G$ as in $O(n)\subset SU(n)$, the quotient singularity $\C^n/G$ is isolated. Moreover, since each $g\in G$ is represented by a real matrix, the eigenvalues must be conjugated complex numbers in pairs. Therefore $\age(g)=\frac{n}{2}$ (see \eqref{eqn:age}) for any $g\ne \Id$ and $\C^n/G$ is terminal if $n>2$ by \Cref{prop:terminal}. If $M$ is a closed orbifold with isolated singularities, then $T^*M$ is a symplectic orbifold with singularities of such type. 
\end{example}

\subsection{Reeb dynamics on $S^{2n-1}/G$}\label{ss:reeb}
Since $G\subset U(n)$, we have that $\alpha=\frac{1}{2}\sum (x_i\rd y_i-y_i\rd x_i)=-\frac{1}{2}r\rd r\circ I$ is $G$-invariant, where $I$ is the standard complex structure on $\C^n$.
Hence $\alpha$ descends to a contact form $\alpha_{\std}$ on $(S^{2n-1}/G,\xi_{\std})$.
We now recall the description of the Reeb dynamics of $\alpha_{\std}$ from \cite[\S 2.1]{McKay}. Since $c^{\Q}_1(\xi_{\std})=0$ and $\pi_1(S^{2n-1}/G)$ is finite, we can assign a well-defined rational generalized Conley-Zehnder index to every Reeb orbit following \cite[\S 4.1]{Reeb}, see also \cite[\S 3.3]{gironella2021exact}.

For $g\in U(n)$, we can diagonalize $g$ as $\diag(e^{\frac{2\pi a_1 \i}{\mathrm{o}(g)}},\ldots,e^{\frac{2\pi a_n \i}{\mathrm{o}(g)}})$, where $\mathrm{o}(g)$ is the order of $g$ and $1\le a_i<\mathrm{o}(g)$. Then we define 
\begin{equation}\label{eqn:age}
    \age(g):=\sum_{i=1}^n \frac{a_i}{\mathrm{o}(g)}\in \Q.
\end{equation}
$\age(g)$ only depends on the conjugacy class $(g)$.

For $l>0\in \R$ and $g\in V$, we define 
$$V_{g,l}=\left\{v\in \C^n\left|g(v)=e^{l\i}v\right.\right\}.$$
It is easy to check that $h(V_{g,l})=V_{hgh^{-1},l}$ for $h\in G$. 
Since the Reeb flow of $\alpha$ on $S^{2n-1}$ is given by $\phi_t(z)=e^{t\i}z$, we have the following properties of the Reeb dynamics of $\alpha_{\std}$ on $S^{2n-1}/G$. 

\begin{enumerate}
    \item Let $B_{g,l}=\pi_G(V_{g,l}\cap S^{2n-1})$, where $\pi_G$ is the quotient map $S^{2n-1}\to S^{2n-1}/G$. 
    Then, $B_{g,l}=(V_{g,l}\cap S^{2n-1})/C(g)$ is a submanifold of $S^{2n-1}/G$ of dimension $2\dim_{\C} V_{g,l}-1$, where $C(g)$ is centralizer of $g$.
    Moreover, $B_{g,l}=B_{hgh^{-1},l}$,
    hence, for $(g)\in \Conj(G)$, one can define $B_{(g),l}$ to just be $B_{g,l}$.
    \item $\displaystyle \sqcup_{(g),l} B_{(g),l}$ is the space of all \emph{parameterized} Reeb orbits of $\alpha_{\std}$. 
    For $z\in B_{(g),l}$, the time $l$ Reeb flow from $z$ is an orbit with free homotopy class $(g)\in \Conj(G)=\Conj(\pi_1(S^{2n-1}/G))$.
    \item  The generalized Conley-Zehnder index of the family $B_{(g),l}$ is given in \cite[Theorem 2.6]{McKay} by 
    $$\mu_{\CZ}(B_{(g),l})=n-2\age(g)+2\sum_{l'<l}\dim_{\C} V_{g,l'}+\dim_{\C}V_{g,l}.$$
    Then we have $\mu_{\CZ}(B_{(g),l+2k\pi})=\mu_{\CZ}(B_{(g),l})+2kn$. 
    \item\label{reeb} The space of \emph{unparameterized} Reeb orbits is the orbifold $B_{(g),l}/S^1$, which is Morse-Bott non-degenerate in the sense of \cite[Definition 1.7]{MR2703292}. By \cite[Lemma 2.3]{MR2703292}, we can find a small perturbation of $\alpha_{\std}$, such that all Reeb orbits with period smaller than any given $R\gg 0$ are non-degenerate and correspond to critical points of auxiliary Morse functions on  $B_{(g),l}/S^1$. The minimal rational Conley-Zehnder index is given by 
    $$\min_{(g)\in \Conj(G)}\left\{n-2\age(g)+1\right\}.$$
    After such a perturbation, the $\Z/2$ Conley-Zehnder index of the orbit in each conjugacy class with the minimal rational Conley-Zehnder index (or period) is $n+1$. One way to see this is through direct computation. There is a roundabout as follows: in the computation of the symplectic cohomology of $\C^n/G$ in \cite[Theorem B]{gironella2021exact}, there are differentials from those orbits with minimal rational Conley-Zehnder index to constant twisted sectors at the origin. Now by \cite[Example 3.7]{gironella2021exact}, those constant twisted sectors have $\Z/2$ Conley-Zehnder index $n$, which implies the claim.
\end{enumerate}

In \S \ref{s3}, we will consider Hamiltonians on hypothetical fillings of $(S^{2n-1}/G,\xi_{\std})$ that are in the form of $H(r)$, where $r\in [1,\infty)$ is the radical coordinate determined by the contact form $\alpha_{\std}$. We will assume that  with $H''(r)\ge 0$ for $r\ge 1$ and $H'(r)=a \notin \frac{2\pi \N}{|G|}$ for $x>1+\epsilon$ for some $\epsilon>0$. Then non-constant Hamiltonian orbits of $H$ are parameterized by $B_{(g),l}$ for $l<a$. In particular, the symplectic cohomology can be equivalently defined using the cascades construction, c.f.\ \cite{cascades,divisor}, i.e.\ we need to choose a Morse function $f_{(g),l}$ for each $B_{(g),l}$ and the generators from non-constant orbits of $H$ are replaced by critical points of $f_{(g),l}$. 
The Conley-Zehnder index of a critical point $x$ of $f_{(g),l}$ is given by 
\begin{equation}\label{eqn:CZ}
    \mu_{CZ}(x)=n-2\age(g)+2\sum_{l'<l}\dim_{\C}V_{g,l'}+1+\ind(x),
\end{equation}
where $\ind(x)$ is the Morse index of $x$. With such conventions, the cascade flow lines in $B_{(g),l}$ are \emph{negative} gradient flow lines of $f_{(g),l}$ in the cascade construction. On the other hand, if we use $f_{(g),l}$ to perturb the Hamiltonian following \cite{cascades}, then each critical point of $x$ gives rise to a non-degenerate Hamiltonian orbit $\gamma_x$, whose Conley-Zehnder index is given by \eqref{eqn:CZ}. The symplectic action of $\gamma_x$ is greater than that of $\gamma_y$ iff $f_{(g),l}(x)<f_{(g),l}(y)$. Moreover, we may choose $f_{(g),l}$ to have a unique minimum. Then for each conjugacy class $(g)$, there is a unique generator $x_{(g)}$ with the minimal rational Conley-Zehnder index $n-2\age(g)+1$, or equivalently minimal period, or maximal symplectic action. Those generators will have critical roles in \S \ref{s3}, \ref{s4}.

In general, given a $2n-1$ dimensional contact manifold $(Y,\xi)$ with $H_1(Y;\Z)$ is torsion and $c^\Q_1(Y)=0$, then every Reeb orbit can be assigned with a well-defined rational Conley-Zehnder index $\mu_{CZ}(\gamma)$\footnote{The Reeb orbit is not necessarily non-degenerate, the Conley-Zehnder index is defined in \cite{zbMATH00550345}.}. Then we define the lower SFT index by
\begin{equation}\label{eqn:lSFT}
\mathrm{lSFT}(\gamma):=\mu_{CZ}(\gamma)+(n-3)-\frac{1}{2}\dim \ker (\rD_{\gamma(0)}\phi_L|_{\xi}-\Id),
\end{equation}
where $\phi_L$ is the time $L$ flow of the Reeb vector field and $L$ is the period of $\gamma$. When $\gamma$ is non-degenerate, $\mathrm{lSFT}(\gamma)$ is the standard SFT grading. When $\gamma$ comes from a Morse-Bott family, then $\mathrm{lSFT}$ is the minimal SFT grading of the non-degenerate orbits after a small perturbation. Following \cite{Reeb}, the highest minimal SFT index is defined as
$$\mathrm{hmi}(Y,\xi):= \sup_{\alpha} \inf_{\gamma} \{\mathrm{lSFT}(\gamma)\}.$$

\subsection{Chen-Ruan cohomology of $\C^n/G$}
From a string theory point of view, the correct cohomology theory of an almost complex orbifold is the Chen-Ruan cohomology introduced in \cite{CheRua04}, which can be viewed as the constant part of the orbifold quantum product introduced in \cite{CheRua01}. The Chen-Ruan cohomology is the replacement of the ordinary cohomology in the definition of orbifold symplectic cohomology \cite{gironella2021exact}. We refer readers to \cite{Ruan07} for the definition and properties of the Chen-Ruan cohomology.

For orbifold $\C^n/G$, with $G\subset U(n)$ having only isolated singularity at $0$, we have 
\begin{equation}\label{eqn:CR_cohomology}
    H^*_{\CR}(\C^n/G;R)=\oplus_{(g)\in \Conj(G)} R[-2\age(g)].
\end{equation}

\subsubsection{Product on Chen-Ruan cohomology}\label{ss:coproduct}

To describe the product structure on $ H^*_{\CR}(\C^n/G;R)$, we denote by $[(g)]$ the (positive) generator in cohomology for the conjugacy class $(g)$. For $(g_1),(g_2)\ne (\Id)$, one can define 
	\[I_{g_1,g_2}=\{(h_1,h_2)|h_i\in (g_i), \age(h_1)+\age(h_2)=\age(h_1h_2)\}/(h_1,h_2)\sim (gh_1{g^{-1}},gh_2g^{-1}).\]
Then by \cite[Examples 4.26 and 4.27]{Ruan07}, we have 
	\begin{equation}
	\label{eqn:decomposition}
	[(g_1)]\cup [(g_2)]=\sum_{(h_1,h_2)\in I_{g_1,g_2}} \frac{\vert C(h_1h_2)\vert}{|C(h_1)\cap C(h_2)|} [(h_1h_2)],
	\end{equation}
	for $(g_1),(g_2)\ne (\Id)$, while $[(\Id)]$ is the ring unit in $H^*_{\CR}(\C^n/G;R)$. 

The formula \eqref{eqn:decomposition} can be explained through holomorphic curves as follows. Just like the manifold case, the cup product can be defined by counting moduli spaces of punctured holomorphic curves of zero energy with $2$ inputs (positive punctures) and $1$ output (negative puncture). But in the orbifold case at hand, there are constant holomorphic curves which, near punctures, are asymptotic to twisted sectors at the singularity. In fact, each $(h_1,h_2)\in I_{g_1,g_2}$ corresponds to a constant holomorphic curve with two positive punctures asymptotic to $(g_1),(g_2)$ and one negative puncture asymptotic to $(h_1h_2)$, which is cut out transversely (by the conditions on the $\age$ invariants). Moreover, such a holomorphic curve is an orbifold point in the space of curves in an orbifold with isotropy group $C(h_1)\cap C(h_2)\cap C(h_1h_2)=C(h_1)\cap C(h_2)$. In view of the pull-push definition of the structural constants in \cite[\S 2.1.3]{gironella2021exact}, the coefficient is given by  ${\vert C(h_1h_2)\vert}/{|C(h_1)\cap C(h_2)|}$.

\subsubsection{Coproduct on Chen-Ruan cohomology}
By the same philosophy, we can define the Chen-Ruan coproduct on the Chen-Ruan cohomology, which is dual to the Chen-Ruan cohomology through the Poincar\'e/Lefschetz duality on Chen-Ruan cohomology. In the case of $\C^n/G$, we define
	\[I_{g}=\{(h_1,h_2)|h\in (g), h_1h_2=h, \age(h_1)+\age(h_2)=\age(h)+n\}/(h_1,h_2)\sim (ph_1{p^{-1}},ph_2p^{-1}).\]
Then we define the Chen-Ruan coproduct to be  
\begin{equation}\label{eqn:coproduct}
    \Delta[(g)] = \sum_{(h_1,h_2)\in I_g} \frac{|C(h_1)|\cdot |C(h_2)|}{|C(h_1)\cap C(h_2)|}[(h_1)]\otimes [(h_2)].
\end{equation}
Here, as before, each element in $I_g$ corresponds to a constant holomorphic curve with two negative punctures asymptotic to $(h_1),(h_2)$ and one positive puncture asymptotic to $(g)$, which is cut out transversely with isotropy group $C(h_1)\cap C(h_2)$. The coefficient is again a consequence of the pull-push definition of the structural constants on orbifolds. It is indeed dual to the Chen-Ruan product on the \emph{compactly supported} Chern Ruan cohomology considering the Lefschetz duality \cite[\S 4.2]{Ruan07} on the inertia orbifold of $\C^n/G$ has coefficients from the isotropy groups. 

As a special case, we have
$$\Delta[(\Id)] = \sum_{(g)\ne (\Id)\in \Conj(G)} |C(g)| [(g)]\otimes [(g^{-1})].$$
In other words, the ``orbifold intersection form" is non-degenerate.

\subsection{Minimal discrepancy number of $\C^n/G$}
Let $A$ be a singularity, that is isolated at $0$. Let $\pi:\widetilde{A}\to A$ be a resolution (if $0$ is smooth, we need to blow up once), such that $\pi^{-1}(0)$ a union of smooth normal crossing divisors $E_i$. We view the canonical bundle $K_{\widetilde{A}}$ of $\widetilde{A}$ as a $\Q$-Cartier divisor. We say $A$ is numerically $\Q$-Gorenstein if there exists a $\Q$-Cartier divisor $K^{\rm{num}}_{\widetilde{A}}=\sum a_i E_i$ with the property that $C\cdot (K^{\rm{num}}_{\widetilde{A}}-K_{\widetilde{A}})=0$ for any projective algebraic curve $C$ in $\pi^{-1}(0)$. Here $a_i \in \Q$ is called the discrepancy of $E_i$. Following \cite[Lemma 3.3]{Reeb}, $A$ being numerically $\Q$-Gorenstein is equivalent to that $c_1^{\Q}(L_A,\xi)=0$, where $L_A$ is the link of the singularity and $\xi$ is induced contact structure \cite{MR570582}. Moreover, $c^{\Q}_1(\widetilde{A})$ viewed as a class in $H^2_c(\widetilde{A};\Q)=H^2(\widetilde{A},L_A;\Q)$ is the Lefschetz dual to $\sum a_i[E_i]$, where $a_i$ is the  discrepancy of $E_i$ above. Since $H^2(S^{2n-1}/G;\Q)=0$, isolated quotient singularities are numerically $\Q$-Gorenstein.

\begin{definition}[{\cite[Definition 3.4]{Reeb}}]\label{def:terminal}
The minimal discrepancy number $\mathrm{md}(A,0)$ is the given by infimum of $a_i$ over all resolutions. A singularity is called terminal if the minimal discrepancy number is positive. A singularity is called canonical if the minimal discrepancy number is non-negative.
\end{definition}
A beautiful theorem of McLean relates the algebraic geometry of a singularity with the contact geometry of the link, which applies to quotient singularities.
\begin{theorem}[McLean {\cite[Theorem 1.1]{Reeb}}]\label{thm:McLean}
Let $A$ have a normal isolated singularity at $0$ that is numerically $\Q$-Gorenstein with $H^1(L_A; \Q) = 0$. Then:
\begin{itemize}
    \item If $\mathrm{md}(A, 0) \ge 0$ then $\mathrm{hmi}(L_A, \xi_A) = 2\mathrm{md}(A, 0)$;
    \item If $\mathrm{md}(A, 0) < 0$ then $\mathrm{hmi}(L_A, \xi_A) < 0$.
\end{itemize}
\end{theorem}

As an application of McLean's theorem, we obtain a symplectic proof of the Reid--Shepherd-Barron--Tai criterion \cite[(3.1)]{canonical}, \cite[Proposition 3.2]{Kodaira} for isolated quotient singularities.

\begin{proposition}\label{prop:terminal}
The isolated singularity $\C^n/G$ is terminal (canonical) if and only if $\min_{g\ne \Id \in G}\{\age(g)\}>1$ ($\ge 1$). More precisely, we have $ \md(\C^n/G)= \min_{g\ne \Id \in G}\{\age(g)\}-1$, if the singularity is canonical.
\end{proposition}
\begin{proof}
By the standard Reeb dynamics on $S^{2n-1}/G$ in \S \ref{ss:reeb}, we have $\mathrm{hmi}(S^{2n-1}/G,\xi_{\std})\ge \min_{g\in G}\{2n-2\age(g)-2\}$. On the other hand, by \cite[Theorem B]{gironella2021exact}, $SH^*(\C^n/G;\Q)=0$, which implies that $SH^*_{S^1}(C^n/G;\Q)=0$ by the Gysin exact sequence \cite[Theorem 1.1]{Gysin} (this holds in the exact orbifold setting from the discussion in \cite[\S 3.6, 3.7 ]{gironella2021exact}). As a consequence, we have $SH^{*-1}_{+,S^1}(\C^n/G;\Q)=H^*_{\CR}(\C^n/G;\Q)\otimes (\Q[u,u^{-1}]/\la u\ra)$. In particular, $SH^*_{+,S^1}(\C^n/G;\Q)$ is nontrivial in degree $2\age(g)-1$ for any $g\in G$. If $\mathrm{hmi}(S^{2n-1}/G,\xi_{\std})>  \min_{g\in G}\{2n-2\age(g)-2\}$, then there exists a contact form $\alpha$, with $\inf_{\gamma} \{\mathrm{lSFT}(\gamma) \}> \min_{g\in G}\{2n-2\age(g)-2\}$. By \cite[\S 4.3]{Reeb}, for arbitrarily large period threshold $K$, any small perturbation of $\alpha$ yields a contact form such that all Reeb orbits with period smaller than $R\gg 0$ is non-degenerate and the SFT degree is at least $\inf_{\gamma} \{\mathrm{lSFT}(\gamma) \}$. Then by \cite{gutt}, we have that $SH^*_{+,S^1}(\C^n/G;\Q)$ is supported in degrees at most $2n-3-\inf_{\gamma} \{\mathrm{lSFT}(\gamma) \}< 2n-3-\min_{g\in G}\{2n-2\age(g)-2\}=2\max_{g\in G} \{\age(g)\}-1$, contradiction. Therefore we have $\mathrm{hmi}(S^{2n-1}/G,\xi_{\std})= \min_{g\in G}\{2n-2\age(g)-2\}$. Finally note that $2n-2\age(\Id)-2=2n-2\ge 2$ and $\age(g)+\age(g^{-1})=2n$ for $g\ne \Id$. We have $\min_{g\in G}\{2n-2\age(g)-2\}>0(\ge 0)$ is equivalent to that $\min_{g\ne \Id\in G}\{\age(g)\}>1(\ge 1)$. Then the claim follows from \Cref{thm:McLean}.
\end{proof}

\section{Symplectic cohomology of fillings of $(S^{2n-1}/G,\xi_{\std})$}\label{s3}

\subsection{Symplectic cohomology}
We first give a very brief account of the definition and properties of symplectic cohomology to set up notations for later reference. More detailed discussions on various aspects of symplectic cohomology can be found in \cite{ES,biased}. 

Given an exact filling $(W,\lambda)$, roughly speaking, symplectic cohomology is the ``Morse cohomology" of the free loop space w.r.t.\ the symplectic action functional
\begin{equation}\label{eqn:action}
    \cA(x):=-\int x^*\widehat{\lambda} +\int_{S^1} (x^*{H})\rd t
\end{equation}
where $H:S^1\times \widehat{W}\to \R$ is a Hamiltonian on the completion $(\widehat{W},\widehat{\lambda}):=(W\cup \partial W \times (1,\infty)_r,\lambda\cup r\lambda|_{\partial W})$ such that ``the slope $\frac{\rd }{\rd r}H$" goes to infinity as $r$ goes to infinity. Let $R$ be any commutative ring, the cochain complex $C^*(H)$ is a free $R$-module generated by $1$-periodic orbits of $X_H$ and the differential is defined by counting rigid solutions to the Floer equation. Let $x,y$ be two generators in $C^*(H)$ represented by periodic orbits, we use $\cM_{x,y}$ to denote the compactified moduli space of Floer cylinders from $x$ to $y$, whose specific meaning depends on the construction of the symplectic cohomology as follows.
\begin{enumerate}
    \item If $H$ is a non-degenerate Hamiltonian, i.e.\ all $1$-periodic orbits are non-degenerate, then we have
    $$\cM_{x,y}=\overline{\left\{u:\R_s\times S^1_t\to \widehat{W}\left|\partial_su+J(\partial_ut-X_H)=0, \displaystyle \lim_{s\to \infty} u = x, \lim_{s\to-\infty} u=y\right. \right\}/\R }.$$
    \item If $H$ is Morse-Bott non-degenerate, including the case of $H=0$ on $W$ considered in \cite{divisor,ADC}, then $x,y$ are critical points of the auxiliary Morse functions on the Morse-Bott families of periodic orbits (which can be viewed as a submanifold in $\widehat{W}$ using the starting point of the orbit). In this case, $\cM_{x,y}$ is the compactified moduli space of cascades from $x$ to $y$, which can be pictorially described as 
    \begin{figure}[H]\label{fig:cascades}
	\begin{tikzpicture}
	\draw (0,0) to [out=90, in = 180] (0.5, 0.25) to [out=0, in=90] (1,0) to [out=270, in=0] (0.5,-0.25)
	to [out = 180, in=270] (0,0) to (0,-1);
	\draw[dotted] (0,-1) to  [out=90, in = 180] (0.5, -0.75) to [out=0, in=90] (1,-1);
	\draw (1,-1) to [out=270, in=0] (0.5,-1.25) to [out = 180, in=270] (0,-1);
	\draw (1,0) to (1,-1);
	\draw[->] (1,-1) to (1.25,-1);
	\draw (1.25,-1) to (1.5,-1);
	\draw (1.5,-1) to [out=90, in = 180] (2, -0.75) to [out=0, in=90] (2.5,-1) to [out=270, in=0] (2,-1.25)
	to [out = 180, in=270] (1.5,-1) to (1.5,-2);
	\draw[dotted] (1.5,-2) to  [out=90, in = 180] (2, -1.75) to [out=0, in=90] (2.5,-2);
	\draw (2.5,-2) to [out=270, in=0] (2,-2.25) to [out = 180, in=270] (1.5,-2);
	\draw (2.5,-1) to (2.5,-2);
	\draw[->] (2.5,-2) to (2.75,-2);
	\draw (2.75,-2) to (3,-2);
	\draw[->] (-0.5,0) to (-0.25,0);
	\draw (-0.25,0) to (0,0);
	\node at (0.5,-0.5) {$u_1$};
	\node at (2, -1.5) {$u_2$};
	\fill (-0.5,0) circle[radius=1pt];
	\node at (-0.7, 0) {$x$};
	\fill (3,-2) circle[radius=1pt];
	\node at (3.2, -2) {$y$};
	\end{tikzpicture}
	\begin{tikzpicture}
	\fill (-0.5,0) circle[radius=1pt];
	\node at (-0.7, 0) {$x$};
	\draw (0,0) to [out=90, in = 180] (0.5, 0.25) to [out=0, in=90] (1,0) to [out=270, in=0] (0.5,-0.25)
	to [out = 180, in=270] (0,0) to (0,-1);
	\draw[dotted] (0,-1) to  [out=90, in = 180] (0.5, -0.75) to [out=0, in=90] (1,-1);
	\draw (1,-1) to [out=270, in=0] (0.5,-1.25) to [out = 180, in=270] (0,-1);
	\draw (1,0) to (1,-1);
	\draw[->] (1,-1) to (1.25,-1);
	\draw (1.25,-1) to (1.5,-1);
	\draw (1.5,-1) to [out=90, in = 180] (2, -0.75) to [out=0, in=90] (2.5,-1) to [out=270, in=0] (2,-1.25)
	to [out = 180, in=270] (1.5,-1) to (1.5,-2);
	\draw (2.5,-2) to [out=270, in=0] (2,-2.5) to [out = 180, in=270] (1.5,-2);
	\draw (2.5,-1) to (2.5,-2);
	\draw[->] (-0.5,0) to (-0.25,0);
	\draw (-0.25,0) to (0,0);
	\draw[->] (2,-2.5) to (2.5,-2.5);
	\draw (2.5,-2.5) to (3,-2.5);
	\fill (3,-2.5) circle[radius=1pt];
	\node at (3.2,-2.5) {$y$};
	\node at (0.5,-0.5) {$u_1$};
	\node at (2, -1.5) {$u_2$};
	\end{tikzpicture}
	\caption{$2$ level cascades}
\end{figure}
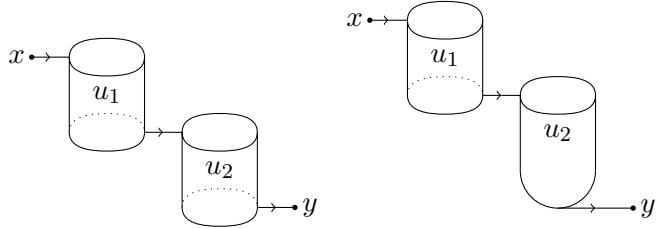
     \noindent
     Here the horizontal lines are \emph{negative gradient flow} of the auxiliary Morse function used in \eqref{eqn:CZ} except for the bottom line of the right of the figure above, which is the \emph{gradient flow} of an admissible Morse function \cite[Definition 2.1]{ADC} on $W$ when $H=0$ on $W$. For the more formal description of the moduli spaces see \cite[Definition 4.7]{divisor} or \cite[\S 2]{RP} for details. 
\end{enumerate}
In both cases, for a generic compatible (which is also cylindrically convex to guarantee the integrated maximum principle \cite{MR2602848,ES}) almost complex structure $J:S^1\to \End(T^*\widehat{W})$, $\cM_{x,y}$ is cut out transversely as a manifold with boundary for those with expected dimension $\virdim \cM_{x,y}\le 1$. Moreover, $\cM_{x,y}$ can be oriented in a coherent way such that 
$$\delta(x)=\sum_{y,\virdim \cM_{x,y}=0}(\# \cM_{x,y}) y$$
defines a differential on $C^*(H)$.

\begin{remark}
The symplectic cochain complex is graded by $n-\mu_{\CZ}(x)$. In general, since $\mu_{\CZ}$ is only well-defined in $\Z/2$, symplectic cohomology always has a $\Z/2$ grading. If $c^{\Z}_1(W)=0$, then $\mu_{\CZ}$ is well-defined in $\Z$ for any periodic orbits with a finite order homology class. Moreover, if $c^{\Q}_1(W)=0$, then  $\mu_{\CZ}$ is well-defined in $\Q$ (as in \S \ref{ss:reeb}) for any periodic orbits with finite order homology class and symplectic cohomology can be graded by $\Q$. For our situation, although $c^{\Q}_1(S^{2n-1}/G,\xi_{\std})=0$ and $\pi_1(S^{2n-1}/G)$ is of finite order, the Conley-Zehnder index in \S \ref{ss:reeb} is only well-defined using a trivialization of $\det_{\C}(\oplus^N\xi_{\std})$, which can not extend to the filling $W$ if $c^{\Q}_1(W)\ne 0$. Therefore typically, the symplectic cohomology of $W$ is not graded by $\Q$. However, the existence of the boundary grading indeed plays an important role in the proof.
\end{remark}

Symplectic cohomology $SH^*(W;R)$ has the following properties:
\begin{enumerate}
    \item If we choose $H$ to be $C^2$ small on $W$ and to be $h(r)$ on $\partial W \times (1,\infty)_r$ with $h''(r)>0$, then the periodic orbits of $X_H$ are either constant orbits on $W$ or non-constant orbits on $\partial W \times (1,\infty)$, which, in pairs after a small perturbation, correspond to Reeb orbits on $(\partial W,\lambda|_{\partial W})$.  The constant orbits generate a subcomplex corresponding to the cohomology of $W$ (i.e.\ the space of constant loops), and the non-constant orbits generate a quotient complex $C^*_+(H)$, whose cohomology is called the positive symplectic cohomology $SH^*_+(W;R)$. Then we have a tautological long exact sequence,
    \begin{equation}\label{eqn:LES}
        \ldots \to H^*(W;R)\to SH^*(W;R)\to SH^{*}_+(W;R)\to H^{*+1}(W;R)\to \ldots
    \end{equation}
    \item $SH^*(W;R)$ is a unital ring and $H^*(W;R)\to SH^*(W;R)$ is a unital ring map.
\end{enumerate}

Most of the above discussion (except for the ring structure) can be generalized to the orbifold setting using classical techniques, see \cite{gironella2021exact}. Since $\C^n/G$ is the natural exact orbifold filling of $(S^{2n-1}/G,\xi_{\std})$, the following result provides the natural reference to compare our hypothetical exact fillings to.
\begin{theorem}[{\cite[Theorem B]{gironella2021exact}}]\label{thm:quotient_SH}
Let $R$ be a commutative ring such that $|G|$ is invertible. Then we have $SH^*(\C^n/G;R)=0$ for the isolated quotient singularity $\C^n/G$. In particular, we have $SH^*_+(\C^n/G;R)=H^{*+1}_{\CR}(\C^n/G;R)$ (using the $\Q$-grading), which is generated by $x_{(g)}$ for $(g)\in \Conj(G)$. Here $x_{(g)}$ are those generators (Hamiltonian orbits) with minimal rational Conley-Zehnder index for conjugacy class $(g)$ defined after \eqref{eqn:CZ}.
\end{theorem}

\subsection{Symplectic cohomology of exact fillings of $(S^{2n-1}/G,\xi_{\std})$}
Using the Hamiltonian in \S \ref{ss:reeb}, we define the boundary grading of a non-constant periodic orbit (or a critical point of the auxiliary Morse function on a Morse-Bott family of non-constant period orbits) as 
$$|x|^{\partial}:=n-\mu_{\CZ}(x),$$
where $\mu_{\CZ}$ is the rational Conley-Zehnder index in \S \ref{ss:reeb}. With this rational boundary grading, we have a filtration $F^kC_+(H)\supset F^{k+1}C_+(H)$ for $k\in \Z$ defined by
\begin{equation}\label{eqn:filtration}
    F^kC_+(H):=\left\langle x\left| |x|^{\partial}\ge \frac{k}{|G|}\right.\right\rangle.
\end{equation}

\begin{proposition}\label{prop:SS}
When $\C^n/G$ is terminal, the differential is compatible with the filtration and the first $|G|+1$ pages of the associated spectral sequence are independent of exact fillings/augmentations. When  $|G|$ is invertible in the coefficient ring,  the $(|G|+1)$th page of the spectral sequence is spanned by $\{ x_{(g)}\}_{(g)\in \Conj(G)}$ and degenerates.
\end{proposition}
\begin{proof}
Let $x,y$ be two generators of the Hamiltonian-Floer cohomology, we can apply neck-stretching to the moduli space $\cM_{x,y}$ along the contact boundary $S^{2n-1}/G$, see \cite[\S 2.2]{RP} for more details on neck-stretching in a similar setting. In the fully-stretched picture, the top level consists of Hamiltonian-Floer cylinders possibly with negative punctures asymptotic to Reeb orbits in $B_{(g_1),l_1},\ldots,B_{(g_m),l_m}$. If there is at least one negative puncture, the virtual dimension of the corresponding moduli space is 
\begin{equation}\label{eqn:top}
|y|^{\partial}-|x|^{\partial}-\sum_{i=1}^m\lSFT(B_{(g_i),l_i})-1\le |y|^{\partial}-|x|^{\partial}-2\md(\C^n/G)-1.
\end{equation}
We point out here even though $|y|^{\partial},|x|^{\partial},\lSFT(B_{(g_i),l_i})$ are only rational, the virtual dimension, i.e.\ the left hand side of \eqref{eqn:top} is an integer. Since we use $S^1$-dependent almost complex structures, we can assume the transversality for the top level. Note that $\md(\C^n/G)>0$ by the terminality of the singularity, to have a top-level curve after the neck-stretching, we must have $|y|^{\partial}\ge |x|^{\partial}+1$, hence the differential is compatible with filtration. Moreover, the filtration is clearly Hausdorff, hence induces a convergent spectral sequence. 

Since the differential on the first $|G|-1$ pages is counting  $\cM_{x,y}$ with  $|y|^{\partial}-|x|^{\partial}< 1$.  In view of \eqref{eqn:top}, the expected dimension of the top level with at least one negative puncture is $<-2\md(\C^n/G)<0$. Therefore those differentials are from curves contained in the symplectization. However, $|y|^{\partial}-|x|^{\partial}< 1$, those moduli spaces in the symplectization have negative virtual dimensions and hence are empty. The differential on the $|G|$th page counts  $\cM_{x,y}$ with  $|y|^{\partial}-|x|^{\partial}\le 1$ (note that $|G||x|^{\partial}$ is always an integer),  which by the same argument, is from curves contained in the symplectization and is independent of fillings. Therefore the first $|G|+1$ pages of the induced spectral sequence are independent of augmentations/fillings (not necessary for the differential on the $(|G|+1)$th page).

From the above discussion, we see that moduli spaces $\cM_{x,y}$ with $|y|^{\partial}-|x|^{\partial}=1$ are the only non-empty moduli spaces appearing in the computation of the first $|G|+1$ pages of the spectral sequence.  Now, we can consider the orbifold filling $\C^n/G$ of $(S^{2n-1}/G,\xi_{\std})$, we have $c_1^{\Q}(\C^n/G)=0$. Therefore the filtration is just the natural filtration from the $\Q$-grading on the symplectic cohomology and the differential is solely contributed by those  $\cM_{x,y}$ with $|y|^{\partial}-|x|^{\partial}=1$. By the independence of fillings, the $(|G|+1)$th page of the spectral sequence is spanned by $x_{(g)}$,  as $SH_+^*(\C^n/G)$ is spanned by  $x_{(g)}$ by \Cref{thm:quotient_SH}. Finally, the degeneracy follows tautologically as the $\Z/2$ grading of $x_{(g)}$ is odd by the discussion after \eqref{eqn:CZ} and all differentials in the spectral sequence respect this $\Z/2$ grading.
\end{proof}

The following is a generalization of \cite[Proposition 2.12, 2.14]{RP}, which is another place where the terminality is crucial.
\begin{proposition}\label{prop:upper_bound}
If $\C^n/G$ is an isolated terminal singularity. Let $W$ be an exact filling of $(S^{2n-1}/G,\xi_{\std})$ and $R$ be a ring where $|G|$ is invertible, then $SH^*(W;R)=0$, $H^*(W;\R)$ is a free $R$-module with $\rank H^*(W;R) = \# \Conj(G)$ and $H^*(W;R)$ is supported in even degrees. 
\end{proposition}
\begin{proof}
Let $a$ represent the cohomology class in $SH^*_+(W)$ represented by $x_{(\Id)}$ in the spectral sequence in \Cref{prop:SS}. This $x_{(\Id)}$, when lifted to $S^{2n-1}$, is the Hamiltonian orbit with the minimal Conley-Zehnder index ($=n+1$) and kills the unit in $SH^*(\C^n)$. In the following, we will show that $x_{(\Id)}$ kills the unit for hypothetical exact fillings of  $S^{2n-1}/G$ as well. Since $a=x_{(\Id)}$ in the spectral sequence, we can only assume  $a=x_{(\Id)}+\sum_{i=1}^n c_ix_i$ for $c_i\in R$ and $|x_i|^{\partial}\ge |x_{(\Id)}|^{\partial}+\frac{1}{|G|}=-1+\frac{1}{|G|}$. Now we consider differentials from a non-constant orbit $x$ to $1$, or equivalently, the moduli space $\cM_{x,q}$ in \cite[Proposition 2.12]{RP} with $q\in S^{2n-1}/G$ a point constraint on the contact boundary\footnote{The point constraint can be chosen freely, typically from the interior of the filling (as the unique minimum of an auxiliary Morse function on the filling). Here we choose it to be on the boundary for the convenience of neck-stretching.}. We can apply neck-stretching for $\cM_{x_{(\Id)},q}$ as in \cite[Proposition 2.12 step 1]{RP}. The virtual dimension of the top SFT level is 
$$2n-|x_{(\Id)}|^{\partial}-\sum_{i=1}^l \lSFT(B_{(g_i),l_i})-2n-1\le 0,$$
where $B_{(g_i),l_i}$ are where the asymptotics of the negative punctures land. Since $ \lSFT(B_{(g_i),l_i})>0$ by the terminality assumption, the equality holds iff $l=0$. Then by the same covering argument in \cite[Proposition 2.12 step 3]{RP}, we have $\#\cM_{x_{(\Id),q}}=|G|$. Next for  $\cM_{x_i,q}$, the virtual dimension of the top SFT level after neck-stretching is 
$$2n-|x_i|^{\partial}-\sum_{i=1}^l \lSFT(B_{(g_i),l_i})-2n-1<0.$$
Hence we may assume $\cM_{x_i,q}=\emptyset$. As a consequence, we have $SH^*_+(W;R)\to H^{*+1}(W;R)\to H^0(W;R)$ is surjective by $a$, i.e.\ $a$ is mapped to $|G|+A$ through $SH^*_+(W;R)\to H^{*+1}(W;R)$ for $A\in \oplus_{j>0}H^j(W;R)$. Note that $|G|+A$ is a unit in $H^*(W;R)$, the long exact sequence \eqref{eqn:LES} implies that $SH^*(W;R)=0$ and $SH_+^*(W;R)=H^{*+1}(W;R)$ in the $\Z/2$-grading sense. Since the $(|G|+1)$th page of the spectral sequence in \Cref{prop:SS} consists of free finitely generated $R$-modules, there is no non-trivial extension, i.e.\ $SH_+^*(W;R)$ is a free $R$-module of rank $\# \Conj(G)$. Hence the claim follows as the $\Z/2$ grading of $x_{(g)}$ is odd.
\end{proof}
\begin{remark}\label{rmk:grading}
Under a non-canonical isomorphism $ \oplus_{(g)\in \Conj(G)} \la x_{(g)} \ra \simeq SH_+^*(W;R)\simeq H^{*+1}(W;R)$ in \Cref{prop:upper_bound}, $x_{(g)}$ for $g\ne \Id$ corresponds to an element in $\oplus_{i>0} H^{2i}(W;R)$. For otherwise, $x_{(g)}$ will be mapped to the generator of $H^0(S^{2n-1}/G;R)$ in the projection $H^*(W;R)\to H^0(S^{2n-1}/G;R)$, which shall not happen by the argument in in \Cref{prop:upper_bound}. We use $y_{(g)}$ to denote the image of $x_{(g)}$ in $H^*(W;\R)$ under the identification above. 
\end{remark}
\begin{remark}\label{rmk:orbifold}
Assume the theory of symplectic cohomology for exact orbifolds is established, i.e.\ in addition to the definition of the group level in \cite{gironella2021exact}, we also need various structures like the ring structure, which will be constructed using polyfolds in a future work. Then \Cref{prop:upper_bound} holds for exact orbifold fillings and $R=\Q$ (this is due to the nature of polyfolds).
\end{remark}

\section{The secondary coproduct on symplectic cohomology}\label{s4}

\subsection{Secondary coproduct and intersection form}
The ordinary coproduct on symplectic cohomology, i.e.\ counting solutions to perturbed Cauchy-Riemann equations on a pair of pants with one positive end and two negative ends, as explained by Seidel in \cite{biased}, is highly degenerate due to degeneration of the Riemann surface (pinching a circle near a negative end). In particular, the coproduct on the positive symplectic cohomology is necessarily trivial. From such vanishing result, considering solutions to perturbed Cauchy-Riemann equations on the family of Riemann surfaces connecting the two pinchings yields the following secondary coproduct on the positive symplectic cohomology, which was introduced by Ekholm and Oancea in \cite{dga},
$$\bm{\lambda}:SH^*_+(W;\bm{k}) \to \left(SH^*_+(W;\bm{k})\otimes SH^*_+(W;\bm{k})\right)[2n-1].$$
Here $\bm{k}$ is some field coefficient\footnote{The requirement of field coefficient is for the sake of the K\"unneth formula.}. 
\begin{enumerate}
    \item A secondary coproduct on $SH^*_+(T^*M)$ was defined by Abbondandolo and Schwarz \cite{MR3289852}. The corresponding coproduct in string topology, through the Viterbo isomorphism, was defined by Goresky and Hingston \cite{MR2560110} in its dual form, also see \cite{MR2079379}.
    \item Another (extended) secondary coproduct appears in the recent work of Cieliebak and Oancea \cite{cieliebak2020multiplicative}. For the case of cotangent bundles, the role of the secondary coproduct and the equivalence between those secondary coproducts were studied in recent works of Cieliebak, Hingston and Oancea \cite{cieliebak2020loop}.
\end{enumerate}
The secondary coproduct can be defined pictorially by counting perturbed Floer equations on the following family of surfaces
 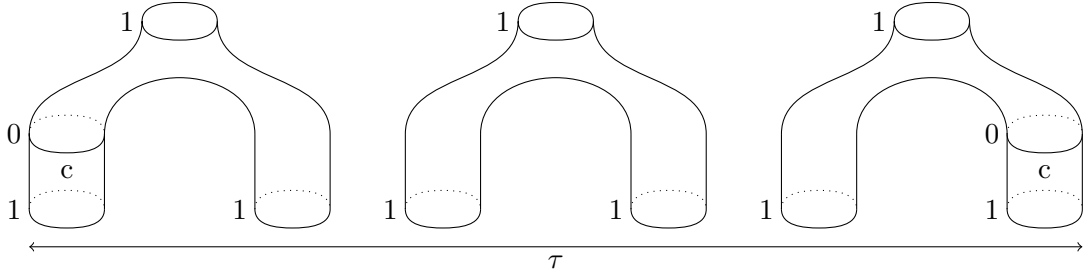
\begin{figure}[H]
	\begin{tikzpicture} 
	\draw[dotted] (0,0) to [out=90, in = 180] (0.5, 0.25) to [out=0, in=90] (1,0);
	\draw (1,0) to [out=270, in=0] (0.5,-0.25) to [out = 180, in=270]  (0,0) to (0,-1);
	\draw[dotted] (0,-1) to  [out=90, in = 180] (0.5, -0.75) to [out=0, in=90] (1,-1);
	\draw (1,-1) to [out=270, in=0] (0.5,-1.25) to [out = 180, in=270] (0,-1);
	\draw (1,0) to (1,-1);
	\draw (0,0) [out=90, in=270] to (1.5,1.5)  to [out=90, in = 180] (2, 1.75) to [out=0, in=90] (2.5,1.5)  [out=270, in=0] to (2,1.25) to [out = 180, in=270] (1.5,1.5);
	\draw (1,0) [out=90, in=180] to (2,0.75)  [out=0, in=90] to (3,0);
	\draw (3,0) to (3,-1) [out=270,in=180] to (3.5,-1.25) [out=0,in=270] to (4,-1) [out=90, in=270] to (4,0) [out=90, in=270] to (2.5,1.5);
	\draw[dotted] (3,-1) [out=90, in=180] to (3.5,-0.75) [out=0,in=90] to (4,-1);
	\node at (1.3,1.5) {$1$};
	\node at (-0.2,0) {$0$};
	\node at (-0.2,-1) {$1$};
	\node at (2.8 ,-1) {$1$};
	\node at (0.5,-0.5) {c};
		
    \draw (5,0) to (5,-1);
	\draw[dotted] (5,-1) to  [out=90, in = 180] (5.5, -0.75) to [out=0, in=90] (6,-1);
	\draw (6,-1) to [out=270, in=0] (5.5,-1.25) to [out = 180, in=270] (5,-1);
	\draw (6,0) to (6,-1);
	\draw (5,0) [out=90, in=270] to (6.5,1.5)  to [out=90, in = 180] (7, 1.75) to [out=0, in=90] (7.5,1.5)  [out=270, in=0] to (7,1.25) to [out = 180, in=270] (6.5,1.5);
	\draw (6,0) [out=90, in=180] to (7,0.75)  [out=0, in=90] to (8,0);
	\draw (8,0) to (8,-1) [out=270,in=180] to (8.5,-1.25) [out=0,in=270] to (9,-1) [out=90, in=270] to (9,0) [out=90, in=270] to (7.5,1.5);
	\draw[dotted] (8,-1) [out=90, in=180] to (8.5,-0.75) [out=0,in=90] to (9,-1);
	\node at (6.3,1.5) {$1$};
	\node at (4.8,-1) {$1$};
	\node at (7.8 ,-1) {$1$};
	
    \draw (10,0) to (10,-1);
	\draw[dotted] (10,-1) to  [out=90, in = 180] (10.5, -0.75) to [out=0, in=90] (11,-1);
	\draw (11,-1) to [out=270, in=0] (10.5,-1.25) to [out = 180, in=270] (10,-1);
	\draw (11,0) to (11,-1);
	\draw (10,0) [out=90, in=270] to (11.5,1.5)  to [out=90, in = 180] (12, 1.75) to [out=0, in=90] (12.5,1.5)  [out=270, in=0] to (12,1.25) to [out = 180, in=270] (11.5,1.5);
	\draw (11,0) [out=90, in=180] to (12,0.75)  [out=0, in=90] to (13,0);
	\draw (13,0) to (13,-1) [out=270,in=180] to (13.5,-1.25) [out=0,in=270] to (14,-1) [out=90, in=270] to (14,0) [out=90, in=270] to (12.5,1.5);
	\draw[dotted] (13,-1) [out=90, in=180] to (13.5,-0.75) [out=0,in=90] to (14,-1);
	\draw (13,0) [out=270,in=180] to (13.5,-0.25) [out=0,in=270] to (14,0);
	\draw[dotted] (13,0) [out=90,in=180] to (13.5,0.25) [out=0,in=90] to (14,0);
	\node at (11.3,1.5) {$1$};
	\node at (12.8,0) {$0$};
	\node at (9.8,-1) {$1$};
	\node at (12.8 ,-1) {$1$};
	\node at (13.5,-0.5) {c};
	\draw[<->] (0,-1.5) to (14,-1.5);
	\node at (7,-1.7) {$\tau$};
	\end{tikzpicture}
	\caption{Secondary coproduct $\bm{\lambda}$, numbers are weights at the puncture}\label{fig:coproduct}
\end{figure}

In a more formal way, the secondary coproduct can be defined as follows \cite[\S 4.1]{cieliebak2020loop}. Let $\Sigma$ be a $3$-punctured Riemann sphere with one positive puncture and two negative punctures. A $1$-form $\beta$ has weight $A$ near a puncture if $\beta = A\rd t$ in the cylindrical coordinate $[0,\infty)_s \times S^1_t$ (or $(-\infty,0]_s \times S^1_t$). Let $H$ be a Hamiltonian as in \S \ref{s3}. We choose a $1$-parameter family of $1$-forms $\beta_{\tau}$ for $\tau\in (0,1)$ with the following properties:
\begin{enumerate}
    \item $\rd\beta_{\tau}\le 0$ for all $\tau$;
    \item $\beta_\tau$ equals $\rd t$ near all three punctures;
    \item as $\tau\to 0$, $\beta_\tau$ equals $0\rd t$ near the first negative puncture whose length goes to infinity, so that $\beta_0$ consists of a $1$-form with weights $1,0,1$ on $\Sigma$ and $0,1$ on a cylinder attached to the first negative puncture; hence ``c" in Figure \ref{fig:coproduct} stands for the continuation map corresponding to the cylinder with weights $0,1$.
    \item as $\tau\to 0$, $\beta_\tau$ equals $0\rd t$ near the second negative puncture whose length goes to infinity, so that $\beta_1$ consists of a $1$-form with weights $1,1,0$ on $\Sigma$ and $0,1$ on a cylinder attached to the second negative puncture.
\end{enumerate}
The algebraic count of rigid solutions to 
$$(\rd u - X_H\otimes \beta_{\tau})^{0,1}=0,\quad \tau\in (0,1)$$
defines a secondary coproduct $C^*_+(H)\to (C^*_+(H)\otimes C^*_+(H))[2n-1]$. It is a cochain map as the extra boundary corresponding to $\tau=0,1$ will not appear since we only consider the positive cochain complex. Passing to the direct limit, we get the secondary coproduct $\bm{\lambda}$ on cohomology.

The significance of the secondary coproduct for our purpose is the following result due to Eliashberg, Ganatra and Lazarev \cite{EGL}. 
\begin{theorem}[\cite{EGL}]\label{thm:EGL}
The following diagram is commutative
$$
\xymatrix{
SH^*_+(W) \ar[rr]^{\bm{\lambda}\hspace{1.5cm}}\ar[d]^{\delta} & &(SH^{*}_+(W)\otimes SH^{*}(W)[2n-1]\ar[d]^{\delta\otimes \delta}\\
H^{*}(W)[1] \ar[rr]^{\Delta\hspace{1.5cm}} & & (H^{*}(W)\otimes H^{*}(W))[2n+1]
}
$$
where $\delta:SH^{*}(W)\to H^{*+1}(W)$ is the connecting map in the tautological exact sequence and the coproduct $\Delta:H^{*}(W)\to (H^{*}(W)\otimes H^{*}(W))[2n]$ is the dual of the cup product $H^{i}_{c}(W)\otimes H^{j}_c(W) \to H^{i+j}_c(W)$.
\end{theorem}
The proof can be explained schematically by the following figures. In the following, we recall a sketch of Eliashberg, Ganatra and Lazarev's proof.
 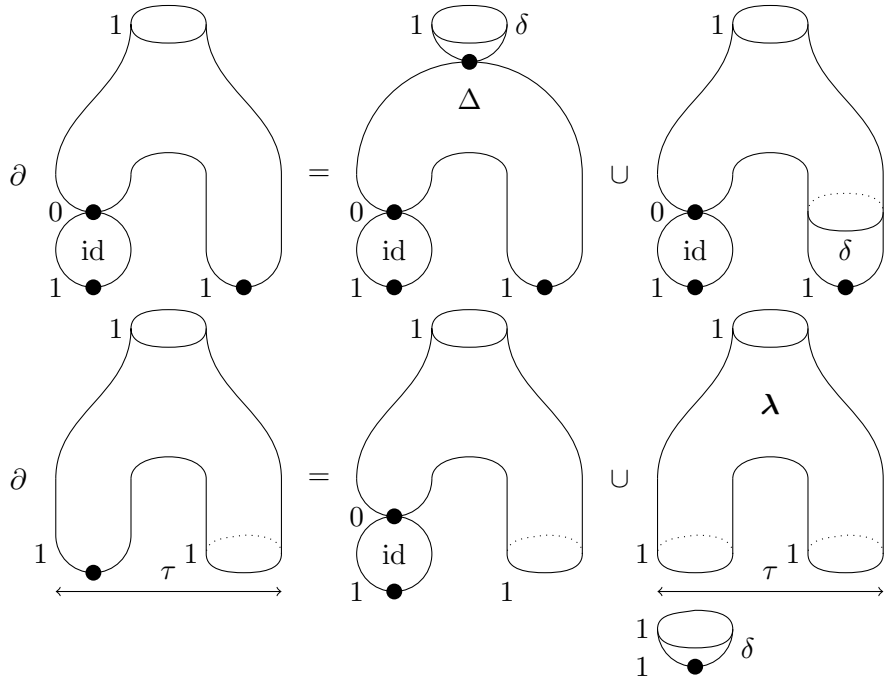
\begin{figure}[H]
	\begin{tikzpicture}
	\node at (-0.5,0) {$\partial$};
	\draw (0,0) [out=270, in=180] to (0.5,-0.5) [out=0, in=270] to (1,0);
	\draw (0.5,-1) circle (0.5);
	\draw (0,0) [out=90, in=270] to (1,2) [out=90,in=180] to (1.5,2.25) [out=0, in=90] to (2,2) [out=270,in=0] to (1.5,1.75) [out=180,in=270] to (1,2);
	\draw (2,2) [out=270, in=90] to (3,0) to (3,-1) [out=270, in=0] to (2.5,-1.5) [out=180,in=270] to (2,-1) [out=90,in=270] to (2,0) [out=90,in=90] to (1,0) ;
	\node at (0.8,2) {$1$};
	\node at (0,-0.5) {$0$};
	\node at (0,-1.5) {$1$};
	\node at (2,-1.5) {$1$};
	\fill (0.5,-0.5) circle[radius=3pt];
	\fill (0.5,-1.5) circle[radius=3pt];
	\fill (2.5,-1.5) circle[radius=3pt];
	\node at (0.5,-1)  {$\Id$};
	
	\node at (3.5,0) {$=$};
	\draw (4,0) [out=270, in=180] to (4.5,-0.5) [out=0, in=270] to (5,0);
	\draw (4.5,-1) circle (0.5);
	\draw (4,0) [out=90, in=180] to (5.5,1.5);
	\draw (5,2) [out=90,in=180] to (5.5,2.25) [out=0, in=90] to (6,2) [out=270,in=0] to (5.5,1.75) [out=180,in=270] to (5,2);
	\draw (5,2) [out=270,in=170] to (5.5,1.5) [out=0,in=270] to (6,2);
	\draw (5.5,1.5) [out=0, in=90] to (7,0) [out=270,in=90] to (7,-1) [out=270, in=0] to (6.5,-1.5) [out=180,in=270] to (6,-1) [out=90,in=270] to (6,0) [out=90,in=90] to (5,0) ;
	\node at (7.5,0) {$\cup$};
	\node at (4.8,2) {$1$};
	\node at (4,-0.5) {$0$};
	\node at (4,-1.5) {$1$};
	\node at (6,-1.5) {$1$};
	\fill (4.5,-0.5) circle[radius=3pt];
	\fill (4.5,-1.5) circle[radius=3pt];
	\fill (6.5,-1.5) circle[radius=3pt];
	\fill (5.5, 1.5) circle[radius=3pt];
	\node at (4.5, -1) {$\Id$};
	\node at (5.5,1) {$\Delta$};
	\node at (6.2,2) {$\delta$};
	
	\draw (8,0) [out=270, in=180] to (8.5,-0.5) [out=0, in=270] to (9,0);
	\draw (8.5,-1) circle (0.5);
	\draw (8,0) [out=90, in=270] to (9,2) [out=90,in=180] to (9.5,2.25) [out=0, in=90] to (10,2) [out=270,in=0] to (9.5,1.75) [out=180,in=270] to (9,2);
	\draw (10,2) [out=270, in=90] to (11,0) to (11,-1) [out=270, in=0] to (10.5,-1.5) [out=180,in=270] to (10,-1) [out=90,in=270] to (10,0) [out=90,in=90] to (9,0);
	\draw[dotted] (10,-0.5) [out=90,in=180] to (10.5,-0.25) [out=0,in=90] to (11, -0.5);
	\draw (10,-0.5) [out=270,in=180] to (10.5,-0.75) [out=0,in=270] to (11, -0.5);
	\node at (8.8,2) {$1$};
	\node at (8,-0.5) {$0$};
	\node at (8,-1.5) {$1$};
	\node at (10,-1.5) {$1$};
	\fill (8.5,-0.5) circle[radius=3pt];
	\fill (8.5,-1.5) circle[radius=3pt];
	\fill (10.5,-1.5) circle[radius=3pt];
	\node at (8.5, -1) {$\Id$};
    \node at (10.5, -1) {$\delta$};
	\end{tikzpicture}

	\begin{tikzpicture}
	\draw (0,0) [out=270, in=180] to (0.5,-0.5) [out=0, in=270] to (1,0);
	\draw (0.5,-1) circle (0.5);
	\draw (0,0) [out=90, in=270] to (1,2) [out=90,in=180] to (1.5,2.25) [out=0, in=90] to (2,2) [out=270,in=0] to (1.5,1.75) [out=180,in=270] to (1,2);
	\draw (2,2) [out=270, in=90] to (3,0) to (3,-1) [out=270, in=0] to (2.5,-1.25) [out=180,in=270] to (2,-1) [out=90,in=270] to (2,0) [out=90,in=90] to (1,0);
	\draw[dotted] (3,-1) [out=90,in=0] to (2.5,-0.75) [out=190,in=90] to (2,-1);
	\node at (0.8,2) {$1$};
	\node at (0,-0.5) {$0$};
	\node at (0,-1.5) {$1$};
	\node at (2,-1.5) {$1$};
	\fill (0.5,-0.5) circle[radius=3pt];
	\fill (0.5,-1.5) circle[radius=3pt];
	\node at (0.5,-1)  {$\Id$};
	\node at (3.5,0) {$\cup$};
	
	\draw (4,0) [out=90, in=270] to (5,2) [out=90,in=180] to (5.5,2.25) [out=0, in=90] to (6,2) [out=270,in=0] to (5.5,1.75) [out=180,in=270] to (5,2);
	\draw (6,2) [out=270, in=90] to (7,0) to (7,-1) [out=270, in=0] to (6.5,-1.25) [out=180,in=270] to (6,-1) [out=90,in=270] to (6,0) [out=90,in=90] to (5,0);
	\draw[dotted] (7,-1) [out=90,in=0] to (6.5,-0.75) [out=190,in=90] to (6,-1);
	\draw (4,0) to (4,-1) [out=270,in=180] to (4.5,-1.25) [out=0,in=270] to (5,-1) [out=90,in=270] to (5,0);
	\draw[dotted] (5,-1) [out=90,in=0] to (4.5,-0.75) [out=190,in=90] to (4,-1);
	\draw[<->] (4,-1.5) to (7,-1.5);
	\node at (5.5, -1.25) {$\tau$};
	\node at (5.5,1) {$\bm{\lambda}$};
	
	\draw (5,-2) [out=90,in=0] to (4.5,-1.75) [out=190,in=90] to (4,-2) [out=270,in=180] to (4.5,-2.25) [out=0,in=270] to (5,-2);
	\draw (5,-2) [out=270, in=0] to (4.5,-2.5) [out=180,in=270] to (4,-2);
	
	\node at (4.8,2) {$1$};
	\node at (3.8,-1) {$1$};
	\node at (5.8,-1) {$1$};
	\node at (3.8,-2) {$1$};
	\node at (3.8,-2.5) {$1$};
	\node at (5.2,-2.25) {$\delta$};
	\fill (4.5,-2.5) circle[radius=3pt];
	
	\draw (-4,0) [out=90, in=270] to (-3,2) [out=90,in=180] to (-2.5,2.25) [out=0, in=90] to (-2,2) [out=270,in=0] to (-2.5,1.75) [out=180,in=270] to (-3,2);
	\draw (-2,2) [out=270, in=90] to (-1,0) to (-1,-1) [out=270, in=0] to (-1.5,-1.25) [out=180,in=270] to (-2,-1) [out=90,in=270] to (-2,0) [out=90,in=90] to (-3,0);
	\draw[dotted] (-1,-1) [out=90,in=0] to (-1.5,-0.75) [out=190,in=90] to (-2,-1);
	\draw (-4,0) to (-4,-0.75) [out=270,in=180] to (-3.5,-1.25) [out=0,in=270] to (-3,-0.75) [out=90,in=270] to (-3,0);
	\draw[<->] (-4,-1.5) to (-1,-1.5);
	\node at (-2.5, -1.25) {$\tau$};
	\fill (-3.5,-1.25) circle[radius=3pt];
	\node at (-3.2,2) {$1$};
	\node at (-4.2,-1) {$1$};
	\node at (-2.2,-1) {$1$};
	\node at (-0.5,0) {$=$};
	\node at (-4.5,0) {$\partial$};
	\end{tikzpicture} 
	\caption{Pictorial proof of \Cref{thm:EGL}}\label{fig:EGL}
\end{figure}
The numbers stand for weights as before. By $\bullet$, we mean that the end is asymptotic to a constant orbit. By a circle, we mean the end is asymptotic to a non-constant orbit, i.e.\ a generator of the positive cochain complex. The configuration equipped with $\tau$ is the moduli space in Figure \ref{fig:coproduct}. The configuration marked with $\Delta$ defines the coproduct on $H^*(W)$. The boundary components correspond to the differential on the positive cochain complex and zero cochain complex (Morse cochain complex) are omitted, which will correspond to the homotopy in the commutative diagram in \Cref{thm:EGL}. More precisely, we insert $x$ that is closed in the positive cochain complex (i.e.\ a representative of a class in $SH^*_+(W)$) as the input in the top row of Figure \ref{fig:EGL}, then we can get 
\begin{equation}\label{eqn:r1}
    (\mathrm{c} \otimes \Id)\circ \Delta \circ \delta\pm (\mathrm{c}\otimes \delta)\circ \Theta = 0\in H^*(W)\otimes H^*(W) 
\end{equation}
Here $\Theta:C_+(1\cdot H)\to C(0\cdot H )\otimes C_+(1\cdot H)$ is defined by counting the top level of the top right curve in Figure \ref{fig:EGL}, $\mathrm{c}$ is the continuation map $C(0\cdot H) \to C(1\cdot H)$, i.e.\ identity onto the constant orbits part of $C(H)$. \eqref{eqn:r1} only holds on cohomology, as we omit the Morse breakings in a cascades model or breakings appearing below a constant orbit in a $C^2$-small perturbed model. Now we look at the second row of Figure \ref{fig:EGL},  the first configuration in the union is the left end of Figure \ref{fig:coproduct}. A priori, there is another boundary from pinching the second negative end in $\tau$-family, but it will not appear since the asymptotic of the second negative puncture is a non-constant orbit in Figure \ref{fig:EGL}. Therefore, the second row of Figure \ref{fig:EGL} reads
\begin{equation}\label{eqn:r2}
    (\mathrm{c}\otimes \Id)\circ \Theta\pm (\delta\otimes \Id)\circ \bm{\lambda}=0\in H^*(W)\otimes SH_+^*(W).
\end{equation}
As before, this relation holds on cohomology only as we omitted other more standard boundaries. Combining \eqref{eqn:r1} and \eqref{eqn:r2}, we see that $\Delta\circ \delta=(\delta\otimes \delta)\circ \bm{\lambda}$ after choosing appropriate signs and orientation conventions.

\subsection{The secondary coproduct for $\C^n/G$}

We have the following special case of the orbifold version of \Cref{thm:EGL}.
\begin{proposition}\label{prop:orb_EGL}
The secondary coproduct is well-defined on $SH^*_+(\C^n/G;\bm{k})$. When $|G|$ is invertible in $\bm{k}$, the following is commutative,
$$
\xymatrix{
SH^*_+(\C^n/G;\bm{k}) \ar[rr]^{\bm{\lambda}\hspace{1.5cm}}\ar[d]^{\delta} & &(SH^{*}_+(\C^n/G;\bm{k})\otimes SH^{*}(\C^n/G;\bm{k}))[2n-1]\ar[d]^{\delta\otimes \delta}\\
H^{*}_{\CR}(\C^n/G;\bm{k})[1] \ar[rr]^{\Delta \hspace{1.5cm}} & & (H^{*}_{\CR}(\C^n/G;\bm{k})\otimes H^{*}_{\CR}(\C^n/G;\bm{k}))[2n+1]
}
$$
where the bottom row is the Chen-Ruan coproduct in \S \ref{ss:coproduct} and the grading is the rational grading. 
\end{proposition}
\begin{proof}
In view of \cite[\S 3.7.2]{gironella2021exact}, the difficulty of generalizing \Cref{thm:EGL} to the orbifold setting is the appearance of constant holomorphic curves with more than $3$ punctures of negative virtual dimension. Those curves are not somewhere injective and can not be perturbed away by the enforcement of the symmetry, hence we need the machinery of virtual perturbations to deal with them. However, in the special case of $\C^n/G$, such a problem will not appear and the proof sketched above for \Cref{thm:EGL} goes through. 

First note that the secondary coproduct is defined on $SH^*_+(W;\bm{k})$ for a general exact orbifold filling $W$, as the relevant moduli spaces do not involve constant curves in the compactification. Now by \Cref{thm:quotient_SH} (\cite[Theorem B]{gironella2021exact}), $SH^*_+(\C^n/G;\bm{k})$ is spanned by classes represented by $x_{(g)}$ and is equipped with a global $\Q$-grading. Moreover, $x_{(g)}$ is sent to the constant twisted sector corresponding to the conjugacy class $(g)$ in the connecting map  $SH^*_+(\C^n/G;\bm{k})\to H_{\CR}^{*+1}(\C^n/G;\bm{k})$. With a slight abuse of notation, we use $(g)$ to denote the constant twisted sector. Therefore it suffices to prove the diagram on those generators (instead of the whole cochain complex). In the moduli spaces in Figure \ref{fig:EGL}, the only place to see constant holomorphic curves with more than $2$ punctures is from the right side of the first row. Note that $x_{(g)}$ has rational grading $2\age(g)-1$, hence $\bm{\lambda}(x_{(g)})$ has rational grading $2\age(g)+2n-2$ and $(\delta\otimes \delta)\circ \bm{\lambda} (x_{(g)})$ has rational grading $2\age(g)+2n$. Therefore to prove the commutative diagram, it is sufficient to plug in $x_{(g)}$ into the input end and constant twisted sectors $(h_1),(h_2)$ into the output ends of the top row of Figure \ref{fig:EGL}, subjected to the following two constraints:
\begin{enumerate}
    \item $(h_1h_2)=(g)\in Conj(G)$, this is the constraint from the homotopy classes.
    \item $\age(h_1)+\age(h_2)=\age(g)+n$, this is the constraint from the rational grading above.
\end{enumerate}
With such asymptotics, in the top row of Figure \ref{fig:EGL}, the breaking from the input end must end in $(g)$ by the homotopy class. Therefore the constant thrice-punctured curve must have asymptotics $(g),(h_1),(h_2)$, which, due to the second constraint on $\age$, is cut out transversely\footnote{The kernel of the linearized operator is always zero-dimensional for any constant curve. Since $2n+2\age(g)-2\age(h_1)-2\age(h_2)$ is the Fredholm index, $\age(h_1)+\age(h_2)=\age(g)+n$ implies transversality.}. Therefore the proof of \Cref{thm:EGL} goes through for $\C^n/G$. 
\end{proof}
\begin{remark}
    To prove \Cref{thm:EGL} for general exact orbifolds with contact manifold boundary,  we must check the diagram on general generators instead of the special collection $\{x_{(g)}\}$ in the special case of \Cref{prop:orb_EGL}.  Then in the top row of Figure \ref{fig:EGL}, we may run into a situation where the top curve for the operator $\delta$ has a moduli space of positive dimension, and the constant curve for $\Delta$ has a negative virtual dimension, with a correct total virtual dimension. Such a configuration can not resolved by classical perturbations. 
\end{remark}

We endow the tensor product $C_+(H)\otimes C_+(H)$ with a filtration
$$F^kC_+(H)\otimes C_+(H) = \left\langle x\otimes y\left| |x|^{\partial}+|y|^{\partial}-2n+1\ge \frac{k}{|G|}\right. \right\rangle.$$
By the same argument of \Cref{prop:SS}, the induced spectral sequence also degenerates at $(|G|+1)$th page, which is spanned by $\{ x_{(g)}\otimes x_{(h)}\}_{(g),(h)\in \Conj(G)}$, whenever $\C^n/G$ is terminal and $|G|$ is invertible in $\bm{k}$. The extra degree shift by $-2n+1$ is for $\bm{\lambda}$ capturing the non-trivial leading terms (i.e.\ on the associate graded) in the proposition below.
\begin{proposition}\label{prop:intersection_SS}
Assume $\C^n/G$ is terminal. Let $W$ be an exact filling of $(S^{2n-1}/G;\xi_{\std})$. The secondary coproduct on $SH^*_+(W;\bm{k})$ is compatible with the boundary grading filtration. In particular, it induces a map on the associated graded, which is isomorphic to the secondary coproduct on $SH^*_+(\C^n/G;\bm{k})$.
\end{proposition}
\begin{proof}
Since $\md(\C^n/G)>0$, the neck-stretching argument shows that $\la \bm{\lambda}(x),y\otimes z \ra = 0$ whenever $|x|^{\partial}+2n-1-|y|^{\partial}-|z|^{\partial}<0$. Note that the boundary rational gradings are multiplies of $\frac{1}{|G|}$, this implies that $\bm{\lambda}$ is compatible with filtration. The induced map of $\bm{\lambda}$ on the first page of the spectral sequence only depends on moduli spaces defining $\bm{\lambda}$ with input asymptotic $x$ and output asymptotic $y,z$ such that $|x|^{\partial}+2n-1-|y|^{\partial}-|z|^{\partial}=0$.  Neck-stretching and $\md(\C^n/G)>0$ imply that this map is independent of fillings. Note that the induced map of $\bm{\lambda}$ on the associated graded is induced from the induced map on the first page. Since for $\C^n/G$, the filtration is the rational grading and the induced map on the associated graded is identified with the secondary coproduct $\bm{\lambda}$ on $\C^n/G$ (determined by \Cref{prop:orb_EGL} and \eqref{eqn:coproduct}). The claim for a general filling follows from the independence of filling and a comparison with $\C^n/G$.
\end{proof}

Recall that the cup product length of a topological space $W$ is defined to be
$$\mathrm{cl}(W):=\max \left\{ n\left|\exists \alpha_i \in \oplus_{*>0}H^*(W;\Q), i\le n, \text{ such that } \cup \alpha_i\ne 0 \right.\right\}.$$
A consequence of \Cref{prop:intersection_SS} is the following upper bound.
\begin{corollary}\label{cor:cuplength}
Assume $\C^n/G$ is terminal. Let $W$ be an exact filling of $(S^{2n-1}/G;\xi_{\std})$. Then we have 
$$\mathrm{cl}(W)\le \left \lfloor \frac{n-1-\md(\C^n/G)}{1+\md(\C^n/G)}\right \rfloor.$$
In particular, we always have $\mathrm{cl}(W)\le n-2$.
\end{corollary}
\begin{proof}
By \Cref{prop:upper_bound} and \Cref{rmk:grading}, we can write $H^*(W;\Q)\simeq \oplus_{(g)\in \Conj(G)} \la y_{(g)}\ra$, where $y_{(g)}$ is the image of $x_{(g)}$ under a non-canonical isomorphism $ \oplus_{(g)\in \Conj(G)} \la x_{(g)} \ra \simeq SH_+^*(W;\bm{k})\simeq H^{*+1}(W;\bm{k})$. We grade $y_{(g)}$ by the rational number $2\age(g)$. Under such isomorphism, \Cref{prop:intersection_SS} implies that the coproduct $\Delta(y_{(g)})$ has a minimal rational grading at least $2n+2\age(g)$. In view of the Lefschetz duality $H^*_c(W;\Q)=\Hom(H^{2n-*}(W;\Q),\Q)$, we may write $H^*_c(W;\Q)\simeq \oplus_{(g)\in \Conj(G)} \la y^{\vee}_{(g)}\ra$, where $y^{\vee}_{(g)}$ is the linear dual of $y_{(g)}$ in that basis. We define a rational grading of $y^{\vee}_{(g)}$ to be $2n-2\age(g)$ (i.e.\ $2\age(g^{-1})$). Under such an isomorphism, the dual of \Cref{prop:intersection_SS} implies that the cup product $\cup :H^*_c(W;\Q)\otimes H^*_c(W;\Q) \to H^*_c(W;\Q)$ has the property that the minimal rational grading of $y^{\vee}_{(g)}\cup y^{\vee}_{(h)}$ is at least $4n-2\age(g)-2\age(h)$. 

Since $S^{2n-1}/G$ is a rational homology sphere, we have that $H^*_c(W;\Q)\to H^*(W;\Q)$ is an isomorphism for $0<*<2n$ and the cup product length of $W$ is precisely the cup product length of the monoid $\Ima(H^*_c(W;\Q)\to H^*(W;\Q))$ or equivalently $H^*_c(W;\Q)/H^{2n}_c(W;\Q)$. By \Cref{rmk:grading}, $y_{(g)}$ in the isomorphism $H^*(W;\Q)\simeq \oplus_{(g)\in \Conj(G)} \la y{(g)}\ra$ does not have component in $H^0(W;\Q)$ when $g\ne \Id$, while $y_{(\Id)}$ does. As a consequence, the cup length product of $H^*_c(W;\Q)/H^{2n}_c(W;\Q)$ is same as the cup product length on $\oplus_{(g)\ne (\Id)\in \Conj(G)} \la y^{\vee}_{(g)}\ra$. The cup product length of the latter can be estimated by the rational grading. Since the minimal and maximal rational gradings of $y^{\vee}_{(g)}$ for $g\ne \Id$ are $2(1+\md(\C^n/G))$ and $2n-2(1+\md(\C^n/G))$ respectively. Hence the claim follows, as the product does not decrease the rational grading.
\end{proof}
By \Cref{prop:upper_bound}, we have that $H^*(W)$ is supported in even degrees, hence the cup product length is tautologically smaller than $n$. Although \Cref{cor:cuplength} only improves the upper bound by $1$ if we only know that $\md(\C^n/G)>0$, it is crucial in the proof of \Cref{thm:U(n)}.

In some cases, the map on the associated graded can be identified with the map on the original space. Then \Cref{prop:intersection_SS} can recover the coproduct on cohomology. The following is the simplest case when it happens. 
\begin{corollary}\label{cor:intersection_SS}
Let $\C^n/G$ be a terminal singularity of the following cases, then any exact filling of $(S^{2n-1}/G,\xi_{\std})$ has the same coproduct as $\C^n/G$ rationally. 
\begin{enumerate}
    \item Diagonal actions on $\C^{n=2m}$ by finite subgroups of $SU(2)$.
    \item $G\subset O(n)\subset SU(n)$ with isolated singularity.
\end{enumerate}
\end{corollary}
\begin{proof}
The missed information of the induced map on the associated graded is the component in the higher filtration of the image. In both cases, we have $\age(g)=n/2$ for any $g\ne \Id$, then there is no higher filtration for $\Delta(y_{(\Id)})$ or $\Delta(y_{(g)})$. 
\end{proof}
Since the meaning of higher filtration depends on the stride length ($\frac{1}{|G|}$) of the filtration in \eqref{eqn:filtration}. Therefore, one can generalize \Cref{cor:intersection_SS} to other terminal singularities as long as $\age(g)$ is concentrated around $n/2$ for $g\ne \Id$. Moreover, one can adjust the stride length so that results in \S \ref{s3} and \Cref{prop:intersection_SS} still hold if we know more about $\md(\C^n/G)$, which will lead to different requirement of the concentration. We will not pursue the precise statements here.

	
\section{Fundamental groups}\label{s5}
In this section, we discuss fundamental groups for \emph{strong fillings} of contact links of quotient singularities. This requires the definition of (positive) symplectic cohomology for strong fillings. The existence of positive symplectic cohomology and tautological long exact sequence is guaranteed by the asymptotic behaviour lemma \cite[Lemma 2.3]{ES} instead of the usual action separation. The transversality issues from sphere bubbles require the deployment of virtual techniques, for example, Pardon's setup for Hamiltonian-Floer cohomology for a general closed symplectic manifold \cite{pardon} suffices for the setup for symplectic cohomology used in this paper. If one wishes to stay within the realm of geometric perturbations, then strong fillings should be replaced by semi-positive fillings \cite[Definition 6.4.1]{bible} in the following results. 

\begin{proposition}\label{prop:fundamental}
Let $W$ be a strong filling of $(S^{2n-1}/G,\xi_{\std})$. We write $G_0$ as the kernel of $\pi_1(S^{2n-1}/G)\to \pi_1(W)$.  If $\C^n/G_0$ is terminal, then $\pi_1(S^{2n-1}/G)\to \pi_1(W)$ is surjective.
\end{proposition}
\begin{proof}
Assume $\pi_1(S^{2n-1}/G)\to \pi_1(W)$ is not surjective. We write $Q=\pi_1(W)/\Ima(\pi_1(S^{2n-1}/G)\to \pi_1(W))\ne 0$. We can consider the symplectic cohomology for the universal cover $\widetilde{W}$ as in \cite[\S 3.3]{ADC}. Under our assumption, $\partial \widetilde{W}$ is $\# Q$ many copies of $(S^{2n-1}/G_0,\xi_{\std})$. Similarly to \Cref{prop:SS}, by comparing to the natural filling $\bigsqcup_{Q} \C^n/G_0$ of $\partial \widetilde{W}$, we have that $SH^*_+(\widetilde{W};\Lambda) \simeq \prod_{Q} H^{*+1}_{\CR}(\C^n/G_0;\Lambda)$ where $\Lambda$ is the Novikov field. Moreover, by the same proof of \Cref{prop:upper_bound}, the composition $SH_+^*(\widetilde{W};\Lambda)\to H^{*+1}(\widetilde{W};\Lambda)\to H^{*+1}(\partial\widetilde{W};\Lambda)\to H^0(\partial \widetilde{W})\otimes \Lambda$ is
$$SH^*_+(\widetilde{W};\Lambda) \stackrel{\simeq}{\to} \prod_{Q} H^{*+1}_{\CR}(\C^n/G_0;\Lambda) \twoheadrightarrow \prod_{Q} H^{0}_{\CR}(\C^n/G_0)\otimes\Lambda=  H^0(\partial \widetilde{W})\otimes \Lambda = \prod_Q \Lambda$$
hence surjective. On the other hand, the map above factors through $H^0(\widetilde{W})\otimes \Lambda \to  H^0(\partial \widetilde{W})\otimes \Lambda$, which is clearly not surjective, contradiction.
\end{proof}
In some sense, the argument above hints that $(S^{2n-1}/G_0,\xi_{\std})$ should not be the boundary of a connected strong domain with multiple boundary components, i.e.\ not co-fillable, see \Cref{rmk:cofill} for explanations. 

If we only consider exact fillings, the un-deformed ring structure on quantum cohomology allows us to say the following without the terminal assumption.
\begin{proposition}\label{prop:injective}
Let $W$ be an exact filling of $(S^{2n-1}/G,\xi_{\std})$ for $n\ge 2$. Then $\pi_1(S^{2n-1}/G)\to \pi_1(W)$ is not injective.
\end{proposition}
\begin{proof}
Assume $\pi_1(S^{2n-1}/G)\to \pi_1(W)$ is injective. In the neck-stretching process, those negative punctures of the top-level curve must be asymptotic to contractible Reeb orbits on $S^{2n-1}/G$ by the injectivity assumption, as those orbits are eventually capped off by holomorphic planes in $W$. Note that all contractible orbits have positive integer SFT degrees, as they correspond to Reeb orbits on $(S^{2n-1},\xi_{\std})$ with the same Conley-Zehnder index. Similar to \Cref{prop:SS}, the first $|G|+1$ pages of the spectral sequence are independent of such $W$. Note that $\pi_1(S^{2n-1}/G)\to \pi^{\orb}_1(\C^n/G)$ is an isomorphism, we can compare the hypothetical $W$ with the orbifold $\C^n/G$ and the arguments of \Cref{prop:SS,prop:upper_bound} go through and the same conclusions hold for $W$ without the terminal assumption. That is we have
\begin{enumerate}
    \item the $(|G|+1)$th page of $SH_+^*(W;\Q)$ is generated by $x_{(g)}$ (in \Cref{thm:quotient_SH}), and the spectral sequence degenerates there;
    \item $SH_+^*(W;\Q)\to H^{*+1}(W;\Q)$ is an isomorphism. 
\end{enumerate}
Therefore elements in  $SH_+^*(W;\Q)$ are represented by $x_{(g)}+a$ for $(g)\in \Conj(G)$ and $a$ has a lower filtration order compared to $x_{(g)}$. Here the homotopy class of $a$ must be the same as $x_{(g)}$ as  $\pi_1(S^{2n-1}/G)\to \pi_1(W)$ is injective, which is non-trivial when $(g)\ne (\Id)$. This contradicts the fact that $SH_+^*(W;\Q)\to H^{*+1}(W;\Q)$ is an isomorphism,  as $x_{(g)}+a$ must map to zero in $H^{*+1}(W;\Q)$ by the restriction of homotopy classes.
\end{proof}

\begin{corollary}\label{cor:simply_connected}
Let $\C^n/G$ be a terminal singularity, then any exact filling of $(S^{2n-1}/G,\xi_{\std})$ is simply connected.
\end{corollary}
\begin{proof}
Let $W$ be an exact filling. Since $\C^n/G$ is terminal, by \Cref{prop:terminal}, $\C^n/G_0$ must be terminal as well, where $G_0=\ker(\pi_1(S^{2n-1}/G)\to \pi_1(W))$. Therefore \Cref{prop:fundamental} implies that $\pi_1(S^{2n-1}/G)\to \pi_1(W)$ is surjective. If $\pi_1(W)\ne 0$, we must have a non-trivial conjugacy class $(g)$, such that $x_{(g)}$ is not contractible in $W$. However, by \Cref{{prop:upper_bound}}, the element in $SH_+^*(W;\Q)$ represented by $x_{(g)}$ plus lower order terms must map to a non-trivial class in $H^*(W;\Q)$, contradicting the fact that $x_{(g)}$ is not contractible in $W$.
\end{proof}

\begin{proposition}\label{prop:simply_connected}
Let $m\in \N$ such that there is no non-trivial square that divides $m$. Assume $\C^n/(\Z/m)$ is terminal, then any strong filling of $(S^{2n-1}/(\Z/m),\xi_{\std})$ is simply connected.
\end{proposition}
\begin{proof}
Let $W$ be a strong filling. By \Cref{prop:fundamental}, we have that $\Z/m=\pi_1(S^{2n-1}/(\Z/m))\to \pi_1(W)$ is surjective. If $\pi_1(W)\ne 0$, then $\pi_1(W)=\Z/p$, where $p$ divides $m$. Therefore, by the universal coefficient theorem, we have $H^1(W;\Z/m)\to H^1(S^{2n-1}/(\Z/m);\Z/m)$ and  $H^2(W;\Z/m)\to H^2(S^{2n-1}/(\Z/m);\Z/m)$  are both non-trivial and contain $\Z/p\subset \Z/m$ in the image.  From \Cref{rmk:connecting}, we have that $H^*(B\Z/m;\Z/m)=H^*(S^{2n-1}/(\Z/m);\Z/m)$ for $*\le 2n-1$ and the projection is an algebra map.  Since $H^*(B\Z/m;\Z/m)=(\Z/m)[a,b]$ as an algebra for $\deg(a)=1$ and $\deg(b)=2$ and $m$ is not divided by $p^2$, we get $\alpha\in H^1(W;\Z/m), \beta \in H^2(W;\Z/m)$, such that $(\alpha \cup \beta^{n-1})|_{\partial W} \ne 0$, i.e.\ $H^{2n-1}(W;\Z/m)\to H^{2n-1}(S^{2n-1}/(\Z/m);\Z/m)$ is nontrivial. Hence  $H^{2n}(W,S^{2n-1}/(\Z/m);\Z/m)$ is a non-trivial quotient of $\Z/m$ from the long exact sequence, this is impossible when $W$ is orientable.
\end{proof}

\Cref{thm:A} is assembled from \Cref{prop:fundamental,prop:injective,prop:simply_connected} and \Cref{cor:simply_connected}

\begin{remark}
The $m=2$ case of \Cref{prop:simply_connected} was proven by Ghiggini and Niederkr\"{u}ger \cite[Theorem 1.2]{MR4406869}. Unlike the $m=2$ case, when $m\ge 3$, there are multiple $\Z/m$ actions such that $\C^n/(\Z/m)$ is terminal. There are also many non-terminal examples $\C^2/(\Z/m)$, whose contact links admit Weinstein fillings that are not simply connected, see e.g.\ \cite{lens}.
\end{remark}

\begin{remark}[Proposition \ref{prop:B}]\label{rmk:cofill}
If $\C^n/G$ is a terminal singularity, the arguments in \S \ref{s3} imply that there is always a holomorphic curve passing through a point constraint independent of the point and strong fillings. We claim that $(S^{2n-1}/G,\xi_{\std})$ is not co-fillable. Assume otherwise,  $(S^{2n-1}/G,\xi_{\std})\sqcup Y$ can be filled by a connected strong domain $W$. We can set up our Hamiltonian to be standard on the cylindrical end of $(S^{2n-1}/G,\xi_{\std})$ as in \S \ref{s3} and $0$ on the cylindrical end on $Y$. By moving the point constraint to the cylindrical end of $Y$, we then obtain a contradiction with the maximum principle. However, we  expect that $(S^{2n-1}/G,\xi_{\std})$ is not co-fillable without any constraint on the isolated singularity $\C^n/G$. In dimension $3$, by \cite{MR2371183}, $(S^{3}/(\Z/m),\xi_{\std})$ supports a planar open book, hence is not co-fillable by \cite{MR2126827}. In higher dimensions, it is interesting to see whether $(S^{2n-1}/G,\xi_{\std})$ supports an iterated planar open book \cite{MR4391886}. If the answer to this question is affirmative, then $(S^{2n-1}/G,\xi_{\std})$ is not co-fillable by \cite{MR4391886}. However, one should not expect the (iterated) planar property to hold in general, as there are already counterexamples in dimension $3$ \cite{MR2597306}. On the other hand, from the symplectic cohomology point of view in this paper, when the terminality condition is removed, we expect fillings of  $(S^{2n-1}/G,\xi_{\std})$ to have $k$-semi-dilations \cite{zhou2019symplectic}, see \cite{2D} for the special case in complex dimension $2$. The existence of $k$-semi-dilations, which is also an existence of holomorphic curves with a point constraint,   leads to obstructions to co-fillablity by a similar argument. 
\end{remark}

\section{Non-existence of exact fillings for real projective spaces}\label{s6}
The main result in this section is 
\begin{theorem}\label{thm:2_power}
For $n=2^k\ge 4$, $(\RP^{2n-1},\xi_{\std})$ is not exactly fillable.
\end{theorem}

Combining with \cite[Theorem 1.1]{RP}, we have
\begin{theorem}[\Cref{thm:RP}]
For any $n\ge 3$, $(\RP^{2n-1},\xi_{\std})$ is not exactly fillable.
\end{theorem}

\subsection{Properties of a hypothetical exact filling}\label{ss:property}
Assume $W$ is a hypothetical exact filling of $(\RP^{2n-1},\xi_{\std})$ for $n=2^{k}>2$. We first list the properties of $(\RP^{2n-1},\xi_{\std})$ and $W$ from previous sections.
\begin{enumerate}
    \item By \cite[Remark 2.3]{RP}, the total Chern class of $(\RP^{2n-1},\xi_{\std})$ is trivial.
    \item By \Cref{cor:simply_connected}, we have $\pi_1(W)=0$. As a consequence, we have $H_1(W;\Z)=\Tor H^2(W;\Z)=0$.
    \item By \Cref{cor:intersection_SS}, we have $H^*(W;\Q)=\Q\oplus \Q[-2^k]$ and $H^*_c(W;\Q)=\Q[-2^k]\oplus \Q[-2^{k+1}]$ with a positive intersection form. Therefore we have $H^2(W;\Z)=0$ and hence $c_1(W)=0$.
    \item Let $C$ be the disk bundle of the  degree $2$ line bundle over $\CP^{n-1}$, which is the natural symplectic cap of $(\RP^{2n-1},\xi_{\std})$. Then we can glue $C$ to $W$ along $\RP^{2n-1}$ to obtain a closed symplectic manifold $X$. Using the rational cohomology, it is easy to see that Euler characteristic $\chi(X)=n+2$ and signature $\sigma(X)=2$.
\end{enumerate}

\begin{proposition}\label{prop:middle}
We have that $\Tor H^{n}(W;\Z)\to H^{n}(\RP^{2n-1};\Z)$ is zero but $ H^{n}(W;\Z)\to H^{n}(\RP^{2n-1};\Z)$ is surjective.
\end{proposition}
\begin{proof}
Note that we have the following long exact sequence
\begin{equation}\label{eqn:exact_RP}
    0\to H^{n}(W,\RP^{2n-1};\Z)\to H^{n}(W;\Z)\to H^{n}(\RP^{2n-1};\Z) \to  H^{n+1}(W,\RP^{2n-1};\Z) \to  H^{n+1}(W;\Z)\to 0.
\end{equation}
By Lefschetz duality, \eqref{eqn:exact_RP} can be rewritten as 
\begin{equation}\label{eqn:exact_RP'}
    0 \to H_{n}(W;\Z)\to H^{n}(W;\Z)\to \Z/2 \to H_{n-1}(W;\Z)\to H^{n+1}(W;\Z)\to 0.
\end{equation}
By the universal coefficient theorem, we have
$$H_{n}(W;\Z)=\Z\oplus  H^{n+1}(W;\Z),\quad H_{n-1}(W;\Z)=\Tor H^{n}(W;\Z).$$
Hence \eqref{eqn:exact_RP'} can be written as 
\begin{equation}\label{eqn:exact_RP''}
    0 \to \Z\oplus  H^{n+1}(W;\Z)\to \Z\oplus \Tor H^{n}(W;\Z)\to \Z/2 \to \Tor H^{n}(W;\Z) \to H^{n+1}(W;\Z)\to 0.
\end{equation}
If $\Tor H^{n}(W;\Z)\to H^{n}(\RP^{2n-1};\Z)= \Z/2 $ is nonzero, then we have $\Tor H^{n}(W;\Z) = H^{n+1}(W;\Z)$. On the other hand, \eqref{eqn:exact_RP''} implies that the torsion group $H^{n+1}(W;\Z)$ injects into $\Tor H^{n}(W;\Z)$. Hence we have $|\Tor H^{n}(W;\Z)|\ge 2|H^{n+1}(W;\Z)|$, which is a contradiction. Hence $\Tor H^{n}(W;\Z)\to H^{n}(\RP^{2n-1};\Z)= \Z/2 $ is zero.

On the other hand, if $ H^{n}(W;\Z)\to H^{n}(\RP^{2n-1};\Z)$ is zero. We have $|\Tor H^{n}(W;\Z)|= 2|H^{n+1}(W;\Z)|$ from the right three terms of \eqref{eqn:exact_RP''}, while $|\Tor H^{n}(W;\Z)|= |H^{n+1}(W;\Z)|$ from the left two terms of \eqref{eqn:exact_RP''}, contradiction.
\end{proof}

Note that the rank of $H^{n}(X;\Z)$ is two. Let $\mu$ denote a generator of the free part of $H^n(W;\Z)$\footnote{Strictly speaking, there is no canonical free part as $0\to\Tor H^n(W;\Z)\to H^n(W;\Z)\to \Z \to 0$ does not have a natural splitting. But since we only discuss the intersection number, such imprecision does not matter.}. In view of the Mayer-Vietoris  sequence for $X=W\cup C$, the free part of $H^{n}(X;\Z)$ is generated by $\mu+u^{n/2}$ and $2u^{n/2}$, where $u\in H^2(C;\Z)=H^2(\CP^{n-1};\Z)$ is the generator represented by the first Chern class of $\cO(1)$. On the other hand, since $H^*(\RP^{2n-1};\Q)=H^*(S^{2n-1};\Q)$, we have $H^n(X;\Q)=H^n(C;\Q)\oplus H^n(W;\Q)=H^n_c(C;\Q)\oplus H^n_c(W;\Q)$, the later is from the $\Q$-isomorphism $c_W:H_c^n(W;\Q)\to H^n(W;\Q), c_C:H_c^n(C;\Q)\to H^n(C;\Q)$ (which is non-trivial on the integer lattice). Therefore it makes sense to write $\int_X (u^{n/2})^2 := \int_X 0\oplus (u^{n/2})^2=\int_C (c_C^{-1}(u^{n/2}))^2=\frac{1}{2}$ as $c_C$ is multiplying by $2$ on the integer lattice. Similarly, we write $\int_X \mu^2$ to mean $\int_X \mu_2^2\oplus 0=\int_W c_W^{-1}(\mu)^2$. 
\begin{proposition}\label{prop:pair}
We have
$$\int_X \mu^2=\frac{1}{2}.$$
\end{proposition}
\begin{proof}
Since $\int_X (2u^{n/2})^2=2$ and $\int_X (2u^{n/2})\cup (\mu+u^{n/2})=1$ and $\int_X (\mu+u^{n/2})^2= \int_X \mu^2+\frac{1}{2}>\frac{1}{2}$. To have the intersection being unimodular, we must have $\int_X \mu^2=\frac{1}{2}$.
\end{proof}

\subsection{Proof of \Cref{thm:2_power}}
The proof of \Cref{thm:2_power} seeks contradictions by finding non-integer Chern numbers using tools like the Hirzebruch signature theorem. As a warm-up, we first look at the $n=8$ case by direct computations.

\subsubsection{The $n=8$ case}

\begin{proof}[Proof of \Cref{thm:2_power} for $n=8$]
In this case, the $L$-genus is
\begin{eqnarray*}
L_4(X) & = & \frac{1}{14175}(381p_4(X)-71p_1(X)p_3(X)-19p_2^2(X)+22p_1^2(X)p_2(X)-3p_1^4(X)),\\
p_1(X) & = & c_1^2(X)-2c_2(X), \\
p_2(X) & = & c_2^2(X)-2c_1(X)c_3(X)+2c_4(X),\\
p_3(X) & = & c_3^2(X)-2c_2(X)c_4(X)+2c_1(X)c_5(X)-2c_6(X),\\
p_4(X) & = & c_4^2(X)-2c_3(X)c_5(X)+2c_2(X)c_6(X)-2c_1(X)c_7(X)+2c_8(X).
\end{eqnarray*}
We also have 
$$c(C)=1+10u+44u^2+112u^3+182u^4+196u^5+140u^6+64u^7, u\in H^2(C;\Z),$$
and $\int_X u^8 =\frac{1}{2}$. We also have $\int_X c_8(X)=\chi(X)=10$. Using $c^{\Q}_i(X)=c^{\Q}_i(W)\oplus c^{\Q}_i(C)$, any Chern number does not involve $c_4$ can be computed in $C$ and $\int_X c_I(X)c_4(X)=\int_X c_I(C)c_4(C)$ for $c_I=\prod_{j=1} c^{i_j}_j$ if $i_4=0$, hence we can get specific numbers. The only remaining term is $\int_X c_4^2(X)$, which involves the unknown $c_4(W)$.  Then the Hirzebruch signature theorem reads
$$2=\sigma(X)=\int_X L_4(X)=\frac{1}{14175}(305 \int_Xc_4^2(X)-5036092).$$
Hence we have that $\int_X c_4^2(X)=\frac{5064442}{305}\notin \Z$, contradiction.
\end{proof}

\subsubsection{The general case of $n=2^k>8$}
In principle, one can continue the computation above for larger $n$.  However, it is unrealistic to carry out the specific computations, as the complexity grows exponentially w.r.t.\ $k$ both in coefficients and Chern numbers. The computation is simplified significantly if we compare the discrepancy between the Hirzebruch signature formulae for $X,X'$, where $X'$ is the orbifold obtained from gluing $C$ to $\C^n/(\Z/2)$ and only look at the factor of $2$. We first set up some modulo-$2$ properties for later purposes. 

\begin{proposition}\label{prop:mod4}
For $n=2^k$ and $k\ge 2$, we have that the free part of $c_{n/2}(W)=0 \mod 4$ for a hypothetical exact filling $W$.
\end{proposition}
\begin{proof}
Following \cite{MR2192936}, we have
$\sigma(X)-\int_X v_n^2 = 0\mod 8,$ where $v_n$ is an integral lift of the middle Wu class. Since $X$ has an almost complex structure, there are integral lifts of the Wu classes in terms of the Chern classes by \cite[E 1.2]{MR2192936}. Namely, the total integral Wu classes have characteristic series ($x$ is the Chern root)
$$1+x+x^3+\ldots+x^{2^m-1}+\ldots.$$
Let $v_n$ denote the integral lift of the middle Wu class in $H^n(X;\Q)$. As $H^n(X;\Q)=H^n(W;\Q)\oplus H^n(C;\Q)$, we will compute the two components of $v_n$.

For the component on $C$, as $TC|_{\CP^{n-1}}\oplus \underline{\C}=\cO(1)^{n}\oplus \cO(2)$, using the characteristic series, we have 
$$v_n|_C   =   \text{ the } u^{2^{k-1}} \text{ coefficient of }(1+u+u^3+\ldots)^{2^k}(1+2u+8u^3+\ldots ).$$
Note that $\binom{2^k}{i}=0\mod 4$ if $1\le i <2^{k-1}$ and $\binom{2^k}{2^{k-1}}=2\mod 4$. Therefore, modulo $\la 4, u^{2^{k-1}+1}\ra$, we have
\begin{eqnarray*}
(1+u+u^3+\ldots)^{2^k}(1+2u+8u^3+\ldots ) & = &  (1+2(u+u^3+\ldots)^{2^{k-1}})(1+2u)  \\
& = & 1+2u+2u^{2^{k-1}}
\end{eqnarray*}
In particular, we have $v_n|_C=2u^{2^{k-1}} \mod 4$. On the other hand, by \Cref{lemma:product} below,  we have $v_n|_W =  c_{n/2}(W)$. By \Cref{prop:middle} and $c_{n/2}(\RP^{2n-1},\xi_{\std})=0$, we have $c_{n/2}(W) = 0 \mod 2$. If we write $c_{n/2}(W)=2m\mu$ (the free part only), then by \Cref{prop:pair}, we have 
$$\int_X v_n^2 =2+2m^2 \mod 8.$$
Since $\sigma(X)=2$, we must have $m=0\mod 2$, i.e.\ the free part of $c_{n/2}(W)=0 \mod 4$ for the hypothetical exact filling $W$.

\end{proof}

\begin{lemma}\label{lemma:product}
The coefficient of $c_{n/2}$ in the integral Wu class $v_n$ is $1$. 
\end{lemma}
\begin{proof}
We write $n=2m$. Let $f(x)=1+x+x^3+\ldots$ be the characteristic series. We want to write 
$$f(x_1)\cdot \ldots \cdot f(x_{2m}) = g_{2m}(e_1,\ldots,e_{2m})$$
where $e_1,\ldots,e_{2m}$ are the elementary symmetric polynomials. Then the coefficient we are looking for is the coefficient of the monomial $e_m$ in $g_{2m}$. Note that $e_{2m}$ divides $f(x_1)\cdot \ldots \cdot f(x_{2m})-g_{2m-1}(e_1,\ldots,e_{2m-1})$ as $x_{2m}$ divides it. We have
$$f(x_1)\cdot \ldots \cdot f(x_{2m})-g_{2m-1}(e_1,\ldots,e_{2m-1})=e_{2m}s$$
for some symmetric polynomial $s$. As a consequence, the $e_m$ coefficient of $g_{2m}$ is the same as the $e_{m}$ coefficient of $g_{2m-1}$. We can keep applying the argument for $m-1$ more times and conclude that $e_m$ coefficient of $g_{2m}$ is the same as the $e_{m}$ coefficient of $g_m$, which is clearly $1$.
\end{proof}

\begin{proof}[Proof of Theorem \ref{thm:2_power}]
Let $X'$ be the orbifold obtained from gluing $\C^n/(\Z/2)$ to $C$. With the orbifold Hirzebruch signature theorem in \cite{orbifold_signature}, we have 
$$2=\sigma(X)=\int_X L(TX),\quad 1=\sigma(X')=\int_{X'} L(TX'),$$
Here $\sigma(X')$ is the signature on the \emph{de Rham cohomology} of the orbifold $X'$, which is isomorphic to the real singular cohomology of the orbit space,  \emph{not the Chen-Ruan cohomology}. The $L$-genus of $TX'$ is defined through the orbifold Chern-Weil theory, which can be expressed in terms of (orbifold) Pontryagin classes and Chern classes with the same coefficients as in the manifold case. It is important to note that the orbifold Hirzebruch signature theorem usually involves integrating characteristic classes over the singular locus. However, it vanishes in our case, as the corresponding characteristic is given by \cite[\S 6, page 581]{MR236952} (also see \Cref{prop:orb_sig}) over the isolated singular point. Note that the only difference between $H^*(X;\R),H^*(X';\R)$ is in the middle degree from the difference between $W$ and $\C^n/(\Z/2)$. Hence we know that the difference of $\int_X L(TX)$ and $\int_{X'} L(TX')$ is given by a multiple of the integration of $c_n$ and a multiple of the integration of $c^2_{n/2}(W)$. For even $m$, we can write the $L$-genus of a complex bundle in terms of Chern classes as
$$L_m = a_m c_{2m}+b_m c_m^2+\text{other combinations of Chern classes.}$$
Then we have
\begin{eqnarray}
1 & = & \int_X L(TX) - \int_{X'} L(TX') \nonumber \\
  & = & \int_X a_{n/2}c_n(X)+\int_X b_{n/2} c_{n/2}^2(W) - \int_{X'} a_{n/2} c_n(X') \label{eqn:error}.
\end{eqnarray}
Now by the Gauss-Bonnet theorem for orbifolds in \cite{MR95520} and the identification between the Pfaffian polynomial and the top Chern class on a complex vector space, we have
 $$\int_{X'} c_n(X')=\chi^{\orb}(X')=n+\frac{1}{2},$$
 where the singularity is counted as $1/2$. Since $\int_X c_n(X)=\chi(X)=n+2$, and we assume $c_{n/2}(W)=4l\mu$ by \Cref{prop:mod4}, \eqref{eqn:error} reads
 \begin{equation}\label{eqn:relation}
     \frac{3}{2}a_{n/2}+8l^2b_{n/2}=1.
 \end{equation}
Therefore we need to figure out the coefficient $a_{n/2}$ and $b_{n/2}$. First for even $m$, we write $L_m$ as 
$$\alpha_m p_m +\beta_m p^2_{m/2}+\text{other combinations of the Pontryagin classes.}$$
Note that 
$$p_m=2c_{2m}+c_m^2+\text{other quadratic terms of Chern classes},$$
$$p_{m/2}=\pm 2c_m+\text{other quadratic terms of Chern classes}, $$
we have $a_m=2\alpha_m$ and $b_m=\alpha_m+4\beta_m$. Hence \eqref{eqn:relation} becomes
\begin{equation}\label{eqn:relation'}
3\alpha_{n/2}+8l^2(\alpha_{n/2}+4\beta_{n/2})=1.
\end{equation}
By \cite[(5)]{MR3856808}, 
we have 
\begin{eqnarray*}
    \alpha_m & = & \frac{2^{2m}(2^{2m-1}-1)}{(2m)!}B_m, \\
    \beta_m & = & \frac{2^{2m-1}}{\pi^{2m}}(-\zeta^*(2m)+\zeta^*(m)^2) \\
      & = & -\frac{2^{2m-1}(2^{2m-1}-1)}{(2m)!}B_m+\frac{2^{2m-1}(2^{m-1}-1)^2}{m!m!}B^2_{m/2},
\end{eqnarray*}
where $B_*$ are the Bernoulli numbers with a positive sign, see \Cref{app}. For $m=n/2=2^{k-1}$, \eqref{eqn:relation'} becomes
$$
3\cdot 2^{2m}(2^{2m-1}-1)B_m+2^{2m+3}\left(-(2^{2m-1}-1)B_m+2\cdot (2^{m-1}-1)^2\binom{2m}{m} B_{m/2}^2\right)l^2=(2m)!
$$
By \Cref{thm:CS}, $B_m$ in its lowest fraction form can be written as $N_m/(2D_m)$ for an odd number $D_m$. Hence we have 
$$
3\cdot 2^{2m-1}(2^{2m-1}-1)N_mD^2_{m/2}+2^{2m+2}\left(-(2^{2m-1}-1)N_mD_{m/2}^2+(2^{m-1}-1)^2\binom{2m}{m} N_{m/2}^2D_m\right)l^2=(2m)!D_mD_{m/2}^2.
$$
Note that $2^{2m-1}$ divides $(2m)!$ as $2m$ is a power of $2$.\footnote{This holds only if $m$ is a power of $2$, hence it yields a new proof of the main result in \cite{RP}.} We have
\begin{equation}\label{eqn:mod64}
    8\left((2^{m-1}-1)^2\binom{2m}{m} N_{m/2}^2D_m-(2^{2m-1}-1)N_mD_{m/2}^2\right)l^2 = D^2_{m/2}\left(\frac{(2m)!}{2^{2m-1}}D_m-3(2^{2m-1}-1)N_m\right)
\end{equation}
By \Cref{lemma:numerator_Bernoulli,lemma:trivial} for $m=2^{k-1}>8$, the right side of \eqref{eqn:mod64} is $-8 \mod 64$. While the left side of \eqref{eqn:mod64} is $8(1-\binom{2m}{m})l^2 \mod 64$ by \Cref{lemma:numerator_Bernoulli}. By \Cref{lemma:trivial}, 
$$\binom{2m}{m} = 6 \mod 8, \quad l^2 = 0,1,4 \mod 8,$$
hence 
$$8(1-\binom{2m}{m})l^2 = 0,24,32 \mod 64,$$
contradiction!

When $m=8$, we have $N_8=3617,N_4=1, D_8= 255, D_4=15$. Then \eqref{eqn:mod64} becomes
$$ 8\cdot 26266354875 \cdot l^2=  15^2\cdot 162465228408.$$
In particular, $l^2$ is not an integer, contradiction.

\end{proof}

\begin{remark}
There is a more direct approach to obtain the key relation \eqref{eqn:relation} without using the orbifold  Hirzebruch signature theorem and Gauss-Bonnet theorem. Note the total Chern class of the symplectic cap is in the product form
$$c(C)=(1+u)^n(1+2u), \mod u^{n+1}.$$
Roughly speaking, the $L$-genus $L(TX)$ should be ``close" to that of $C$, i.e.\ $(\frac{u}{\tanh u})^n \frac{2u}{\tanh 2u}$. More precisely, $c(C)=c(X)$ except for $c_{n/2}$ and $c_n$. Using that $\int_X u^n = \frac{1}{2}$, we have
\begin{eqnarray*}
\int_X (\frac{u}{\tanh u})^n \frac{2u}{\tanh 2u}  & =  & \frac{1}{2} \cdot \text{ the coefficient of $u^n$ in }(\frac{u}{\tanh u})^n \frac{2u}{\tanh 2u} \\
& = & \frac{1}{4\pi \i }\oint (\frac{u}{\tanh u})^n \frac{2u}{\tanh 2u} \frac{1}{u^{n+1}} \rd u \\
& = &  \frac{1}{4\pi \i }\oint \frac{(1+\tanh^2 u)}{\tanh^{n+1} u}  \rd u \\
& =  &  \frac{1}{4\pi \i }\oint \frac{(1+z^2)}{z^{n+1}(1-z^2)} \rd z =1, \qquad (z=\tanh u)
\end{eqnarray*}
Since the $u^n$ coefficient of $c(C)$ is $2n+1$, by that $\int_X u^n = \frac{1}{2}$, the above computes $\int_X L(TX)$ using ``$\int_X c_n(X)=n+\frac{1}{2}$ and $c_{n/2}(X) = 0\oplus c_{n/2}(C)$". If we use the correct $c_{n/2}$ and that $\int_X c_n(X)=n+2$, the  Hirzebruch signature theorem becomes
$$2=\frac{3}{2}a_{n/2}+b_{n/2}\int_X c^2_{n/2}(W)+\int_X (\frac{u}{\tanh u})^n \frac{2u}{\tanh 2u},$$
that is precisely \eqref{eqn:relation}.

However, the orbifold approach generalizes to a general quotient singularity without any reference to a particular symplectic cap $C$, while the direct computation above requires knowledge of $C$.
\end{remark}

\subsubsection{The $n=4$ case}
To prove the $n=4$ case, the Hirzebruch signature theorem alone (or with the modulo $8$ property of the modified signature in \cite{MR2192936}) will not give us a contradiction. We also need to use that the $\hat{A}$-genus is an integer for spin manifolds.

\begin{proof}[Proof of \Cref{thm:2_power} for $n=4$]
The second $L$-genus is given by 
$$L_2(X)=\frac{1}{45}(7p_2(X)-p^2_1(X)).$$
Since $X$ is an almost complex manifold, hence the Pontryagin classes can be expressed using Chern classes as 
$$p_1(X)=c^2_1(X)-2c_2(X),\quad p_2(X)=c_2^2(X)-2c_1(X)c_3(X)+2c_4(X).$$
As a consequence, we have
$$L_2(X)=\frac{1}{45}(3c_2^2(X)-14c_1(X)c_3(X)+14c_4(X)-c_1^4(X)+4c_1^2(X)c_2(X)).$$
On one hand, we have 
$$c(C)=(1+u)^4(1+2u) \mod u^4=1+6u+14u^2+16u^3, u\in H^2(C;\Z)$$
We have $\int_X u^4=\frac{1}{2}$ when we view $u\in H^2(C;\Z)\subset H^2(C;\Q)\simeq H^2_c(C;\Q)$. Finally, we have $\int_X c_4(X)=6$ and $\int_{X}c^2_1(X)c_2(X)=\int_X c_1^2(C)c_2(C)$ using the natural inclusion from $c^{\Q}_i(X)=c^{\Q}_i(W)\oplus c^{\Q}_i(C)$. Then by the Hirzebruch signature theorem, we have
$$\int_X L_2(X)=\sigma(X)=2.$$
Hence we can compute that $\int_X c_2^2(X)=106$. As a consequence, we have $\int_W c_2^2(W)=8$. Since $p_1(W)=2c_2(W)$, we have $\int_W p_1^2(W)=32$. 

Now we consider a new closed manifold $Y=W\cup_{\phi} W$, where $\phi:\RP^7\to \RP^7$ is the orientation reversing diffeomorphism induced from $(z_1,z_2,z_3.z_4)\to (\overline{z}_1,z_2,z_3,z_4)$. Then we have $\sigma(Y)=2$ and $\int_Y p_1^2(Y)=2\int_W p_1^2(W)=64$. Hence from  the Hirzebruch signature theorem
$$\frac{1}{45} \int_Y 7p_2(Y)-p_1^2(Y) = 2,$$
we have $\int_Y p_2(Y)=22$. 

We claim that $Y$ is a spin manifold. To prove $w_2(Y)=0$, it is sufficient to prove $\la w_2(Y),A\ra = 0$ for any $A\in H_2(Y;\Z/2)$. Since $W$ is simply connected, we have that $Y$ is also simply connected. Then by the universal coefficient theorem, Hurewicz theorem, and Whitney embedding theorem, we have that $H_2(Y;\Z/2)$ is generated by embedded spheres. Therefore we have that $H^2(Y;\Z/2)$ is generated by embedded spheres contained in one of the two copies of $W$ and the glued sphere from two disks in each $W$ whose boundary is nontrivial in $\pi_1(\RP^7)$. Assume $A\in H_2(Y;\Z/2)$ is represented by a map $f:S^2\to Y$, then $\la w_2(Y),A \ra =0$ is equivalent to that $f^*TY$ is trivial. If $f$ is contained in one of $W$, since $c_1^{\Z}(W)=0$, we have $f^*TY$ is trivial. If $f$ is the glued sphere, we may assume the separating circle $C$ is the circle in the $z_1$ complex line. We can fix a trivialization of $TY|_C$, which is derived from a trivialization of the normal bundle of $C$ in $\RP^7$, a chosen orientation of $C$, and outward normal vectors of the first copy of $W$. Since $\phi$ maps $C$ to $C$ in an orientation reversing manner, $\phi$ push the trivialization to the same (up to homotopy as it is multiplying the constant map $S^1\to \diag(-1,-1,1,\ldots,1)$) trivialization on the second copy of $W$. Therefore the second Stiefel-Whitney class of $f^*TY$ is measured by the sum of the discrepancy of the trivialization of $f^*TY$ on the two disks with the above trivialization, which is clearly a multiple of $2$, hence zero. Therefore $Y$ is a spin manifold. Now by the Atiyah-Singer index theorem for Dirac operators, we have 
$$\Z \ni \int_Y \hat{A}(Y) = \frac{1}{5760}\int_Y -4p_2(Y)+7p_1^2(Y) = \frac{1}{16}$$
which is a contradiction.
\end{proof}

\subsection{Unique exact orbifold fillings}
\begin{proposition}\label{prop:h-cob}
Let $M$ be an exact cobordism from $(S^{2n-1},\xi_{\std})$ to itself for $n\ge 2$, then $M$ is an $h$-cobordism.
\end{proposition}
\begin{proof}
Since $M$ glued with $\D^{2n}$ is an exact filling of  $(S^{2n-1},\xi_{\std})$, which is diffeomorphic to $\D^{2n}$. Hence we have $\pi_1(M)=0$ and $H^*(M,\partial_-M;\Z)=0$ by the tautological long exact sequence of pairs and excision. Then by Lefschetz duality and the universal coefficient theorem, we have $H_*(M,\partial_-M;\Z)=H_*(M,\partial_+M;\Z)=0$. Then the Hurewicz theorem implies that $M$ is an $h$-cobordism.
\end{proof}

\begin{theorem}[\Cref{thm:unique}]\label{thm:2_unique}
Let $W$ be an exact orbifold filling of $(\RP^{2n-1},\xi_{\std})$ for $n\ge 3$ odd. Then $W$ is diffeomorphic to $\D^{2n}/(\Z/2)$. 
\end{theorem}
\begin{proof}
Let $W$ be an exact orbifold filling, by \cite[Theorem G]{gironella2021exact}, there is exactly one orbifold point modelled on $\C^n/(\Z/2)$. Then, outside a small neighborhood of the singularity,  we get an exact cobordism $M$ with both concave boundary $\partial_-M$ and convex boundary $\partial_+M$ are $(\RP^{2n-1},\xi_{\std})$. By the same argument of \Cref{prop:fundamental}, we know that $\pi_1(\RP^{2n-1})\to \pi^{\orb}_1(W)$ is surjective. Since $\pi_1(\RP^{2n-1})\to \pi_1^{\orb}(\C^n/(\Z/2))$ is an isomorphism, by the van-Kampen theorem, we have $\pi^{\orb}_1(W)=\pi_1(M)$. Therefore we have $\pi^{\orb}_1(W)=\pi_1(M)=0$ or $\pi^{\orb}_1(W)=\pi_1(M)=\Z/2$. If $\pi^{\orb}_1(W)=\pi_1(M)=0$, then $H^2(M;\Z)$ has no torsion. Since $c_1(\RP^{2n-1},\xi_{\std})\ne 0$ when $n$ is odd, we must have $\rank H^2(M;\Z)\ge 1$. We can stack several copies of $M$ together to get an orbifold filling of $(\RP^{2n-1},\xi_{\std})$, such that the cobordism $M$ from deleting the orbifold point satisfying $\rank H^2(M;\Z) \ge 3$. 

Now we use the homotopy argument in the proof of \cite[Theorem B]{gironella2021exact}, namely we consider the moduli space of thrice-punctured sphere solving the perturbed Cauchy-Riemann equation with two positive punctures and one negative puncture, i.e.\ the moduli space defining the pair-of-pants product. By \Cref{prop:upper_bound}, $x_{(\Id)}+a$ represents a class in $SH_+^*(W;\Q)$ such that $x_{(\Id)}+a$ is sent to $|G|$ in the composition $SH_+^*(W;\Q)\to H^{*+1}_{\CR}(W;\Q)\to H^0_{\CR}(W;\Q)$. We plug in $x_{(\Id)}+a$ as one of the inputs and a constant orbit $y$ representing a class in $H^2(M;\Q)\subset H^2_{\CR}(W;\Q)$ as another input. Then by analyzing the degeneration, the output is a cochain $z$ from the symplectic cohomology cochain complex, such that 
\begin{equation}\label{eqn:prod}
    \delta(z)=m(\delta_{+,0}(x_{(\Id)}+a),y)
\end{equation}
where $\delta_{+,0}$ is the cochain map defining $SH_+^*(W;\Q)\to H_{\CR}^{*+1}(W;\Q)$ and $m$ is the map defining the pair-of-pants product. By our assumption, $\delta_{+,0}(x_{(\Id)}+a)$ is $|G|$ plus constant orbits representing a class $A$ in the kernel of $H^*_{\CR}(W;\Q)\to H^0_{\CR}(W;\Q)$. As both $\delta_{+,0}(x_{(\Id)}+a)$, $y$ are constant orbits, and $m(\delta_{+,0}(x_{(\Id)}+a),y)$ also consists of constant orbits. Moreover, since $y$ is from $M$, no constant holomorphic curve concentrated at the singularity appears in the above analysis. We write $z=z_++z_0$, where $z_+$ is the positive cochain part and $z_0$ is the Morse theory part.  Then \eqref{eqn:prod} implies that 
$$\delta_+(z_+)=0, \quad \delta_{+,0}(z_+) = m(|G|+A,y) \in H^*_{\CR}(W;\Q)$$
where $\delta_+$ is the differential on the positive cochain complex. Since $m(|G|+A,y)$ does not involve constant holomorphic curves concentrated at the singularity, $m(|G|+A,y)$ is the cup-product of $y$ and $|G|+A'$ in the cohomology $H^*(W;\Q)\subset H_{\CR}^*(W;\Q)$, where $A'$ is the $H^*(W;\Q)$-component of $A$. In particular, $m(|G|+A,\cdot)$ is an injective map from $H^*(M;\Q)$. Therefore we have that $\dim SH^*_+(W;\Q)$ is at least $\dim H^*(M;\Q)\ge 3$ by $z_+$ above. This contradicts with \Cref{prop:SS}.\footnote{We get an instant contradiction if  \Cref{prop:upper_bound} holds for exact orbifolds, this boils down to establishing the pair-of-pants product for exact orbifolds. }

When  $\pi^{\orb}_1(W)=\pi_1(M)=\Z/2$. In this case $\pi_1(\RP^{2n-1})\to \pi^{\orb}_1(W)=\pi_1(M)=\Z/2$ is an isomorphism. We claim that $\Z/2=\pi_1(\partial_-M)\to \pi_1(W)$ is also an isomorphism. Otherwise, the map is zero. We can consider the universal cover of $M$, which an exact cobordism with convex boundary $(S^{2n-1},\xi_{\std})$ and concave end $(\RP^{2n-1},\xi_{\std})\sqcup (\RP^{2n-1},\xi_{\std})$. This would yield an exact orbifold filling of $(S^{2n-1},\xi_{\std})$ with  two orbifold points, contradicting \cite[Theorem C]{gironella2021exact}. Now since $\Z/2=\pi_1(\partial_-M)\to \pi_1(W)$ is an isomorphism, the universal cover $\widetilde{M}$ is an $h$-cobordism by \Cref{prop:h-cob}. In particular, $M$ itself an $h$-cobordism. Since the Whitehead group of $\Z/2$ is trivial, $M$ is an $s$-cobordism, hence a trivial cobordism as $n\ge 3$. As a consequence, we have that $W$ is diffeomorphic to $\D^n/(\Z/2)$.
\end{proof}

\section{More non-existence results}\label{s7}
\Cref{prop:upper_bound,prop:intersection_SS} provide evidences for the conjecture that exact orbifold fillings of $(S^{2n-1}/G,\xi_{\std})$ might be unique, namely $\C^n/G$, when it is a terminal singularity. If this conjecture is true, then in particular,  $(S^{2n-1}/G,\xi_{\std})$ has no exact filling when  $\C^n/G$ is terminal. In addition to the case of real projective spaces discussed in the previous section, we give various but non-exhaustive lists of cases where we can establish the non-existence of exact fillings using the topological information obtained in previous sections. The following are some of the arguments we use to arrive at a contradiction assuming there is an exact filling $W$:
\begin{enumerate}
    \item From enough non-vanishing maps $H^*(W;\Z)\to H^*(S^{2n-1}/G;\Z)$, we obtain a contradiction with $\dim H^*(W;\Q)=|\Conj(G)|$ using Lefschetz duality and the universal coefficient theorem. This can be achieved by finding enough non-vanishing Chern classes of $(S^{2n-1}/G,\xi_{\std})$, which is the method used in 
     \cite{RP} or by exploiting the non-trivial product structure on $H^*(S^{2n-1}/G;\Z)$ as well as the cup product length estimate in \S \ref{s4},  e.g.\ \Cref{thm:U(n)}.
    \item Let $X$ be the compact symplectic manifold obtained by gluing a symplectic cap of dimension $4k$ and consider its signature $\sigma(X)$. Then a potential contradiction of having non-integral characteristic classes may arise using the the Hirzebruch signature theorem. Moreover, by \cite{MR2192936}, we have $\sigma(X)-\int_X v^2_{2k}=0\mod 8$, where $v^2_{2k}$ is an integral lift of the middle Wu class, which always exists for almost complex manifolds. This also provides a place to look for a contradiction.
    \item If $c_1^{\Q}(W)=0$, then the boundary rational grading is a global rational grading and the isomorphism $SH^*_+(W;\Q)\simeq H^{*+1}(W;\Q)$ is a graded isomorphism. We can get a contradiction if $SH^*_+(W;\Q)$ is not supported in integer degrees.
\end{enumerate}

\subsection{Preliminary facts}
\begin{proposition}\label{prop:dual}
Assume $\C^n/G$ is an isolated quotient singularity. Let $W$ be a compact oriented manifold such that $\partial W = S^{2n-1}/G$. Then $ H^i(W;\Q)= H^{2n-i}(W;\Q)$ for $0<i<2n$. 
\end{proposition}
\begin{proof}
Note that we have the long exact sequence 
$$\ldots \to H^*_c(W;\Q)\to H^*(W;\Q) \to H^*(\partial W; \Q) \to \ldots$$
Since $H^*(S^{2n-1}/G;\Q)=H^*(S^{2n-1};\Q)$, we have $H^{i}_c(W;\Q)=H^{i}(W;\Q)$ when $0<i<2n$. Then by Lefschetz duality, we have $H^i(W;\Q)=H^{2n-i}_c(W;\Q)=H^{2n-i}(W;\Q)$ for $0<i<2n$.
\end{proof}

\begin{proposition}\label{prop:restriction'}
Assume $\C^n/G$ is an isolated quotient singularity. Suppose that both $H^*(W;\Q)$ and $H^{*}(BG;\Z)$ are supported in even degrees. Let $\phi^{*}:H^*(W;\Z)\to H^*(S^{2n-1}/G;\Z),\phi^*_{\Tor}:\Tor H^*(W;\Z)\to H^*(S^{2n-1}/H;\Z)$ denote the maps given by fullbacks along the inclusion of the boundary. For $0<i<n$, we have
$$|\Ima \phi^{2n-2i}_{\Tor}|=|\coker \phi^{2i}|.$$
\end{proposition}
\begin{proof}
By \Cref{prop:cohomology}, $H^*(S^{2n-1}/G;\Z)=H^*(BG;\Z)$ whenever $*<2n-1$. Therefore, by our assumption that $H^{*}(BG;\Z)$ is supported in even degrees,  we have exact sequences
$$0 \to  H^{2i}(W,\partial W;\Z) \to  H^{2i}(W;\Z)\to H^{2i}(S^{2n-1}/G;\Z)  \to H^{2i+1}(W,\partial W;\Z) \to H^{2i+1}(W;\Z) \to 0,    $$
$$0 \to  H^{2n-2i}(W,\partial W;\Z) \to  H^{2n-2i}(W;\Z)\to H^{2n-2i}(S^{2n-1}/G;\Z)  \to H^{2n-2i+1}(W,\partial W;\Z) \to H^{2n-2i+1}(W;\Z) \to 0.    $$
Let $k=\dim H^{2i}(W;\Q)=\dim H^{2n-2i}(W;\Q)$. By Lefschetz duality and the universal coefficient theorem, we have 
$$H^{2i}(W,\partial W;\Z) = H_{2n-2i}(W;\Z) = \Z^k\oplus \Tor H^{2n-2i+1}(W;\Z)=\Z^k\oplus H^{2n-2i+1}(W;\Z),$$
$$H^{2n-2i}(W,\partial W;\Z) = H_{2i}(W;\Z) = \Z^k\oplus \Tor H^{2i+1}(W;\Z)=\Z^k\oplus H^{2i+1}(W;\Z),$$
$$H^{2i+1}(W,\partial W;\Z) = H_{2n-2i-1}(W;\Z) =  \Tor H^{2n-2i}(W;\Z),$$
$$H^{2n-2i+1}(W,\partial W;\Z) = H_{2i-1}(W;\Z) =  \Tor H^{2i}(W;\Z).$$
Hence the two long exact sequences become
$$0 \to \Z^k\oplus H^{2n-2i+1}(W;\Z) \to  \Z^k\oplus \Tor H^{2i}(W;\Z)\to H^{2i}(S^{2n-1}/G;\Z) \to \Tor H^{2n-2i}(W;\Z), \to H^{2i+1}(W;\Z) \to 0,    $$
$$0 \to \Z^k\oplus  H^{2i+1}(W;\Z)\to  \Z^k\oplus \Tor H^{2n-2i}(W;\Z)\to H^{2n-2i}(S^{2n-1}/G;\Z) \to \Tor H^{2i}(W;\Z) \to H^{2n-2i+1}(W;\Z) \to 0.    $$
From the first exact sequence, we have
$$|\Tor H^{2n-2i}(W;\Z)|=|H^{2i+1}(W;\Z)|\cdot |\coker \phi^{2i}|.$$
From the second exact sequence, we have a short exact sequence,
$$0\to  H^{2i+1}(W;\Z) \to \Tor H^{2n-2i}(W;\Z) \to \Ima \phi^{2n-2i}_{\Tor}\to 0,$$
hence 
$$|\Tor H^{2n-2i}(W;\Z)|=|H^{2i+1}(W;\Z)|\cdot |\Ima \phi^{2n-2i}_{\Tor}|,$$
which implies the claim.
\end{proof}

That $H^*(W;\Q)$ is supported in even degrees holds for our hypothetical fillings if the singularity is terminal by \Cref{prop:upper_bound} and that $H^*(BG;\Z)$ is supported in even degrees holds for cyclic or $A,D,E$ type groups. A special case of \Cref{prop:restriction'} is when $n$ is even and $i=n/2$. If $|H^{n}(S^{2n-1}/G;\Z)|$ is not a square, then we have $|\Ima \phi^{n}_{\Tor}|\ne |\Ima \phi^{n}|$, i.e.\ $H^n(W;\Q)\ne 0$. If one can argue that $\phi^{2i},\phi^{2n-2i}$ are both surjective and $H^{2i}(S^{2n-1}/G;\Z)=H^{2n-2i}(S^{2n-1}/G;\Z) \ne 0$,\footnote{The first equality always holds by Poincar\'e duality and the universal coefficient theorem.} then we must have $H^{2i}(W;\Q)= H^{2n-2i}(W;\Q)\ne 0$. This is the argument for \cite[Theorem 1.1]{RP}. More generally, we have

\begin{proposition}\label{prop:restriction}
Under the same assumptions of \Cref{prop:restriction'}. Suppose we have $H^{2i}(S^{2n-1}/G;\Z)=H^{2n-2i}(S^{2n-1}/G;\Z)=\Z/m$ where $1<i<n$. Let $p$ be a prime number that divides $m$. If we know that both $\Ima \phi^{2i},\Ima \phi^{2n-2i}$ contain elements which are non-zero modulo $p$, then $H^{2i}(W;\Z),H^{2n-2i}(W;\Z)$ can not  be torsion.
\end{proposition}
\begin{proof}
Under the assumption, we know that $\coker \phi^{2i},\coker \phi^{2n-2i}$ can only be $\Z/k$, where $k$ is a factor of $m$ and is coprime to $p$. If $H^{2n-2i}(W;\Z)$ is torsion, by \Cref{prop:restriction'}, we know that $|\Ima \phi^{2n-2i}|=|\Ima \phi^{2n-2i}_{\Tor}|=|\coker \phi^{2i}|$ is coprime to $p$, contradicting with $|\Ima \phi^{2n-2i}|\cdot |\coker \phi^{2n-2i}|=m$, which is divisible by $p$.
\end{proof}

Finally, we recall the orbifold Hirzebruch signature theorem for orbifolds with isolated singularities following \cite{MR236952,orbifold_signature}. Assume $G\subset U(n)$, such that $\C^n/G$ is an isolated singularity. For each $g\in G\backslash \{\Id\}$, $g$ can be diagonalized as $\diag(\theta^1_g,\ldots,\theta^n_g)$ for $0<\theta^j_g<2\pi$. Then we define
$$L_{\C^n/G}:=\frac{1}{|G|}\sum_{\substack{g\in G\backslash\{\Id\}\\ g\ne -\Id }} \prod_{j=1}^n(\i \tan(\theta^j_g/2))^{-1}=\sum_{\substack{(g)\in \Conj(G)\backslash\{(\Id)\}\\ (g)\ne (-\Id) }} \frac{1}{|C_g|}\prod_{j=1}^n(\i \tan(\theta^j_g/2))^{-1}.$$
Here $\tan(\pi/2)$ is defined to be $\infty$ (this also explains why we drop $g=-\Id$). In particular, when $G=\Z/2=\{\Id,-\Id\}$, we define $L_{\C^n/G}=0$.
\begin{proposition}[\cite{orbifold_signature}]\label{prop:orb_sig}
Assume $X$ is $4n$ dimensional oriented orbifold with only isolated singularities modeled on $\C^{2n}/G_i$ with the complex orientation for $1\le i \le k$. Then we have
$$\sigma(X)=\int_X L(TX)+\sum_{i=1}^k L_{\C^{2n}/G_i}.$$
\end{proposition}

\subsection{General results}
Before moving to the special cases of cyclic groups and finite subgroups of $SU(2)$, we first state some general results on the non-existence of exact fillings. 

\begin{theorem}[\Cref{thm:J}]\label{thm:2_conj}
Assume $\C^n/G$ is an isolated terminal singularity with $n$ odd and $|\Conj(G)|$ even, then the contact link $(S^{2n-1}/G,\xi_{\std})$ has no exact filling.
\end{theorem}
\begin{proof}
By \cref{prop:upper_bound}, $H^*(W;\Q)$ is supported in even degrees. By \Cref{prop:dual} and $n$ being odd, we have $\dim H^*(W;\Q)$ is odd, contradicting with \Cref{prop:upper_bound}, i.e.\ $\dim H^*(W;\Q)=|\Conj(G)|$.
\end{proof}

\begin{theorem}\label{thm:CY}
For any isolated terminal singularity $\C^n/G$ with $G\not\subset SU(n)$, we have $(S^{2n-1}/G,\xi_{\std})$ has no Calabi-Yau filling, namely there is no strong filling $W$ with $c_1^{\Q}(W)=0$.
\end{theorem}
\begin{proof}
In the case of strong fillings, we still have a tautological long exact sequence
$$\ldots \to QH^*(W)\to SH^*(W;\Lambda)\to SH^*_+(W;\Lambda)\to QH^{*+1}(W) \to \ldots $$
where $SH^*(W;\Lambda), SH^*_+(W;\Lambda)$ are defined over the Novikov field $\Lambda$ (over $\Q$). Here $QH^*(W)$ is the quantum cohomology of $W$, which as a $\Lambda$-module is the same as $H^*(W;\Q)\otimes_{\Q}\Lambda$. The notation $QH^*(W)$ is to emphasizes that the product structure is the deformed quantum product. Using such a deformed product, $QH^*(W)\to SH^*(W;\Lambda)$ is still a ring map. The proof of \Cref{prop:upper_bound} implies that we have
$$SH^*_+(W)\to QH^{*+1}(W)\to H^0(\partial W;\Q)\otimes_{\Q} \Lambda$$
is surjective for any strong filling $W$. Hence from the long exact sequence, we know that $1+A\in QH^*(W)$ is mapped to zero in $SH^*(W;\Lambda)$ for some $A\in \oplus_{j>0}H^j(W;\Q)\times_{\Q}\Lambda$. Although $1+A$ is always a unit in the undeformed ring, it may not be in $QH^*(W)$, see \cite[Remark 2.17]{RP}. Symplectically, this is the key difference between strong fillings and exact fillings that we try to take advantage of. However, if we assume $c_1^{\Q}(W)=0$,  we have a global rational grading (which is canonical, as all Reeb orbits are of finite order in homology) for the symplectic cohomology for any Calabi-Yau filling $W$. In particular, $QH^*(W),SH^*(W;\Lambda),SH^*_+(W;\Lambda)$ are all rationally graded such that the tautological long exact sequence holds\footnote{Strictly speaking, there is one long exact sequence for each rational number in $[0,1)$, i.e.\ the residue of the rational grading.}, and the filtration in \Cref{prop:SS} is just the filtration from the global grading. Because of this grading, we see that $A=0$ and $1$ is mapped to zero in $SH^*(W;\Lambda)$. Since $1\in QH^*(W)$ is always a unit even if the product structure is deformed, the argument of \Cref{prop:upper_bound} goes through to show that $SH^*_+(W;\Lambda)\simeq H^{*+1}(W;\Q)\otimes_{\Q}\Lambda$ as graded spaces. Similarly, if we look at the $\Z/2$ grading instead of the $\Q$ grading, we must have $H^*(W;\Q)$ supported in even degrees. Therefore  $SH^*_+(W;\Lambda)\simeq H^{*+1}(W;\Q)\otimes_{\Q}\Lambda$ being a graded isomorphism implies that  $\{2\age(g)\}_{(g)\in \Conj(G)}\in 2\N$. Finally, note that $G\subset SU(n)$ is equivalent to that $\age(g)\in \N$. Hence we arrive at a contradiction if $G\not\subset SU(n)$. 
\end{proof}
\begin{remark}
Having the Calabi-Yau condition makes the boundary rational grading a global rational grading, which puts stronger restrictions on the topology of the hypothetical filling. A lot of the proofs for the non-existence of exact fillings below will prove the Calabi-Yau condition holds first. Indeed, all results on the non-existence of exact fillings in this paper hold for the non-existence of Calabi-Yau fillings, as \Cref{prop:upper_bound} holds for Calabi-Yau fillings by the argument in \Cref{thm:CY}. Such phenomena were previously explored in \cite[Corollary 3.4, Remark 3.6]{RP}.
\end{remark}

\subsection{Cyclic groups}
\begin{theorem}\label{thm:U(n)}
Assume $\C^n/(\Z/m)$ is an isolated terminal singularity. If $c_1(S^{2n-1}/(\Z/m),\xi_{\std})\in \Z/m$ is not zero modulo a prime factor $p$ of $m$, then $(S^{2n-1}/(\Z/m),\xi_{\std})$ is not exactly fillable. In particular, when $m$ is a prime and $\Z/m\subset U(n)$ is not contained in $SU(n)$, then $(S^{2n-1}/(\Z/m),\xi_{\std})$ is not exactly fillable. 
\end{theorem}
\begin{proof}
By \Cref{cor:cuplength}, we have $c_1(W)^{n-1} \in \Tor H^{2n-2}(W;\Z)$ for any hypothetical exact filling $W$ as $\mathrm{cl}(W)<{n-1}$. Note that $\phi^{2n-2}(c_1(W)^{n-1})=c_1(\xi_{\std})^{n-1}$, hence $c_1(\xi_{\std})^{n-1}\in \Ima \phi^{2n-2}_{\Tor}$. Note that $H^*(S^{2n-1}/(\Z/m);\Z)=H^*(BZ/m;\Z)$ when $*<2n-1$ and the isomorphism preserves the product structure, since $H^*(B\Z/m;\Z)=\Z[q]/\la mq \ra$ for $\deg q =2$, we have that $c_1(\xi_{\std})^{n-1}$ is non-trivial modulo $p$. If we write $m=p^sa$ for $(a,p)=1$, since $c_1(\xi_{\std})^{n-1}\in \Ima \phi^{2n-2}_{\Tor}$, we have $\Ima \phi^{2n-2}_{\Tor}$ is generated by $k\in \Z/m$ for $(k,p)=1$. As a consequence, we have  $p^s$ divides $|\Ima \phi^{2n-2}_{\Tor}|$ or equivalently $p$ does not divide $|\coker \phi^{2n-2}_{\Tor}|$. On the other hand, $c_1(\xi_{\std})\in \Ima \phi^2$ is non-zero modulo $p$, hence $p$ does not divide $|\coker \phi^2|$, we get a contradiction by \Cref{prop:restriction'}.
\end{proof}

By the same argument, one can generalize \Cref{thm:U(n)} to other groups and other Chern classes using the cup product length estimate in \Cref{cor:cuplength}. We will not pursue the precise statement in this paper. 

\begin{theorem}[\Cref{thm:3}]
    The contact link of any isolated terminal quotient singularity in complex dimension $3$ has no exact filling.
\end{theorem}
\begin{proof}
    According to \cite{MR715649}, isolated terminal singularities in complex dimension $3$ are either cyclic quotient singularities or cDV singularities. Moreover, terminal cyclic quotient singularities in complex dimension $3$ are completely classified by \cite{MR722406}. Namely, $\C^3/(\Z/m)$ is terminal iff the generator acts by $\diag(e^{\frac{2\pi \i}{m}},e^{-\frac{2\pi \i}{m}},e^{\frac{2a\pi \i }{m}})$, where $a$ is coprime to $m$. Then the first Chern class of $S^{5}/(\Z/m)$ is $a\in \Z/m=H^2(B\Z/m;\Z)$ by \Cref{prop:chern}. \Cref{thm:U(n)} can be applied to conclude the proof.
\end{proof}

\subsubsection{Cyclic actions with the same weight}
We first consider the $\Z/k$ action on $\C^n$ sending the generator to $\diag(e^{\frac{2\pi \i}{k}},\ldots, e^{\frac{2\pi \i}{k}})$. Then the quotient singularity is terminal if and only if $n>k$ and the contact link is the lens space $(L(k;1,\ldots,1),\xi_{\std})$. 

\begin{lemma}[{\cite[Proposition 3.2]{RP}}]\label{lem:p_adic}
         Let $p$ be a prime, $\binom{n}{i}\ne 0 \mod p$ if and only if  $i$ is digit-wise no larger than $n$ in the  $p$-adic representation.
\end{lemma}

\begin{theorem}\label{thm:cyclic}
Assume $n>k$ and there is no prime number $p$ such that $n=p^s,k=p^t$, then the lens space $(L(k;1,\ldots,1),\xi_{std})$ is not exactly fillable.
\end{theorem}
\begin{proof}
Assume $W$ is an exact filling. First by \Cref{prop:upper_bound}, we have $\dim H^*(W;\Q)=k$. By \Cref{prop:intersection_SS} and \Cref{rmk:grading}, there exists an element of $a \in \oplus_{*>0}H^*(W;\Q)$ such that $a^{n-1}\ne 0$, such $a$ corresponds to $y_{(1)}$ where $1\in \Z/k$.  We define $|a|_L$ to be the minimal grading of the components of $a$ in $H^*(W;\Q)$ and formally write $|0|_L=-\infty$. Since $|a|_L\ge 2$,  we have $|a^{i+1}|_L > |a^i|_L$ for $1\le i \le k-2$. Therefore $H^*(W;\Q)$ is spanned by $1$ and $a^i$ for $1\le i \le k-1$ and is supported in degree $0$ and $|a^i|_L$ for $1\le i \le k-1$. Therefore there exist classes of pure degree $\alpha_i$ (possibly with duplicates) for $1\le i \le k-1$, such that $\alpha_1\cdot \ldots \cdot \alpha_{k-1}\ne 0$. We may assume $|\alpha_1|\le \ldots \le |\alpha_k|$, then $\{\alpha_1\cdot \ldots \cdot \alpha_i\}_{1\le i \le k-1}$ are non-zero and with increasing degrees. As a consequence, their degrees should be the support of $H^*(W;\Q)$ minus $0$. If $|\alpha_2|>|\alpha_1|$, then we have $|\alpha_1|<|\alpha_2|<|\alpha_1\alpha_2|$. But we already shown that $H^{|\alpha_2|}(W;\Q)=0$, which is a contradiction. Therefore $|\alpha_2|=|\alpha_1|$. Similarly, if $|\alpha_3|>|\alpha_1|$, then we must have $|\alpha_1\alpha_2|<|\alpha_1\alpha_3|<|\alpha_1\alpha_2\alpha_3|$, which leads to a contradiction. We can keep applying the argument to conclude that  $|\alpha_1|=\ldots=|\alpha_{k-1}|$. To have the symmetry in \Cref{prop:dual}, $H^*(W;\Q)$ must be supported in degree $i\frac{2n}{k}$ for $0\le i \le k-1$. Consequently, we would have a contradiction if $k$ does not divide $n$. We will assume $k|n$ from now on. 

If $k\ne p^t$ for a prime number $p$. There exist primes $p,q$ such that $p,q$ divide $k$. We can write $n=ap^cq^d$ and $k=ep^fq^g$ with $a,e$ coprime to $p$ and $q$ and $c\ge f\ge 1,d\ge g\ge 1$. By \Cref{prop:chern}, we have $c_m(\xi_{\std})=\binom{n}{m} \mod p,q$. By \Cref{lem:p_adic}, we have $c_{p^c}(\xi_{\std}),c_{n-p^c}(\xi_{\std})\ne 0 \mod p$ and $c_{q^d}(\xi_{\std}),c_{n-q^d}(\xi_{\std})\ne 0 \mod q$. Then \Cref{prop:restriction} implies that $H^{2p^c}(W;\Q)$ and $H^{2q^d}(W;\Q)$ are not zero. As a consequence, we have $p^c,q^d$ is divisible by $n/k = \frac{a}{e}p^{c-f}q^{d-g}$. But this is only possible when $a=e,c=f,d=g$, i.e.\ $n=k$. Therefore we reach a contradiction. 

Now assume $k=p^t$ for a prime number, but $n$ is not in the form of $p^s$. As a consequence, we have $n=ap^s$ for $s\ge t$ and $a\ge 2$ coprime to $p$.  By \cite[Proposition 3.2]{RP}, we have $c_{p^s}(\xi_{\std}),c_{n-p^s}(\xi_{\std})\ne 0 \mod p$.  Then \Cref{lem:p_adic} implies that $H^{2p^s}(W;\Q)$ is not zero. As a consequence, we have $p^s$ is divisible by $n/k = a p^{s-t}$, contradicting with that $a$ is coprime to $p$. 

\end{proof}

\subsubsection{$\Z/3$ actions}

\begin{theorem}
For any isolated terminal quotient singularity $\C^n/(\Z/3)$, which is not $\C^{3^s}/(\Z/3)$ from the action of the same weight, the contact link $(S^{2n-1}/(\Z/3),\xi_{\std})$ is not exactly fillable.
\end{theorem}
\begin{proof}
The $\Z/3$ action on $\C^n$ with an isolated singularity is determined by the weights of a generator. Namely an integer $0\le m \le n$, such that the generator acts by $\diag(\underbrace{e^{\frac{2\pi \i}{3}},\ldots, e^{\frac{2\pi \i}{3}}}_{m}, \underbrace{e^{\frac{4\pi \i}{3}},\ldots, e^{\frac{4\pi \i}{3}}}_{n-m})$. By possibly choosing the other non-trivial element as the generator of $\Z/3$,  we may assume $m\le  n-m$. Moreover, being terminal implies that $n\ge 3$. By \Cref{prop:chern}, the total Chern class of $\xi_{\std}$ is
$$(1-u^2)^{m}(1-u)^{n-2m} \mod \la 3, u^n \ra.$$

\textbf{Case 1:} $n-2m\ne 0$. We assume $(m,n-2m)=a3^s$ for $(a,3)=1$. 
\begin{enumerate}
    \item If $3^{s+1}$ does not divide $n-2m$. Then by \Cref{lem:p_adic}, we have $c_{3^{s}}(\xi_{\std})\ne 0$. Now since $3^s$ divides $n$, we have $c_{3^{s}}^{\frac{n-3^s}{3^s}}(\xi_{\std})\ne 0$. Note that $\md(\C^n/(\Z/3))=\frac{n+m}{3}-1$, hence \Cref{cor:cuplength} implies that 
    $$\mathrm{cl}(W)\le \left \lfloor \frac{\frac{2n}{3}-\frac{m}{3}}{\frac{n}{3}+\frac{m}{3}}\right\rfloor\le 1.$$
    Now note that $n/3^s=(n-2m)/3^s+2m/s^3\ge 3$, we have $c_{3^{s}}^{\frac{n-3^s}{3^s}}(W)\in \Tor H^*(W;\Z)$ by $\mathrm{cl}(W)\le 1$. Therefore $\Ima \phi^{2n-2\cdot3^s}_{\Tor}=\Z/3$, while $\coker \phi^{2\cdot 3^s}=0$, contradicting with \Cref{prop:restriction'}.
    \item If $3^{s+1}$ divides $n-2m$. We write $n-2m=b3^t$ for $t>s$ and $(b,3)=1$. Then by \Cref{lem:p_adic}, we have $c_{3^{t}}(\xi_{\std})\ne 0$ and $c_{2\cdot 3^s}(\xi_{\std})\ne 0$. Since $3^t$ divides $n-2m$, we have $\coker \phi^{2n-4m}=0$. Similarly, we have $\coker \phi^{4m}=0$. Then by \Cref{prop:restriction'}, $H^*(W;\Q)$ must be supported in degree $0,4m,2n-4m$. In particular, we have $c_1^{\Q}(W)=0$, i.e.\ we have a global rational grading. Therefore we must have
    $$\left\{4m, 2n-4m\right\}=\left\{\frac{4n}{3}-\frac{2m}{3},\frac{2m}{3}+\frac{2n}{3}\right\},$$
     which implies that $7m=2n$ or $5m=2n$. In both cases, we have $m\ge 2$ is an even number
\begin{enumerate}
    \item If $7m=2n$, we can write $7m=2n=14r$ for an integer $r\ge 1$. Since $m=2r$, we have $2\cdot 3^s$ divides $m$. Then from $c_{2\cdot 3^s}(\xi_{\std})\ne 0$, we have $\phi^{2m}$ is surjective. From the discussion above,  we have $\phi^{2n-4m}$ is surjective, hence $\phi^{2n-2m}$ is also surjective. In view of \Cref{prop:restriction}, we must have $H^{2m}(W;\Q)=H^{2n-2m}(W;\Q)\ne 0$, contradicting with that $H^*(W;\Q)$ must be supported in degree $0,4m,2n-4m$.
    \item If $5m=2n$, we can write $5m=2n=10r$ for an integer $r\ge 1$. Then $n-2m=r$ and $m=2r$, contradicting with $(n-2m,m)=a3^s$ and  $3^{s+1}$ divides $n-2m$.
\end{enumerate}
\end{enumerate}

\textbf{Case 2:} $n=2m$.
By \Cref{prop:restriction'}, we have $\dim H^{n}(W;\Q)\ge 1$. In this case, the total Chern class of  $(S^{2n-1}/(\Z/3),\xi_{\std})$ is $(1-u^2)^{m}\mod 3.$ Hence if $m$ is not $3^k$ or $2\cdot 3^k$, there exists $i<m$ such that $c_i(\xi_{\std}),c_{n-i}(\xi_{\std})\ne 0$ by \Cref{lem:p_adic}. Therefore by \Cref{prop:restriction}, they have no exact fillings. 
\begin{enumerate}
    \item When $m=3^k$, by \Cref{cor:intersection_SS}, the signature of $W$ is $0$. Then by the same argument of \Cref{thm:2_power} and \Cref{prop:orb_sig} , we have
    $$\frac{16}{3}\cdot \frac{2^{2m}(2^{2m-1}-1)}{(2m)!}B_m+\frac{2^{2m}(2^{2m-1}-1)}{(2m)!}B_m\int_Xc_m^2(W)=L_{\C^n/(\Z/3)}=\frac{2}{3^{m+1}},$$
    where $X$ is a closed symplectic manifold glued from $W$ and a symplectic cap. Therefore we have
    \begin{equation}\label{eqn:3^k}
    3^{m}\cdot 2^{2m+3}(2^{2m-1}-1)N_m+3^{m+1}\cdot 2^{2m-1}(2^{2m-1}-1)N_m\int_X c_m^2(W)=2\cdot (2m)!D_m.
    \end{equation}
    Since $3c_m(W)$ lies in the image of $H^{2m}(W,\partial W;\Z)\to H^{2m}(W;\Z)$, we have $3\int_X c_m^2(W)\in \Z$. Now by \Cref{lemma:factor}, $2^{2m-2}$ divides $(2m!)$ if and only if $2m$ when written in the $2$-adic representation has at most 2 non-zero digits. In order to avoid a contradiction, we must have $3^k=2^s+1$, which implies that $k=1,s=1$ or $k=2,s=3$. 
    \begin{enumerate}
        \item  When $k=1$, \eqref{eqn:3^k} is
        $$3^3\cdot 2^9\cdot 31+3^4\cdot 2^5 \cdot 31 \int_X c_3^2(W)=2\cdot 6! \cdot 21.$$
         In particular, $3\int_X c_3^2(W)$ is not an integer as $31$ does not divide the right-hand side, a contradiction.
    \item  When $k=2$, \eqref{eqn:3^k} is
    $$\left(3^9\cdot 2^{21}\cdot (2^{17}-1)+3^{10}\cdot 2^{17} \cdot (2^{17}-1) \int_X c_9^2(W)\right)\cdot 43867=2\cdot (18)! \cdot 399.$$
    We get a contradiction, as $43867$ is a prime number, which does not divide the right-hand side.
    \end{enumerate}

    \item When $m=2\cdot 3^k$, we have 
    \begin{align*}
    \frac{16}{3}\cdot \frac{2^{2m}(2^{2m-1}-1)}{(2m)!}B_m + \left (-\frac{2^{2m}(2^{2m-1}-1)}{(2m)!}B_m+  \frac{2^{2m+1}(2^{m-1}-1)^2}{m!m!}B^2_{m/2}\right)&\int_X c^2_m(W) & \\
    & =\frac{2}{3^{m+1}}
    \end{align*}
    Hence by the same argument above of looking at the factor of $2$, we have a contradiction if $k>2$. 
    \begin{enumerate}
        \item  When $k=1$, we have
        \begin{align*}
            3^6\cdot 2^{15} \cdot (2^{11}-1)\cdot691\cdot 21^2+& 3^{7}\cdot 2^{11}\left(-(2^{11}-1)\cdot 691 \cdot 21^2+(2^{5}-1)^2\binom{12}{6}\cdot 1365 \right)\int_X c_6^2(W) \\
            &=2(12!)\cdot 1365\cdot 21^2.
        \end{align*}
     Note that $3\int_X c_6^2(W)\in \Z$. As $3$ divides $1365$ and $\binom{12}{6}$, we have that $3^8$ divides the left-hand side, while the right-hand side only has $3^7$ in the factorization, a contradiction. 
        \item  When $k=2$, i.e.\ $m=18$, we have
     $$\frac{16}{3}\cdot \frac{2^{36}(2^{35}-1)}{(36)!}B_{18} + \left (-\frac{2^{36}(2^{35}-1)}{(36)!}B_{18}+ \frac{2^{37}(2^{17}-1)^2}{18!18!}B^2_{9}  \right)\int_X c^2_{18}(W)=\frac{2}{3^{19}}.$$
     That is
     \begin{align*}
         3^{18}2^{39}(2^{35}-1)N_{18}D_9^2 + & 3^{19}\left (-2^{35}(2^{35}-1)N_{18}D_9^2+ 2^{35}(2^{17}-1)^2\binom{36}{18}N^2_{9}D_{18}  \right)\int_X c^2_{18}(W)\\
         & =2\cdot 36!D_{18}D_9^2.
     \end{align*}
     We have $D_9=399,D_{18}=959595,N_9= 43867, N_{18}=26315271553053477373.$ Then we can write it as 
     $$2^{35}\cdot 3^{19}\cdot A \int_X c_{18}^2(W)=B,$$
     where\footnote{The computations use \href{https://www.mathsisfun.com/calculator-precision.html}{Full Precision Calculator}  and \href{https://www.numberempire.com/numberfactorizer.php}{Number Factorizer}. }
     \begin{eqnarray*}
     A & = &-(2^{35}-1)N_{18}D_9^2+(2^{17}-1)^2\binom{36}{18}N^2_{9}D_{18}\\
     & = & 143945095833471912023375935245554409 \\
     & = & 3^3 \cdot 7^3 \cdot 19^3 \cdot 31 \cdot 31307 \cdot 2911373 \cdot 9307691 \cdot 86165861
     \end{eqnarray*}
     and 
    \begin{eqnarray*}
     B & = & 2\cdot 36!D_{18}D_9^2-3^{18}2^{39}(2^{35}-1)N_{18}D_9^2\\
     & = & -30545195871712107035689349862759906189459781054063706112 \\
     & = &  2^{35} \cdot 3^{21} \cdot 7^2 \cdot 19^2 \cdot 31 \cdot 39177751 \cdot 3955872080978053464367
     \end{eqnarray*}
     We see that $3\int_X c_{18}^2(W)$ is not an integer, contradiction.
    \end{enumerate} 
\end{enumerate}
\end{proof}

\subsubsection{$\Z/4$ actions}

\begin{theorem}\label{thm:Z4}
For any isolated terminal singularity $\C^n/(\Z/4)$ and $n$ is not a power of $2$, then the contact link is not exactly fillable.
\end{theorem}
\begin{proof}
 The $\Z/4$ action on $\C^n$ with an isolated singularity is determined by the weights of a generator. Namely an integer $0\le m \le n$, such that the generator acts by $\diag(\underbrace{\i,\ldots, \i}_{m}, \underbrace{-\i,\ldots, -\i}_{n-m})$. By possibly choosing the other generator of $\Z/4$,  we may assume $m\le  n-m$. Moreover, being terminal implies that $n\ge 3$. The total Chern class of $(S^{2n-1}/(\Z/3),\xi_{\std})$ is given by
$$(1-u^2)^{m}(1-u)^{n-2m} \mod \la 4, u^n\ra.$$
By \Cref{thm:2_conj}, we can assume $n\ge 4$ is even. Now we assume $(m,n-2m)=2^sa$ for $(2,a)=1$. The $m=0$ case is established in \Cref{thm:cyclic}. We first assume that $n-2m>0$. Note that $\md(\C^n/(\Z/4))=\frac{n}{4}+\frac{m}{2}-1$, by \Cref{cor:cuplength}, we have
$$\mathrm{cl}(W)\le \left\lfloor \frac{\frac{3n}{4}-\frac{m}{2}}{\frac{n}{4}+\frac{m}{2}} \right\rfloor \le 2$$
\begin{enumerate}
    \item If $m/2^s$ is even, then $(1-u^2)^{m}(1-u)^{n-2m}=1+u^{2^s}+\rm{h.o.t.} \mod 2$. Now since $2^s$ divides $n$, we have that $c_{2^s}(\xi_{\std})$ and $c_{2^s}^{n/2^s-1}(\xi_{\std})\ne 0, \mod 2$. Since $n/2^s=(n-2m)/2^s+2m/2^{s}\ge 1+4$, we must have  $c_{2^s}^{n/2^s-1}(W)\in \Tor H^*(W;\Z)$ as $\mathrm{cl}(W)\le 2$. Therefore $\Ima \phi_{\Tor}^{2n-2^{s+1}}=\Z/4$. On the other hand, we have $\coker \phi^{2^{s+1}} = 0$, we obtain a contradiction by \Cref{prop:restriction'}.
    \item If $m/2^s$ is odd, and $(n-2m)/2^s$ can be divided by $4$. Then $(1-u^2)^{m}(1-u)^{n-2m}=1+u^{2^{s+1}}+\rm{h.o.t.} \mod 2$. Since $2^{s+1}$ divides $n-2m$ and $2^{s}$ divides $m$, we have that $2^{s+1}$ divides $n$. We have that $c_{2^{s+1}}(\xi_{\std})$ and $c_{2^{s+2}}^{n/2^{s+1}-1}(\xi_{\std})\ne 0, \mod 2$.
    Since $(n-2m)/2^{s+1}+2m/2^{s+1}\ge 3$, we are done as before. 
    \item If both $m/2^s$ and  $(n-2m)/2^s$ are odd. Then $(1-u^2)^{m}(1-u)^{n-2m}=1+u^{2^{s}}+\rm{h.o.t.} \mod 2$. An obstruction can be derived similarly.
    \item  If both $m/2^s$ and  $(n-2m)/2^{s+1}$ are odd. Then we have 
    $$(1-u^2)^{m}(1-u)^{n-2m} = (1-u^2)^{m+\frac{n}{2}-m}=(1-u^2)^{\frac{n}{2}}\mod 2.$$
    As $n$ is not a power of $2$, then we can derive a contradiction similarly.
\end{enumerate}
When $n-2m$ is $0$, it is covered in \Cref{thm:A_m} below.
\end{proof}

\subsubsection{$\Z/m$ action of $A_{m-1}$ type}\label{SS:A_m}
A $\Z/m$ action on $\C^{2n}$ of $A_{m-1}$ type is given by sending a generator of $\Z/m$ to $\diag(e^{\frac{2\pi \i}{m}},e^{-\frac{2\pi \i }{m}},\ldots, e^{\frac{2\pi \i}{m}},e^{-\frac{2\pi \i }{m}})$. The associated quotient singularity is terminal if and only if $n\ge 2$. Moreover, we have $\age(g)=n$ for any $g\ne \Id \in \Z/m$.

\begin{theorem}\label{thm:A_m}
Assume $\Z/m$ ($m\ge 3$) acts on $\C^{2n}$ of $A_{m-1}$ type, then $(S^{4n-1}/(\Z/m),\xi_{\std})$ has no exact filling if one of the following holds.
\begin{enumerate}
    \item $m\ne p^s$ for a prime number $p$ and $n>2$.
    \item $m=p^s$, and $n\ne 2p^r,p^r$ for a prime number $p$.
\end{enumerate}
\end{theorem}
\begin{proof}
By \Cref{prop:chern}, the total Chern class of $\xi_{\std}$ is
$$(1-u^2)^{n}, \mod \la m,u^{2n} \ra.$$
When $m\ne p^s$, $m$ is divisible by two different primes $p,q$. It is easy to check that $\{p^r,2p^r|r\in \N  \}\cap \{q^r,2q^r|r\in \N  \}=\{1,2\}$. Therefore when $n>2$, either we have $n\ne p^r, 2p^r$ or we have $n\ne q^r, 2q^r$. W.L.O.G., we assume $n\ne p^r, 2p^r$.  By \Cref{lem:p_adic}, there is $s$ such that $p^s|n$ and $n/p^s>2$, such that $c_{2p^s}(\xi_{\std})\ne 0 \mod p$. As a consequence, we have $\coker \phi^{4p^s} = \Z/l$, where $l$ divides $m$ and $l$ is coprime to $p$. Then $(c_{2p^s}(\xi_{\std}))^{\frac{n-p^s}{p^s}} \ne 0 \mod p$. By \Cref{cor:intersection_SS}, we have $(c_{2p^s}(\xi_{\std}))^{\frac{n-p^s}{p^s}}\in \Tor H^{4n-4p^s}(W;\Z)$. As a consequence, we have $\Ima \phi^{4n-4p^s}_{\Tor}=\Z/k$, where $k$ divides $m$ and $m/k$ is coprime to $p$. In particular, we have $p$ divides $k$. Therefore, we get a contradiction by \Cref{prop:restriction'}. When $m=p^s$, and $n\ne 2p^r,p^r$ for a prime number $p$. A contradiction can be derived similarly. 
\end{proof}

\subsection{ADE actions}
Let $G$ be a finite subgroup of $SU(2)$ throughout this subsection, i.e.\ $G$ is of $A_m(m\ge 1),D_m(m\ge 4),E_6,E_7,E_8$ types (of Dynkin diagrams). We will focus on $D,E$ types, as the $A_m$ type is the cyclic $\Z/(m+1)$ action considered in \Cref{thm:A_m}. We are interested in the (non-)existence of exact fillings of $(S^{4n-1}/G,\xi_{\std})$, where $G$ acts on $(\C^2)^n$ through the diagonal map $SU(2)\to \oplus^n SU(2)\to SU(2n)$. 

\begin{proposition}\label{prop:ADE_terminal}
For the diagonal action of $G$ on $\C^{2n}$, we have $\age(g)=n$ for any $g\ne \Id \in G$, in particular, $\C^{2n}/G$ is terminal if and only if $n\ge 2$. 
\end{proposition}
\begin{proof}
For $g\ne \Id \in G$ acts on $\C^2$, since $g\in SU(2)$, we have $\age(g)\in \N$. Note that $\age(g)<2$ as the complex dimension is $2$ and $\age(g)>0$ if we consider $g\ne \Id$. We must have $\age(g)=1$ for $g\ne \Id$. Then for the diagonal action of $g$ on $\C^{2n}$, it is clear that $\age(g)=n$. The last claim follows from \Cref{prop:terminal}. 
\end{proof}

\begin{remark}
When $n=1$, $\C^2/G$ is canonical. Moreover, the contact link $(S^3/G,\xi_{\std})$ is exactly fillable by the plumbing of $T^*S^2$ according to the corresponding Dynkin diagram.
\end{remark}
The $D_{m+2}$ Dynkin diagram corresponds to the binary dihedral group $2D_{2m}$ of order $4m\ge 8$. The $E_6,E_7,E_8$ Dynkin diagrams correspond to the binary tetrahedral group $2T$, the binary octahedral group $2O$, and the binary icosahedral group $2I$ respectively. We have $|2T|=24,|2O|=48,|2I|=120$. We refer readers to \cite{MR1957212} for details of those groups.

\begin{theorem}\label{thm:ADE}
We have the following:
\begin{enumerate}
    \item\label{ADE:1} If $G=2D_{2m}$ for $m\ge 2$, then $(S^{4n-1}/G,\xi_{\std})$ has no exact filling when $n>2$ and $m,n$ are not powers of $2$  simultaneously.
    \item\label{ADE:2} If $G=2T$, then $(S^{4n-1}/G,\xi_{\std})$ has no exact filling when $n>2$.
    \item\label{ADE:3} If $G=2O$, then $(S^{4n-1}/G,\xi_{\std})$ has no exact filling when $n>2$.
    \item\label{ADE:4} If $G=2I$, then $(S^{4n-1}/G,\xi_{\std})$ has no exact filling when $n>2$.
\end{enumerate}
\end{theorem}
\begin{proof}
In view of \Cref{cor:intersection_SS} and the proof of \Cref{thm:A_m}, it suffices to find $i<\frac{n}{2}$ such that $i$ divides $n$ and $c_{2i}(\xi_{\std}) \ne 0 \mod p$, where $p$ is a prime factor of $|G|$. 

\textbf{Proof of \eqref{ADE:1}:} By \Cref{prop:ADE_Chern}, the total Chern class of the contact boundary is given by
$$(1-v)^n\mod \la 4m, v^n \ra.$$
If $m$ is not a power of $2$, then there is an odd prime $p$ dividing $m$. As $n$ can not be in the form of $2^s$ or $p^k$ (or $2p^k$) simultaneously,  there must be $i<\frac{n}{2}$ such that $i$ divides $n$ and $c_{2i}(S^{4n-1}/(2D_{2m}),\xi_{\std}) \ne 0 \mod p$ or $\ne 0 \mod 2$.  Now we assume $m=2^k\ge 2$. If $n$ is not a power of $2$, then there must be $i<\frac{n}{2}$ divining $n$, such that $c_{2i}(S^{4n-1}/(2D_{2m}),\xi_{\std}) \ne 0 \mod 2$ and a contradiction arises similarly.

\textbf{Proof of \eqref{ADE:2}:} If $n$ is not a power of $2$ and $i$th digit of the $2$-adic representation of $n$ is the lowest digit with non-zero value. We have $c_{2^{i}}(\xi_{\std})\ne 0 \mod 2$. If $n=2^k$ but $k\ge 2$, then we have $c_2(\xi_{\std})\ne 0 \mod 3$. 

\textbf{Proof of \eqref{ADE:3} and \eqref{ADE:4}:} The $n\ge 3$ case is similar to case \eqref{ADE:2}.
\end{proof}

The argument used in \Cref{thm:A_m,thm:ADE} can be abstractized as follows.
\begin{theorem}
Assume $G$ acts on $\C^n$ with isolated singularity for $n\ge 3$, such that $\age(g)=\frac{n}{2}$ for any $g\ne \Id \in G$. Hence $\C^n/G$ is terminal by \Cref{prop:terminal}. Then $(S^{2n-1}/G,\xi_{\std})$ has no exact filling if the following conditions are met.
\begin{enumerate}
    \item There exists $i$, which divides $n$ and $n/i\ge 2$, such that $\oplus_{j=0}^\infty H^{2ij}(BG;\Z)=H^*(B\Z/m;\Z)$ as a ring.
    \item $c_{i}(\xi_{\std}) \ne 0 \mod p$ for some prime factor of $m$. 
\end{enumerate}
\end{theorem}

\appendix
\section{Bernoulli numbers}\label{app}
Bernoulli numbers $B_1,B_2,\ldots$ can be defined as the coefficients which occur in the power series expansion
$$\frac{x}{\tanh x} = 1+\frac{B_1}{2!}(2x)^2-\frac{B_2}{4!}(2x)^4+\frac{B_3}{6!}(2x)^6-\ldots.$$
Hence they appear as coefficients in $L$-genus naturally. An alternative formula due to Worpitzky \cite{MR1579945} in 1883 for the Bernoulli numbers is given by
\begin{equation}\label{eqn:W}
    B_n=(-1)^{n-1}\sum_{r=0}^{2n}\frac{1}{r+1}\sum_{s=0}^r (-1)^s\binom{r}{s}s^{2n}
\end{equation}
One fundamental property of the Bernoulli numbers is the following theorem due to Clausen and Staudt back in 1840.
\begin{theorem}[{c.f.\ \cite[Theorem B.3]{MR0440554}}]\label{thm:CS}
$(-1)^n B_n$ is congruent modulo $1$ to $\sum \frac{1}{p}$, where the sum is over all primes $p$ such that $(p-1)|2n$. 
\end{theorem}
An instant corollary of the theorem is that $B_n$, when written in the lowest fraction, is 
$$\frac{N_n}{2D_n}, \text{ where } D_n=\prod_{\substack{ \text{odd prime } p, \\(p-1)|2n}} p.$$
The property of Bernoulli numbers we need is the following, whose proof is motivated from Worpitzky's proof of the Clausen-Staudt Theorem explained in \cite{MR130397}.

\begin{lemma}\label{lemma:numerator_Bernoulli}
When $n=2^k>8$, we have $N_n=1 \mod 64$ and $D_n=-1 \mod 64$.
\end{lemma}
\begin{proof}
$D_n$ is the product of odd primes $p$, such that $p-1$ divides $2^{k+1}$. Hence the prime $p$ is in the form $2^m+1$, which starts with $3,5,17,257,\ldots$. Then the last claim follows from  that  $3\cdot 5 \cdot 17 = -1 \mod 64$.

Then we will prove the first claim for $k>3$. Note that we have 
\begin{eqnarray}
    N_n & = & -2D_n\cdot\sum_{r=0}^{2n}\frac{1}{r+1}\sum_{s=0}^r (-1)^s\binom{r}{s}s^{2n} \nonumber \\
    &= & -2D_n\cdot\sum_{r=0}^{2n}(-1)^r\frac{r!}{r+1} \frac{1}{r!} \sum_{s=0}^r (-1)^{r-s}\binom{r}{s}s^{2n} \label{eqn:sum}
\end{eqnarray}
\begin{enumerate}
    \item The $r=1$ term of \eqref{eqn:sum} is $D_n$, which is $-1 \mod 64$. 
    \item The $r=2$ term of \eqref{eqn:sum} is $- 2\frac{D_n}{3} \left(-2+2^{2n} \right) =  4 \frac{D_n}{3}=20 \mod 64$ as $n>8$.
    \item The $r=3$ term of \eqref{eqn:sum} is 
    \begin{equation}\label{eqn:r=4}
        -D_n \frac{1}{2}(3\cdot 2^{2n}-3-3^{2n}).
    \end{equation}
    Note that $3^{2n}=(9)^n = (81)^{2^{k-1}}=(33)^{2^{k-2}}=(65)^{2^{k-3}} = 1 \mod 128$ as $k\ge 4$, then \eqref{eqn:r=4} is $-2 \mod 64$. 
    \item The $r=4$ term of  \eqref{eqn:sum} is
    $$-2\frac{D_n}{5}(-4+6\cdot 2^{2n}-4\cdot 3^{2n}+4^{2n}),$$
    which modulo $64$ is $-16$ for $k>3$.
    \item The $r=5$ term of  \eqref{eqn:sum} is
    \begin{equation}\label{eqn:r=5}
        -2\frac{D_n}{6}(-5+10\cdot 2^{2n}-10\cdot 3^{2n}+5\cdot 4^{2n}-5^{2n}).
    \end{equation}
    Since $5^{2n}=(25)^n=(49)^{2^{k-1}}=(33)^{2^{k-2}} = 1 \mod 32$ for $k\ge 3$, \eqref{eqn:r=5} $\mod 64$ is $16$ when $k>3$.
    \item The $r=6$ term of  \eqref{eqn:sum} is
    \begin{equation}
        -2\frac{D_n}{7}(-6+15\cdot 2^{2n}-20\cdot 3^{2n}+15\cdot 4^{2n}-6\cdot 5^{2n}+6^{2n}), \nonumber
    \end{equation}
   which is $0 \mod 64$.
    \item The $r=7$ term of  \eqref{eqn:sum} is
    $$-2\frac{D_n}{8}(-7+21\cdot 2^{2n}-35\cdot 3^{2n}+35\cdot 4^{2n}-21\cdot 5^{2n}+7\cdot 6^{2n}-7^{2n}),$$
    which modulo $64$ is $-16$ (which requires checking separately for $k=4$).
\end{enumerate}
Now suppose $r+1=p\ge 11$ is a prime number. Then
\begin{equation}\label{eqn:A}
    A_{2n,p-1}:=\sum_{s=0}^{p-1}(-1)^s \binom{p-1}{s}s^{2n}=\sum_{s=0}^{p-1}s^{2n} \mod p.
\end{equation}
Since $r=p-1$ is even, the $r$-th term of \eqref{eqn:sum} is
$-2D_n\frac{(p-1)!}{p}\frac{A_{2n,p-1}}{(p-1)!}$.

\begin{enumerate}
    \item If $p-1$ does not divide $2n$, then $p$ divides \eqref{eqn:A}. Since  $\frac{A_{2n,p-1}}{(p-1)!}$ is an integer by \cite[(2.2)]{MR130397} and $2(p-1)!$ is divisible by $2\cdot 10!$, we know that the $r=p-1$ term of \eqref{eqn:sum} is $0\mod 64$.
    \item If $p-1$ divides $2n$, then $\eqref{eqn:A}$ is $-1$ mod $p$. Then this $p$ appears in $D_n$. Note that $p=2^m+1$ for $m\ge 4$ and $n> 8$. Hence we have
    $$A_{2n,p-1}=\sum_{s=0}^{2^m}(-1)^s\binom{2^m}{s}s^{2n}=-\sum_{s=1}^{2^{m-1}} \binom{2^m}{2s-1} (2s-1)^{2n} \mod 32.$$
    We claim that $(2s-1)^{2n}=1 \mod 32$. Then above is $-\sum_{s=1}^{2^{m-1}} \binom{2^m}{2s-1} = -2^{2^m-1} = 0 \mod 32$. Therefore the $r=p-1$ term of \eqref{eqn:sum}, i.e.\ $-2D_n\frac{A_{2n,p-1}}{p}$, is $0\mod 64$. To prove the claim  $(2s-1)^{2n}=1 \mod 32$, we first have
    $$(2s-1)^{2n} = 1^{2n}-2s\binom{2n}{1}+4s^2\binom{2n}{2}-8s^3\binom{2n}{3}+16s^4\binom{2n}{4} \mod 32.$$
    Since $n=2^k>8$, we have $\binom{2n}{1}=0 \mod 16, \binom{2n}{2} = 0 \mod 8,\binom{2n}{3}=0 \mod 4, \binom{2n}{4} =0 \mod 2$, hence the claim follows.
\end{enumerate}
Now suppose $r+1\ge 9$ is not a prime. We write $r+1=2^mq$ for an odd number $q$.  By \cite[(2.2)]{MR130397}, we have $\frac{1}{r!} \sum_{s=0}^r (-1)^{r-s}\binom{r}{s}s^{2n}$ is an integer. And $(2^mq-1)!/2^mq$ is also an integer. Since $(2^m-1)! = 0 \mod 2^{2^m-1-m}$, then  $(2^mq-1)!/2^mq = 0 \mod 2^{2^m-1-2m}$. Therefore if $m\ge 4$, the $r+1=2^mq$ term of \eqref{eqn:sum} is $0\mod 64$. When $m=0$,  we have $(q-1)!/q = 0 \mod 64$ as $q\ge 9$. When $m=1$,  we have  $(2q-1)!/(2q) = 0 \mod 64$ as $2q-1\ge 9$. When $m=2$, we have $(4q-1)!/(4q) = 0\mod 64$ as $4q-1\ge 11$. When $m=3$, we have $(8q-1)!/(8q)=0 \mod 64$ as $8q-1\ge 23$. Sum up all the cases above, we have the first claim holds for $k\ge 4$. 
\end{proof}

\begin{remark}
From \href{https://oeis.org/A027641/b027641.txt}{this list}, we have\footnote{Note that $N_k$ in our convention is the absolute value of the $2k$th term in the list }
\begin{eqnarray*}
    N_{16} & = & 7709321041217, \\
    N_{32} & = & 106783830147866529886385444979142647942017, \\
    N_{64} & = & 2677547077425480828869544055852823947792914595925517406\\ & & 29978686063357792734863530145362663093519862048495908453718017.
\end{eqnarray*}
One can check that \Cref{lemma:numerator_Bernoulli} holds.
\end{remark}

\begin{remark}
By the same argument of \Cref{lemma:numerator_Bernoulli}, one can show that $N_{2^k} \mod 2^l$ stabilizes for any fixed $l$ and $k\gg 0$. 
\end{remark}

\begin{lemma}\label{lemma:trivial}
For $m=2^k\ge 16$, we have 
$$\frac{(2m)!}{2^{2m-1}}=11 \mod 64, \quad \binom{2m}{m}= 6 \mod 8.$$
\end{lemma}
\begin{proof}
We first compute
$$
\begin{array}{rclr}
    1\cdot 3 & = & 3 & \mod 64, \\
    1\cdot 3 \cdot 5 \cdot 7 & = & 41 & \mod 64,\\
    1\cdot 3 \cdot 5 \cdot \ldots  \cdot 13\cdot 15& = & 17 & \mod 64,\\
    1\cdot 3 \cdot 5 \cdot \ldots \cdot 29 \cdot 31 & = & 33 & \mod 64,\\
    \prod_{i=1}^{32}(2i-1) & = & 1  &\mod 64.
\end{array}
$$
$\frac{(2m)!}{2^{2m-1}}$ is $(2m)!$ with all factors of $2$ taken out. When $k=4$, $\frac{(2m)!}{2^{2m-1}}$ is the product of the first $4$ terms, which is $11 \mod 64$. When $k=5$,  $\frac{(2m)!}{2^{2m-1}}$ is the product of the first $5$ terms, which is also $11 \mod 64$. Note that for $k\ge 5$
$$\frac{(2^{k+2})!}{2^{2^{k+2}-1}}=\frac{(2^{k+1})!}{2^{2^{k+1}-1}} \cdot  (\prod_{i=1}^{32}(2i-1))^{2^{k-5}} = \ldots = \frac{(64)!}{2^{63}}=11 \mod 64.$$

Note that we have
\begin{equation}\label{eqn:2}
    \frac{1}{2}\binom{2m}{m} = \frac{\frac{(2m)!}{2^{2m-1}}}{\frac{m!}{2^{m-1}}\frac{m!}{2^{m-1}}}.
\end{equation}

The previous computation shows that for $k\ge 4$, we have
$$\frac{(2m)!}{2^{2m-1}} = \frac{m!}{2^{m-1}} = 3 \mod 4.$$
Therefore we have $\eqref{eqn:2}=3 \mod 4$ and $\binom{2m}{m}= 6 \mod 8$.
\end{proof}

\begin{lemma}\label{lemma:factor}
Let $p$ be a prime number and $m=\sum_{i=0}^Na_ip^i$, the number of factor $p$ in $m!$ is given by
$$\sum_{i=1}^N a_i\cdot \frac{p^i-1}{p-1}$$
\end{lemma}
\begin{proof}
Note that there are precisely $ \sum_{i\ge j}^Na_i p^{i-j}$ many integers no larger than $m$ which are divisible by $p^j$. It is clear that number of factor $p$ in $m!$ is given by 
$$  \sum_{j\le 1}^N \sum_{i\ge j}^Na_i p^{i-j}=\sum_{i\ge 1}^N a_i \sum_{j=0}^{i-1}{p^j}=\sum_{i=1}^N a_i\cdot \frac{p^i-1}{p-1}.$$
\end{proof}

\bibliographystyle{plain} 
\bibliography{ref}
\Addresses

\end{document}